\documentclass[12pt]{article}
\usepackage[utf8]{inputenc}
\usepackage{graphicx} % Required for inserting images
\usepackage{amsmath,amssymb,amstext,amsthm}
\usepackage[margin=1in]{geometry}
\usepackage[bbgreekl]{mathbbol}
\usepackage{tikz-cd}
\usepackage{stmaryrd}
\usepackage{url}
\usepackage{appendix}

\usepackage[english]{babel}
\newtheorem{theorem}{Theorem}[section]
\newtheorem{corollary}[theorem]{Corollary}
\newtheorem{lemma}[theorem]{Lemma}
\newtheorem{proposition}[theorem]{Proposition}

\theoremstyle{definition}
\newtheorem{definition}[theorem]{Definition}

\newtheorem{remarkx}[theorem]{Remark}
\newenvironment{remark}
  {\pushQED{\qed}\remarkx}
  {\popQED\endremarkx}
  
\newtheorem{examplex}[theorem]{Example}
\newenvironment{example}
  {\pushQED{\qed}\examplex}
  {\popQED\endexamplex}

\usepackage[backend=biber, style=alphabetic, maxnames=99, backref=true]{biblatex}
\usepackage{csquotes} %idk what this does
\newcommand{\N}{\mathbb N}
\newcommand{\C}{\mathbb C}

\newcommand{\Z}{\mathbb Z}

\newcommand{\Span}{\text{span}}
\newcommand{\mc}[1]{\mathcal{#1}}

\newcommand{\<}{\langle}
\renewcommand{\>}{\rangle}

\newcommand{\Ann}{\text{Ann}}

\renewcommand{\hat}[1]{\widehat{#1}}
\renewcommand{\subset}{\subseteq}

\newcommand{\eh}{\otimes_{eh}}
\newcommand{\h}{\otimes_{h}}

\newcommand{\Tr}{\text{Tr}}
\renewcommand{\bar}[1]{\overline{#1}}

\usepackage{mathtools}

\usepackage{float}

\tikzcdset{scale cd/.style={every label/.append style={scale=#1},
    cells={nodes={scale=#1}}}}

\definecolor{nicegreen}{rgb}{0,0.75,0}
\definecolor{niceblue}{rgb}{0,0,.8}
\usepackage[pdftex,pagebackref=false]{hyperref}
\hypersetup{
    colorlinks=true,        % false: boxed links; true: colored links
    linkcolor=niceblue,         % color of internal links
    citecolor=nicegreen,        % color of links to bibliography
    filecolor=magenta,      % color of file links
    urlcolor=cyan           % color of external links
}

\begin{document}

\title{Categorical (Co)Limits of Quantum Graphs}
\author{Jennifer Zhu}
\date{\today}

\maketitle

\begin{abstract}
%    Quantum graphs have garnered increasing interest in the past few years, having been of independent interest in quantum information theory and operator algebras (among other fields) and providing fertile ground for interactions between these fields. This paper mostly develops the theory of quantum graphs from an operator algebraist's point of view, although we make some allusions to quantum information theory and categorical quantum theory as well.    
    We begin with the characterization of quantum graphs as left ideals in $\mathcal M \otimes_{eh} \mathcal M$ (the extended Haagerup tensor product of $\mathcal M$ with itself) to avoid technicalities surrounding representation dependence of quantum graphs. These left ideals roughly correspond to a canonical complement of a quantum graph. Using these left ideals and some operator space theory, we find a new, representation-free characterization of a morphism of quantum graphs compatible with previous representation-dependent morphisms. A notion of categorical (co)limit of quantum graphs follows. We also briefly explore an alternative quantization of graphs as bimodules over $C^*$-algebras ($C^*$-graphs), mostly to emphasize the point that a morphism of $C^*$-graphs is not a morphism of $C^*$-correspondences.
\end{abstract}

\tableofcontents

\section{Introduction}

Many authors have explored the notion of quantum relations set out by Weaver in \cite{weaverqrelations} from a variety of perspectives. A quantum relation by Weaver's definition is an $\mc M'$-bimodule $\mc S \subset B(\mc H)$ (where $\mc M \subset B(\mc H)$ is a von Neumann algebra). When $\mc H$ is finite dimensional, these are precisely the quantum relations in \cite[Section 7]{musto-reutter-verdon}. The \textit{diagrammatic calculus} (known by many names) used in \cite{musto-reutter-verdon} has ben a fruitful avenue of research. Calculations using these diagrams can be elegant and intuitive, but they are fundamentally limited by finite dimensionality (see Chapters 2 and 3 in \cite{heunen-vicary} for an enjoyable explanation and Appendix A of the author's thesis for a brutally short summary). Thus Weaver's quantum relations provide a natural setting for infinite dimensional extensions and versions of these diagrammatically defined objects. The motivation for this paper was to develop a notion of limit for quantum graphs to build a (possibly) infinite quantum graph from finite quantum graphs. At first blush, quantum relations seem representation dependent but Weaver immediately shows that they only depend on the von Neumann algebra $\mc M$, not the ambient space $B(\mc H)$ (\cite[Theorem 2.7]{weaverqrelations}). Generally, however, one works in a particular representation and then appeals to the fact that the relevant properties of quantum graphs are representation independent.

If we were to use this characterization of quantum graphs in to define a categorical limit, however, it is unclear how a morphism of quantum graphs $\mc S_1 \subset B(\mc H_1)$ and $\mc S_2 \subset B(\mc H_2)$ should behave on the ambient spaces $B(\mc H_1)$ and $B(\mc H_2)$. Furthermore, we may encounter issues regarding the uniqueness of the categorical limit. We are thus pushed to avoid the usual characterization of a quantum graph as an $\mc M'$-bimodule $\mc S \subset B(\mc H)$, leading to a new notion of morphism of quantum graphs in Definition \ref{defn: qgraph morph}. This new notion is equivalent to a CP morphism of quantum graphs in \cite[Definition 7.1]{daws} and a classical morphism of quantum graphs in \cite[Definition 5.4]{musto-reutter-verdon} under the appropriate assumptions. It also has the advantage of giving a clean visualization of categorical limits: compare the diagrams in Section \ref{sec: limit of q graphs} where we use this new notion of morphism and the ones in Theorem \ref{thm: c*graphlimit} where we define the morphism on the level of bimodules.

Although this paper is not focused on the quantum information background or implications, we will see some comments sprinkled through the paper referring to the quantum information theoretic roots. We will introduce the connection between quantum channels and quantum graphs here, but the reader may wish to consult the original paper \cite{dsw} for further explanation. The authors of \cite{dsw} are inspired by Shannon's analysis of confusability graphs of noisy classical channels in \cite{shannon}. We can consider such a channel as a probabilistic function $N: A \to B$ between finite sets $A$ and $B$ where we allow for the possibility that $N$ does not always send an input $a$ to a specified output $N(a)$ (this is the introduction of ``noise'' to the channel). We assume we know the probability $N(b|a)$ of $N$ sending $a$ to $b$ for every $a \in A$ and $b \in B$. Such a channel $N$ gives rise to a \textit{confusability graph}.

\begin{definition}
    The \textit{confusability graph} of a noisy classical channel $N: A \to B$ is the graph $G_N = (A, E)$, where
    \[
    (a_1, a_2) \in E \iff \exists b \in B \text{ s.t. } N(b|a_1 )N(b|a_2) \neq 0.
    \]
\end{definition}

The name is derived from the fact that an edge exists between two inputs if and only if they might be sent to the same output (i.e., confused) after being sent through the channel. With an eye towards quantization, we write this classical channel as a quantum channel. Fix Hilbert spaces $\C^{|A|}$ and $\C^{|B|}$ with (orthonormal) bases $\{|a\>\}_{a \in A}$ and $\{|b\>\}_{b \in B}$ respectively. Identify each element $a \in A$ with the quantum state $|a\>\<a| \in B(\C^{|A|})$, and similarly for each $b \in B$. Define the quantum channel $Q_N : B(\C^{|A|}) \to B(\C^{|B|})$ associated to $N$ by
\begin{align*}
    Q_N (T) := \sum_{a \in A, b \in B}\underbrace{\sqrt{N(b|a)} |b\>\<a|}_{K_{ab}} \, T \,  \underbrace{\sqrt{N(b|a)} |a\>\<b|.}_{K_{ab}^\dagger}
\end{align*}
Note that for $a_1, a_2 \in A$ and $b_1,b_2 \in B$ the product of $K_{a_1b_1}^\dagger K_{a_2b_2}$ is
\[
K_{a_1b_1}^\dagger K_{a_2b_2} = \delta_{b_1b_2} \sqrt{N(b_1|a_1) N(b_2|a_2)} \, |a_1\>\<a_2|
\]
so the confusability graph of $N$ can be recovered from the vector space 
\[
\mc S := \Span\{ K_{a_1b_1}^\dagger K_{a_2b_2} : a_1,a_2 \in A, \, b_1, b_2 \in B\} \subset B(\C^{|A|}).
\]
because $|a_1\>\< a_2| \in \mc S$ if and only if there exists some $b$ such that $N(b|a_1)N(b|a_2) \neq 0$. The authors of \cite{dsw} then quantize the above notions by taking a quantum channel (a CPTP map, in their framework) written in a Kraus form
\begin{align*}
    Q : B(\mc H_1) &\to B(\mc H_2) \\
    T &\mapsto \sum_{i \in I} K_i T K_i^\dagger
\end{align*}
and defining its \textit{quantum confusability graph} \cite[Equation (2)]{dsw} by
\[
\mc S := \Span \{K_i^\dagger K_j : i,j \in I\} \subset B(\mc H_1).
\]
These quantum confusability graphs are (the canonical) examples of quantum graphs to which we refer in this paper.

The sections are laid out as follows. Section \ref{sec: limits-prelim} is an overview of the conventions and necessary operator space and category theory for the remainder of the paper. We will use the framework of quantum relations in \cite[Definition 2.4]{weaverqrelations}, but we will depart from Weaver's naming convention and call them quantum graphs. A quantum graph according to \cite[Definition 2.4]{weaverqrelations} will be specified as a reflexive, symmetric quantum graph in this paper. Section \ref{sec: morphisms} presents the intuition behind our new morphism and defines it. Section \ref{sec: limits-qgraphs} shows how some properties of quantum graphs (considered as operator space bimodules over $\mc M'$ for some von Neumann algebra $\mc M$) are reflected in their annihilators (which are left ideals in $\mc M \eh \mc M$). Section \ref{sec: limit of q graphs} finally takes the categorical limit of a class of quantum graphs. See Remark \ref{rmk: limit of q graphs} for the categorical colimit. In Section \ref{sec: opc*sp} we take heavy inspiration from \cite{mawtod} to explore the consequences of defining graphs as $C^*$-algebra bimodules and characterize the morphism given in Section \ref{sec: morphisms} on the level of operator spaces $\mc S$. Ultimately, however, it seems that the notion of morphism in Section \ref{sec: morphisms} is more natural in the context of (co)limits. Section \ref{sec: remarks} contains some reflections on the previous sections, and Section \ref{sec: questions} collects some further questions and future avenues of research. In Appendix \ref{appendix: string} the reader will find the graphical calculus conventions and definitions needed to prove that a ``classical morphism of quantum graphs'' defined in \cite[Definition 5.4]{musto-reutter-verdon} is equivalent to our new definition of morphism under the more restrictive set of assumptions.

\subparagraph{Acknowledgements.} This paper was written as part of the author's PhD thesis. The author thanks her advisor Michael Brannan for a herculean effort of patience and encouragement over the years and her thesis committee for providing feedback and corrections.

\section{Preliminaries} \label{sec: limits-prelim}

By limits in this paper we refer to categorical limits (projective limits) and colimits (inductive limits); see \cite[Chapter III.4]{cats-working}, \cite[Chapter III.3]{cats-working}. Such limits are defined in purely categorical terms (i.e., by objects and morphisms), so we need only identify the correct notion of quantum relations and their morphisms to apply the definition of categorical (co)limit. 

Our objects will of course be quantum relations as defined in Definition \ref{defn: qgraph}. As mentioned in the introduction, we refer to Weaver's quantum relations as quantum graphs. Weaver's quantum graphs will be specified as symmetric, reflexive quantum graphs.

A quantum graph in \cite{weaverqrelations} is given by a von Neumann algebra $\mc M$ (roughly, the quantization of the vertices) and an $\mc M'$-bimodule $\mc S \subset B(\mc H)$ , where $B(\mc H)$ houses a representation of $\mc M$ (roughly, the quantization of the edges). However, our morphism uses alternate characterization of quantum graphs -- keeping the same von Neumann algebra $\mc M$ as our quantization of vertices, we will quantize the complementary edges of the graph instead. These complementary edges live in $\mc M \eh \mc M$, the extended Haagerup tensor product of $\mc M$ with itself. We will introduce the extended Haagerup tensor product $\otimes_{eh}$ in Section \ref{subsec: opsp} and state more precisely which subspaces correspond to quantum graphs. 

The symbol $\otimes$ will denote the algebraic tensor product. We will be using both $^*$ and $^\dagger$ in this paper. The former will denote the involution in a $*$-algebra and the latter will denote the adjoint of a bounded linear map between Hilbert spaces. At times these notions will coincide, but the context will hopefully be clear enough to avoid confusion. When a predual of a space $X$ exists and is unique, we will denote it by $X_*$. We will denote the topology on $X^*$ induced by a space $X$ by $\sigma(X)$. The unit of an algebra $\mc A$ will be denoted $1_{\mc A}$.

\subsection{Operator Algebras} \label{prelim: op alg}

We will assume the reader is familiar with operator algebras, but we set down some notation and naming conventions here. Throughout this paper $\mathcal H$ will denote a Hilbert space and $B(\mathcal H)$ the bounded operators on $\mc H$. A vector $|\xi\rangle \in \mc H$ will be denoted as so, and our inner products $\langle \cdot |\cdot \rangle$ will be $\C$-linear on the right and anti-linear on the left. Our trace $\Tr_{\mc H}: B(\mc H) \to \C \cup \{\infty\}$ will be the following:
\[
\Tr_{\mathcal H} (T) = \sum_{i \in I}\<\xi_i | T|\xi_i\>
\]
for any orthonormal basis $\{|\xi_i\>\}_{i \in I}$. We denote the identity operator on a Hilbert space $\mc H$ by $I_{\mc H}$.  The \textit{trace class operators} in $B(\mc H)$ are operators $T$ such that 
\[
\Tr_{\mc H} (\sqrt{T^*T}) < \infty.
\]
We will at times identify the trace class operators in $B(\mc H)$ with the predual $B(\mc H)_*$ via 
\begin{align*}
    \{\text{trace class operators}\} &\to B(\mc H)_* \\
    T &\mapsto \Tr_{\mc H}(T \cdot).
\end{align*}
The compact operators in $B(\mc H)$ will be denoted $K(\mc H)$. The notation $B(\ell^2)$ refers to the bounded operators on the Hilbert space $\ell^2(I)$ for some index $I$. For nearly every application in this paper, we may assume $I=\N$.

A von Neumann algebra $\mc M$ will be a $C^*$-algebra with a (unique) Banach space predual $\mc M_*$. The spatial tensor product of von Neumann algebras $\mc M$ and $\mc N$ will be a von Neumann algebra denoted $\mc M \overline{\otimes} \mc N$. We will define it as the dual of $\mc M_* \hat \otimes \mc N_*$ (see Definition \ref{defn: osptp} for $\hat \otimes$), but if one were to represent $\mc M \subset B(\mc H)$ and $\mc N \subset B(\mc K)$ then $\mc M \bar \otimes \mc N = (\mc M \otimes \mc N)'' \subset B(\mc H \otimes_2 \mc K)$ (see \cite[Definition 5.1]{takesaki1}). 

\subsection{Quantum Graphs}

For reference we state the quantum graph definitions and conventions we will be using with the modifications mentioned above.

\begin{definition} \protect{\cite[Definition 2.1]{weaverqrelations}} \label{defn: qgraph}
A \textit{quantum graph} on a von Neumann algebra $\mc M \subset B(\mc H)$ is a weak$^*$ closed subspace $\mc S \subseteq B(\mc H)$ such that $\mc M' \mc S \mc M' \subset \mc S$.
\end{definition}

\begin{definition} \protect{\cite[Definition 2.4(d)]{weaverqrelations}} \label{defn: qgraph properties}
    Let $\mc M \subset B(\mc H)$ be a von Neumann algebra. A quantum graph $\mc S$ on $\mc M$ is
    \begin{enumerate}
        \item \textit{reflexive} if $\mc M' \subset \mc S$
        \item \textit{symmetric} if $\mc S^* = \mc S$
        \item \textit{transitive} if $\mc S^2 \subset \mc S$.
    \end{enumerate}
\end{definition}

\begin{remark}
We immediately have a reflexive, symmetric, transitive quantum graph is a von Neumann algebra containing $\mc M'$. As Weaver remarks in \cite[Definition 2.6]{weaverqrelations}, there is a bijection between von  Neumann algebras containing $\mc M'$ and von Neumann algebras contained in $\mc M$. Hence the choice to make quantum graphs $\mc M'$-bimodules instead of $\mc M$-bimodules is arbitrary from this point of view, although there is some motivation for $\mc M'$-bimodules from the Knill-Laflamme model of quantum error correction. (See \cite[``Subsystems'']{klv} or Chapter 1.3 of the author's thesis.)
\end{remark}

\begin{example}\protect{\cite[Proposition 2.2]{weaverqrelations}} \label{ex: classical-to-quantum-graph}
    A classical graph $G = (V,E)$ on any number of vertices where $E \subset V \times V$ is a quantum graph $(\mc M, \mc S)$ in the following way: Take the von Neumann algebra to be $\mc M := \ell^\infty(V) \subset B(\ell^2(V))$, where we consider a function $f \in \ell^\infty(V)$ to be a bounded operator on $\ell^2(V)$ by pointwise multiplication. Take the operator space $\mc S$ to be
    \[
    \mc S := \overline{\Span}^{w^*} \{|x\>\<y|: (x,y) \in E\} \subseteq B(\ell^2(V)).
    \]
    This coincides with the notion of graph operator system used in \cite{ortpaul}, for example, for graphs on finitely many vertices.
\end{example}

\subsection{Operator Space Theory} \label{subsec: opsp}
As mentioned in the introduction, we will need some operator space theory in order to quantize the complementary edges of a quantum graph. The full explanation of this quantization is given in Subsection \ref{subsec: classical graph morphisms}, but at the end of this subsection we will see how $\eh$ recovers the complementary edges of a finite classical graph. We will also make use of the topologies laid down here to ensure we have the correct continuity of various operations in Section \ref{sec: limit of q graphs}.

Basic definitions from operator space theory are given below, but the following is by no means an in-depth exploration or explanation. A few standard references for operator spaces are \cite{effros-ruan-book}, \cite{intro-op-sp-pisier}, and \cite{paulsen-cb}. 

\begin{definition}
    A \textit{concrete operator space} is a norm-closed subspace $X \subset B(\mc H)$ for some Hilbert space $\mc H$.
\end{definition}

\begin{definition}
    An \textit{abstract operator space} is a vector space $X$ with a sequence of norms $\|\cdot \|_{M_n(X)} : M_n \otimes X \to [0,\infty)$ such that 
    \begin{enumerate}
        \item [(OS1)] the norm $\|\cdot\|_{M_1(X)} : M_1 \otimes X \to [0,\infty)$ endows $M_1 \otimes X \cong X$ with a Banach space structure,
        \item [(OS2)] for $x \in M_m \otimes X$ and $y \in M_n \otimes X$,
        \[
        \|x \oplus y\|_{M_{m+n}(X)} = \max\{\|x\|_{M_m(X)},\|y\|_{M_n(Y)}\},
        \]
        \item [(OS3)] and for $\alpha, \beta \in M_m$ and $x \in M_m \otimes X$
        \[
        \|\alpha x \beta\|_{M_m(X)} \leq \|\alpha\| \|x\|_{M_m(X)} \|\beta\|.
        \]
    \end{enumerate}
    We may occasionally abbreviate $\|\cdot \|_{M_m(X)}$ by $\|\cdot \|_m$ when the context is clear. This sequence of norms is the operator space structure (OSS) on the operator space $X$.
\end{definition}

Concrete and abstract operator spaces are ``the same.'' To state this precisely we need some terminology.

\begin{definition}
    Given a linear map $\theta: X \to Y$ between two abstract operator spaces, we can define an \textit{amplification} of $\theta$:
    \[
    I_{M_n} \otimes \theta : M_n \otimes X \to M_n \otimes Y.
    \]
    We say $\theta$ is \textit{completely bounded} if
    \[
    \|\theta\|_{cb} := \sup_n \{\|Id_n \otimes \theta\|\} < \infty,
    \]
    where $\|Id_n \otimes \theta\|$ is the operator norm of the amplification between the normed spaces $(M_n \otimes X, \|\cdot\|_n)$ and $(M_m \otimes Y, \|\cdot\|_m)$. We say $\theta$ is \textit{completely contractive} if $\|\theta\|_{cb} \leq 1$ and \textit{completely isometric} if for $I_{M_n} \otimes \theta$ is an isometry for each $n$.
    Finally $\theta$ is a \textit{completely isometric isomorphism} if $\theta$ is a surjective complete isometry.
\end{definition}

\begin{remark}
Every $C^*$-algebra $\mc A$ admits a unique norm such that 
\[
\|a^*a\| = \|a\|^2 \qquad \forall a \in \mc A.
\]
Fundamentally, this is because the norm is induced by an algebraic property of the elements. Therefore $M_n \otimes \mc A$ has a unique norm $\|\cdot \|_n$ that ensures it is a $C^*$-algebra, and we always have a canonical OSS on a $C^*$-algebra. 
\end{remark}

Every concrete operator space therefore is an abstract operator space. The reverse implication comes from Ruan's thesis.

\begin{theorem} \cite[Theorem 3.1]{ruan}
    Every abstract operator space is completely isometrically isomorphic to a concrete operator space.
\end{theorem}

It thus generally suffices to assume an operator space is embedded in some $B(\mc H)$.

\subsubsection{Operator Space Tensor Norms}

The compatibility of the morphisms in the limit of quantum graphs relies heavily on a particular operator space tensor product, the extended Haagerup tensor product $\otimes_{eh}$ (and to some extent the normal Haagerup tensor product $\otimes_{\sigma h}$). Rather than introducing $\otimes_{eh}$ in isolation, it may be more elucidating to present all three Haagerup tensor products (and the operator space projective tensor product) together in relation to the map $\Phi$ introduced below.

Given two operator spaces $X$ and $Y$, there are many OSSs one can impose on $X \otimes Y$ under which the completion is an operator space (see \cite{opsptennorm} for a slew of them). Two examples relevant to this paper are the operator space projective tensor norm $\hat \otimes$ and Haagerup tensor norm $\otimes_h$. These are operator spaces obtained by defining a norm on every matrix amplification $M_n(X \otimes Y)$ of the algebraic tensor product and completing the space under the respective norms. We will also be using $\eh$ and $\otimes_{\sigma h}$ which are perhaps better understood as algebraic tensor products completed under a predual topology. In particular, they are not constructed in the usual norm-completion procedure, and so are not technically operator space tensor products according to some authors. We will go into more detail below. 

\paragraph{\underline{$\widehat \otimes$: Operator Space Projective Tensor Product}}

\begin{definition} \cite[Chapter 4]{intro-op-sp-pisier} \label{defn: osptp}
    Let $X,Y$ be operator spaces. Every element $z \in M_n \otimes X \otimes Y$, the $rs$th entry $z_{r,s}$ of $z$ is an element of $X \otimes Y$ that can be written like so:
    \[
    z_{r,s} = \sum_{\substack{1 \leq i,j \leq \ell \\ 1 \leq p,q \leq m}} \alpha_{r,ip}\,  x_{ij} \otimes y_{pq} \, \beta_{jq,s} =: \alpha \cdot (x \otimes y) \cdot \beta
    \]
    where $x \in M_\ell \otimes X$, $y \in M_m \otimes Y$, and $[\alpha_{r,ip}]_{i,p} \in M_{\ell, m}$ and $[\beta_{jq,s}]_{q,j} \in M_{m, \ell}$ are rectangular matrices. Implicitly we are considering $x \otimes y \in M_{\ell m} \otimes X \otimes Y$. The \textit{operator space projective tensor product norm} of $z \in M_n \otimes X  \otimes Y$ is
    \[
    \|z\| = \inf \{\|\alpha\|_{M_{n,\ell m}} \|x\|_{M_\ell(E)} \|y\|_{M_m(F)} \|\beta\|_{M_{\ell m,n}}: z = \alpha \cdot (x \otimes y)\cdot \beta\}.
    \]
    The projective tensor product $X \hat \otimes Y$ is the completion of $X \otimes Y$ under the norm above.
\end{definition}

\begin{example}
    We give an example of what this $\alpha \cdot (x \otimes y) \cdot \beta$ might look like. We will take $\ell = 2$ and $m = 2$ so that $x$ and $y$ are
    \begin{align*}
        x &= \begin{pmatrix}
            x_{11} & x_{12} \\ x_{21} & x_{22}
        \end{pmatrix} \in M_2 \otimes X
        & y &= \begin{pmatrix}
            y_{11} & y_{12}  \\ y_{21} & y_{22} 
        \end{pmatrix} \in M_2 \otimes Y.
    \end{align*}
    We will take $n =2$ so that $\alpha \in M_{2, 4}$ and $\beta \in M_{4,2}$. We then have 
    \begin{align*}
        z &= \alpha \cdot (x \otimes y) \cdot \beta \\
        &= \begin{pmatrix}
            \alpha_{11} & \alpha_{12} &\alpha_{13} &\alpha_{14}  \\
            \alpha_{21} & \alpha_{22} &\alpha_{23} &\alpha_{24} 
        \end{pmatrix}
        \underbrace{\begin{pmatrix}
            x_{11} \otimes \begin{pmatrix}
            y_{11} & y_{12}  \\ y_{21} & y_{22} 
        \end{pmatrix} & x_{12} \otimes \begin{pmatrix}
            y_{11} & y_{12}  \\ y_{21} & y_{22} 
        \end{pmatrix} \\
        x_{21} \otimes \begin{pmatrix}
            y_{11} & y_{12}  \\ y_{21} & y_{22} 
        \end{pmatrix} & x_{22} \otimes \begin{pmatrix}
            y_{11} & y_{12}  \\ y_{21} & y_{22} 
        \end{pmatrix}
        \end{pmatrix}}_{\in M_{4} \otimes X \otimes Y}
        \begin{pmatrix}
            \beta_{11} & \beta_{12} \\
            \beta_{21} & \beta_{22} \\
            \beta_{31} & \beta_{32} \\
            \beta_{41} & \beta_{42} \\
        \end{pmatrix}.
    \end{align*}
\end{example}

\paragraph{\underline{$\h, \eh, \otimes _{\sigma h}$: Haagerup Tensor Products}} 

Rather than introducing the Haagerup tensor products from definitions, we will first exhibit a map that induces a completely isometric isomorphism from each of these tensor products into more familiar spaces. (Definitions will be given below for interested readers.) We will only claim that if $X, Y \subset B(\mc H)$ are operator spaces then $X \otimes_h Y$, $X \eh Y$, and $X \otimes_{\sigma h} Y$ are operator spaces which contain isometric copies of the algebraic tensor product $X \otimes Y$ and that for every $x \in M_n \otimes X$ and $y \in M_m \otimes Y$, the norm on each Haagerup tensor product $\alpha \in \{h,eh,\sigma h\}$ is the same:
\begin{align}
\|x \otimes y\|_{M_{mn}(X \otimes_\alpha Y)} = \|x \|_{M_n(X)} \|y\|_{M_m(Y)}.
\end{align}
The details of these spaces will follow after the presentation of these completely isometric isomorphisms.

\begin{theorem} \label{defn: Phi-map}
    The following map on the algebraic tensor product
    \begin{align*}
    \Phi : B(\mc H) \otimes B(\mc H) &\to B(\mc H) \\
    x \otimes y &\mapsto xTy \qquad \forall T \in B(\mc H)
    \end{align*}
    extends to a completely isometric isomorphism between the following spaces.
    \begin{itemize}
        \item $\Phi: B(\mc H) \otimes_h B(\mc H) \to CB (K(\mc H))$ (The Haagerup tensor product \cite[Theorem 4.3]{haagerup-smith}) 
        \item $\Phi: B(\mc H) \otimes_{eh} B(\mc H) \to CB^\sigma(B(\mc H))$ (The extended Haagerup tensor product \cite[Proposition 2.1]{blecher-smith-w*h}) 
        \item $\Phi: B(\mc H) \otimes_{\sigma h} B(\mc H) \to CB (B(\mc H))$ (The normal Haagerup tensor product \cite[Theorem 2.5]{normal-haagerup}) 
    \end{itemize}
    The notation $CB^\sigma(B(\mc H))$ indicates the subset of the completely bounded maps on $B(\mc H)$ which are additionally normal (that is, continuous with respect to the weak-* topology on $B(\mc H)$).
\end{theorem}

\begin{remark}
    A historical note: the $\otimes_{eh}$ tensor product was initially introduced in \cite{blecher-smith-w*h} as the weak$^*$ Haagerup tensor product $\otimes_{w^*h}$ between dual operator spaces. Effros and Ruan later generalized $\otimes_{w^*h}$ to $\otimes_{eh}$ between any operator spaces in \cite{opsp-hopf}. In this paper we will only use the notation $\otimes_{eh}$ but may refer to results proven for $\otimes_{w^*h}$.
\end{remark}

We will be particularly interested in $\mc M \otimes_{eh} \mc M$ where $\mc M$ is a von Neumann algebra. The key feature of $\mc M \eh \mc M$ is that it immediately preserves the bimodular structure of quantum graphs (see Definition \ref{defn: qgraph}). Theorem \ref{thm: m-haagerups} is also a collection of results from various sources, but a short proof of the second complete isometry can be found in \cite[Theorem 2.2]{gp-obs}.

\begin{definition} \cite[3.1.1]{blecher-lemerdy} \label{defn: op-modules}
    Let $X$ be an operator space and $A$ be an algebra. We call $X$ a \textit{left operator $A$-module} if there is a complete isometry $\pi_X: X \to B(\mc H_1, \mc H_2)$ for some Hilbert spaces $\mc H_1$ and $\mc H_2$ and a completely contractive homomorphism $\pi_A: A \to B(\mc H_2)$ such that $\pi_X(ax) = \pi_A(a) \pi_X(x)$. Right operator $A$-modules are similarly defined but with a map $\rho_A: A \to B(\mc H_1)$. An operator $A$-bimodule is an operator space that is both a left operator $A$-module and a right operator $A$-module.
\end{definition}

\begin{definition} \label{defn: bimod map}
    If $X,Y$ are both operator $A$-bimodules, an \textit{$A$-bimodule map} (i.e., a homomorphism of $A$-bimodules) is a completely bounded map $\pi: X \to Y$ such that for all $a,b \in A$ and $x \in X$ we have
    \[
    \pi(axb) = a \pi(x) b.
    \]
\end{definition}

\begin{theorem} \label{thm: m-haagerups}
    Let $\mc M \subset B(\mc H)$ be a von Neumann algebra. The map
    \begin{align*}
    \Phi : \mc M \otimes \mc M &\to B(\mc H) \\
    x \otimes y &\mapsto xTy \qquad \forall T \in B(\mc H)
    \end{align*}
    extends to a complete isometry between the following spaces.
    \begin{itemize}
        \item $\Phi: \mc M \otimes_h \mc M \to CB_{\mc M' \mc M'} (K(\mc H))$ (The Haagerup tensor product) 
        \item $\Phi: \mc M \otimes_{eh} \mc M  \to CB^\sigma_{\mc M' \mc M'}(B(\mc H))$ (The extended Haagerup tensor product)  
        \item $\Phi: \mc M \otimes_{\sigma h} \mc M \to CB_{\mc M' \mc M'}(B(\mc H))$ (The normal Haagerup tensor product) 
    \end{itemize}
    The notation $CB_{\mc M',\mc M'}(B(\mc H))$ denotes completely bounded $\mc M'$-bimodule maps on $B(\mc H)$, and $CB_{\mc M',\mc M'}^\sigma(B(\mc H))$ denotes the subset of such maps which are additionally normal.
\end{theorem}

\begin{remark} \label{rmk: predual-tops}
There is a subtlety around the topologies on $\mc M \otimes_{eh} \mc M$ that will play a role in Section \ref{sec: limit of q graphs}. The whole of this remark can be summarized by the following inclusion and identifications:
\begin{align*}
    \mc M &\eh \mc M &  &\underset{(3)}{\subset} & \mc M &\otimes_{\sigma h} \mc M \\
    &\rotatebox{90}{$\,=$} \text{\tiny(1)} & &&&\rotatebox{90}{$\,=$}   \text{\tiny(2)} \\
    (\mc M_* &\otimes_h \mc M_*)^* & & & (\mc M_* &\otimes_{eh} \mc M_*)^*
\end{align*}
We therefore have two choices of predual topologies on $\mc M \eh \mc M$ -- one from $\sigma(\mc M_* \otimes_h \mc M_*)$ and one from $\sigma(\mc M_* \otimes_{eh} \mc M_*)$. References for the equalities and inclusion   are below.

\begin{enumerate}
    \item [(1)] From \cite[Section 3]{blecher-smith-w*h}, since $\mc M$ is a dual operator space we have
    \[
    (\mc M_* \otimes_h \mc M_*)^* = \mc M \otimes_{eh} \mc M.
    \]
    We will denote this predual topology by $\sigma(\mc M_* \otimes_h \mc M_*)$. 
    \item [(2)] From the discussion around \cite[Lemma 5.8]{opsp-hopf} and \cite{normal-haagerup},
    \[
    (\mc M_* \eh \mc M_*)^* = \mc M \otimes_{\sigma h} \mc M.
    \]
    \item [(3)] As noted in \cite[Equation 5.15]{opsp-hopf}, the identity map
    \[
    \mc M_* \otimes_h \mc M_*  \to \mc M_* \eh \mc M_* 
    \]
    is a completely isometric injection. The dual of this map yields a projection
    \[
    \mc M  \otimes_{\sigma h} \mc M  \to \mc M  \eh \mc M 
    \]
    so $\mc M \eh \mc M$ is a (complemented) subspace of $\mc M \otimes_{\sigma h} \mc M$:
    \[
    (\mc M_* \eh \mc M_*)^* = \mc M \otimes_{\sigma h} \mc M \supseteq \mc M \otimes_{eh} \mc M.
    \]
\end{enumerate}

Since $\mc M_* \otimes_h \mc M_* \subset \mc M_* \otimes_{eh} \mc M_*$, the predual topology $\sigma(\mc M_* \otimes_{eh} \mc M_*)$ is finer than $\sigma(\mc M_* \otimes_h \mc M_*)$. 
\end{remark}

\paragraph{\underline{$\h$: Haagerup Tensor Product}}

We now give definitions for the Haagerup tensor products in more detail. Elements in the $\otimes_h$ and $\eh$ tensor products can be written as (possibly infinite) sums. The multiplicative product $\odot$ defined below is not necessary to define these sums, but we will use its continuous extension later so we take the time to introduce it here. The definitions are taken from \cite{opsp-hopf} restricted to the bilinear (rather than multilinear) case. Following \cite[Section 3]{opsp-hopf}, for indices $I,J$ we denote by $M_{I,J}(X)$ the space of $I\times J$ (possibly) infinite matrices with entries in $X$ whose finite submatrices are uniformly bounded in norm.

\begin{definition} \protect{\cite[Section 5]{opsp-hopf}} \label{defn: mult-prod-fin}
    Let $X_1, X_2$ be operator spaces and $n_1, n_2, n_3 \in \Z^+$. Given matrices
    \[
    x_1 := [x^{(1)}_{ij}]_{\substack{{i = 1,...,n_1 }\\ {j = 1,...,n_2}}} \in M_{n_1, n_2}(X_1) \qquad x_2 := [x^{(2)}_{k\ell}]_{\substack{{k = 1,...,n_2 }\\ {\ell = 1,...,n_3}}} \in M_{n_2, n_3}(X_2)
    \]
    their \textit{multiplicative product} $x_1 \odot x_2 \in M_{n_1, n_3}(X_1 \otimes X_2)$ is
    \[
    x_1 \odot_{n_2} x_2= [v_1 \odot v_2]_{i\ell} := \bigg[\sum_{k = 1,...,n_2} x_{ik}^{(1)} \otimes x_{k\ell}^{(2)}\bigg]_{i\ell}
    \]
    where the notation $x_1 \odot_{n_2} x_2$ is used to emphasize the index.
\end{definition}

\begin{example}
An example of the multiplicative product in action:
    \begin{align*}
        \begin{pmatrix}
            x_1 & x_2 \\
            x_3 & x_4 
        \end{pmatrix} \odot 
        \begin{pmatrix}
            y_1 \\
            y_2 \\
        \end{pmatrix} = 
        \begin{pmatrix}
            x_1 \otimes y_1 + x_2 \otimes y_2 \\
            x_3 \otimes y_1 + x_4 \otimes y_2
        \end{pmatrix}.
    \end{align*}
\end{example}

For a full exploration of $\h$, see \cite[Chapter 5]{intro-op-sp-pisier}, for instance.

\begin{definition} \cite[Equation 5.4]{opsp-hopf} \label{defn: h}
    Let $X_1, X_2$ be operator spaces. For $x \in M_n \otimes (X_1 \otimes X_2)$, define its norm by
    \begin{align*}
        \|x\|_h := \inf \{\|x_1\| \|x_2\| : x = x_1 \odot_m x_2, x_1 \in M_{n,m} \otimes X_1 , \,  x_2 \in M_{m,n} \otimes X_2\}
    \end{align*}
    where the infimum runs over all possible representations of $x = x_1 \odot_n x_2$ for all $m \in \N$. The completion of $X_1 \otimes X_2$ with respect to this OSS is the \textit{Haagerup tensor product} $X_1 \h X_2$.
\end{definition}

\paragraph{\underline{$\eh$: Extended Haagerup Tensor Product}}

The multiplicative product can be defined for infinite indices as well by taking limits of multiplicative products above.

\begin{definition} \protect{\cite[Section 3]{opsp-hopf}} \label{defn: mult-prod-inf}
Let $I,J,K$ be arbitrary index sets and $y_1 \in M_{I,J}(X_1)$ and $y_2 \in M_{J,K}(X_2)$. For finite subsets $F_1 \Subset I$, $F_2 \Subset J$, $F_3 \Subset K$ let $P(F_i)$ be the projections onto $\ell^2(F_i)$. The truncations of $y_1, y_2$ are
    \[
    P(F_1)y_1P(F_2) := [y_{ij}^{(1)}]_{\substack{i \in F_1 \\ j \in F_2}} \qquad P(F_2)y_2P(F_3) := [y_{k\ell}^{(1)}]_{\substack{k \in F_2 \\ \ell \in F_3}}
    \]
    and their multiplicative product is
    \[
    F:= F_1 \times F_2 \times F_3, \qquad y_1 \odot_F y_2 := P(F_1)y_1P(F_2) \odot_{F_2} P(F_2)y_2P(F_3).
    \]
    The multiplicative product of the infinite matrices is the well-defined SOT-limit
    \[
    y_1 \odot_I y_2 = y_1 \odot y_2 := \lim_{F \Subset I \times J \times K} y_1 \odot_F y_2
    \]
    where we are implicitly using the representation of operator spaces on Hilbert spaces.
\end{definition}

As the name suggests, the extended Haagerup tensor product is much like $\h$.

\begin{definition} \cite[Equation 5.7]{opsp-hopf} \label{defn: eh}
    Let $X_1, X_2$ be operator spaces. For $x \in M_n \otimes (X_1 \otimes X_2)$
    \begin{align*}
        \|x\|_h := \inf \{\|x_1\| \|x_2\| : x = x_1 \odot_I x_2, x_1 \in M_{n,I} \otimes X_1 , \,  x_2 \in M_{I,n} \otimes X_2\}
    \end{align*}
    where the infimum runs over all possible representations of $x = x_1 \odot_I x_2$ for any index $I$. The completion of $X_1 \otimes X_2$ with respect to this OSS is the \textit{extended Haagerup tensor product} $X_1 \eh X_2$.
\end{definition}

If $X$ is furthermore a dual space, we have the following theorem/definition from \cite{blecher-smith-w*h}.
\begin{definition}\cite[Section 2]{blecher-smith-w*h}
    Let $z \in B(\mc H) \otimes_{eh} B(\mc K)$. A \textit{weak representation} of $z$ has the form $z = \sum_{i \in I} x_i \otimes y_i$, where $x_i \in B(\mc H)$ and $y_i \in B(\mc K)$ and the sum converges weak-$*$. More precisely, the notation above means that for all $\phi \otimes \psi \in B(\mc H)_* \otimes B(\mc K)_*$ we have the following convergent net:
    \[
    \lim_{F \Subset I} \sum_{i \in F} \phi(x_i)\psi(y_i) = \phi \otimes \psi (z).
    \]
    Such representations are not generally unique, although there does exist a weak representation such that
    \[
    \bigg\| \sum_i x_i x_i^*\bigg\| \, \bigg\|\sum_i y^*_i y_i \bigg\| = \| \Phi_z\|_{cb}^2,
    \]
    where we recall $\Phi_z := \Phi(z)$.
\end{definition}
\begin{theorem}\cite[Theorem 3.1]{blecher-smith-w*h}
    Suppose $X$ and $Y$ are weak$^*$-closed subspaces of $B(\mc H)$ and $\mc B(K)$ respectively. Then $z \in X \otimes_{eh} Y$ if and only if $z$ has a weak representation $z = \sum_{i \in I} x_i \otimes y_i$ where $x_i \in X$ and $y_i \in Y$.
\end{theorem}

From the existence of weak representations one can deduce the following corollary.

\begin{corollary} \label{cor: alg-eh-density}
    If $X, Y$ are operator spaces, then $X^* \otimes Y^*$ is $\sigma(X \h Y)$-dense in $X^* \eh Y^*$.
\end{corollary}
\begin{proof}
    If we represent $X^* \subset B(\mc H)$ and $Y^* \subset B(\mc K)$, we see that there must exist a preannihilator $N$ such that
    \[
    B(\mc H)_* \h B(\mc K)_* / N \cong X \h Y.
    \]
    From the existence of weak representations we see $X^* \eh Y^*$ is the $\sigma(B(\mc H)_* \h B(\mc K)_*)$ closure of $X^* \otimes Y^*$. Thus $X^* \otimes Y^*$ must also be $\sigma(X \h Y)$-dense in $X^* \eh Y^*$.
\end{proof}

If $\mc A$ and $\mc B$ are operator spaces that are also algebras, then $\mc A \eh \mc B$ is also an algebra. The multiplicative structure of this algebra is key in connecting quantum graphs with left ideals in $\eh$.

\begin{theorem} \cite[Section 4]{blecher-smith-w*h}
    If $\mc A \subset B(\mc H_1)$ and $\mc B \subset B(\mc H_2)$ are w$^*$-closed unital subalgebras, then $A \eh B$ is a dual Banach algebra. The multiplicative structure can be written in terms of weak representations: if $\sum_{i \in I} a_i \otimes b_i, \sum_{j \in J} a_j \otimes b_j \in A \eh B$, the multiplication given by
    \[
    \bigg(\sum_{i \in I} a_i \otimes b_i\bigg)\bigg(\sum_{j \in J} a_j \otimes b_j\bigg) = \sum_{(i,j) \in I \times J} a_ia_j\otimes b_jb_i
    \]
    is well-defined.
\end{theorem}

The following theorem can be found in \cite{blecher-lemerdy} in the case of dual operator spaces and around \cite{opsp-hopf} in the case of general operator spaces.

\begin{theorem} \cites[1.6.9 (Weak$^*$ Haagerup tensor product)]{blecher-lemerdy}[Lemma 5.4]{opsp-hopf} \label{thm: eh extension}
    If $X_1, X_2, Y_1, Y_2$ are operator spaces and $\theta_1: X_1 \to Y_1$, $\theta_2: X_2 \to Y_2$ are c.b. maps, then there exists a canonical completely bounded extension
    \begin{align*}
        \theta_1 \otimes_{eh} \theta_2 : X_1 \otimes_{eh} X_2 &\to Y_1 \otimes_{eh} Y_2
    \end{align*}
    such that 
    \[
    \|\theta_1 \otimes_{eh} \theta_2\|_{cb} \leq \|\theta_1\|_{cb}\|\theta_2\|_{cb}.
    \]
    If $\theta_1,\theta_2$ are homomorphisms between Banach algebras then $\theta_1 \otimes_{eh} \theta_2$ is a homomorphism as well. If $X_1, X_2, Y_1, Y_2$ are furthermore dual operator spaces and $\sum_i x_i \otimes y_i$ is a weak representation in $X_1 \otimes Y_1$ then
    \[
    x_i \otimes y_i \mapsto \sum_i \theta_1(x_i) \otimes \theta_2(y_i)
    \]
    induces a well-defined map between weak representations. If $\theta_1,\theta_2$ are normal maps between Banach algebras then $\theta_1 \otimes_{eh} \theta_2$ is a normal map as well.
\end{theorem}

\paragraph{\underline{$\otimes_{\sigma h}$: Normal Haagerup Tensor Product}}

Unlike the other two Haagerup tensor products, the elements of the $\otimes_{\sigma h}$ tensor product do not generally have explicit forms given by $\odot$. Also unlike the other tensor products, the $\otimes_{\sigma h}$ tensor product is only defined for dual operator spaces. Indeed, its most concise description may be $(X \eh Y)^* = X^* \otimes_{\sigma h} Y^*$ for any operator spaces $X$ and $Y$ (see Remark \ref{rmk: predual-tops}). The definition below is not quite the one given in the original paper, but it follows immediately from the identifications in Theorem \ref{thm: m-haagerups}.

\begin{definition} \cite[p. 262]{normal-haagerup}
    If $\mc M, \mc N$ are von Neumann algebras, their normal Haagerup tensor product $\mc M \otimes_{\sigma h} \mc N$ is defined
    \[
    \mc M \otimes_{\sigma h} \mc N := (\mc M_* \eh \mc N _*)^{*}.
    \]
\end{definition}

We will be using a nice feature of $\otimes_{\sigma h}$ to detect a property of quantum graphs in Proposition \ref{prop: qr-to-ann}. If $\mc M$ is a von Neumann algebra, the multiplication map $m: \mc M \times \mc M \to \mc M$ is well-defined. Being a bilinear map, it must obviously factor through the algebraic tensor product $\mc M \otimes \mc M$. The latter is generally not an operator space, but the $\otimes_{\sigma h}$ tensor product provides the correct space that extends $m$. 

\begin{corollary} \label{cor: sigma-mult}
    For a von Neumann algebra $\mc M$, the multiplication map $m: \mc M \otimes_{\sigma h} \mc M \to \mc M$ is a $\sigma(\mc M_* \eh \mc M_*)$-continuous, completely contractive extension of the multiplication map $m: \mc M \times \mc M \to \mc M$.
\end{corollary}
\begin{proof}
    Under the identification 
    \[
    \mc M \otimes_{\sigma h} \mc M = CB_{\mc M',\mc M'}(B(\mc H))
    \]
    given by \cite[Theorem 2.5]{normal-haagerup}, the multiplication map is simply evaluation of a CB $\mc M'$-bimodule map at $I_{\mc H}$. 
\end{proof}

\begin{remark}
    As noted above, the tensor products $\hat \otimes$ and $\otimes_h$ are completions of the algebraic tensor product. Since $\eh$ allows its elements to factor through infinite matrices, however, we can no longer be sure that the algebraic tensor product is dense in $\eh$ (or $\otimes_{\sigma h}$, since we saw in Remark \ref{rmk: predual-tops} that $\mc M \eh \mc M \subset \mc M \otimes_{\sigma h} \mc M$ completely isometrically). Indeed, this is not generally true. What \textit{is} true is that the algebraic tensor products are dense in $\eh$ and $\otimes_{\sigma h}$ under the induced predual topologies, a fact that follows from a standard Hahn-Banach argument (see \cite[Lemma 2.1]{w*density-haagerup}, for instance).
\end{remark}

\paragraph{\underline{Continuity of the Multiplicative Product $\odot$}}

Finally, we have the following theorem from \cite{gp-obs} that was implicitly used in the definition of the multiplicative product.

\begin{theorem}\protect{\cite[Lemma 3.1]{gp-obs}} \label{defn: mult-prod-gp}
    The multiplicative product $\odot$ is a bilinear map extending the action
    \begin{align*}
    \odot : \overline{B(\ell^2) \otimes \mc M}^{\sigma(B(\ell^2)_* \hat \otimes \mc M_*)} \times \overline{B(\ell^2) \otimes \mc M}^{\sigma(B(\ell^2)_* \hat \otimes \mc M_*)} &\to \overline{B(\ell^2) \otimes (\mc M \eh \mc M)}^{\sigma(B(\ell^2)_* \widehat \otimes (\mc M_* \otimes_{eh} \mc M_*))} \\
    (S \otimes x) \odot (T \otimes y) &\mapsto ST \otimes (x \otimes y).
    \end{align*}
   It is continuous over bounded sets if the domain is endowed with the product SOT and the codomain with the $\sigma(B(\ell^2)_* \hat \otimes (\mc M_* \otimes_{eh} \mc M_*))$ topology.
\end{theorem}

 We may abbreviate the space in the codomain by $B(\ell^2) \overline{\otimes}^{\sigma(eh)} (\mc M \eh \mc M)$. 

\begin{remark} \label{rmk: gp-topologies}
    The topologies above may require explanation. The space $B(\ell^2) \overline{\otimes} \mc M$ is the usual von Neumann algebraic tensor product (see Subsection \ref{prelim: op alg}) whose unique Banach space predual $B(\ell^2)_* \widehat \otimes \mc M$ is the operator space projective tensor product of the predual $B(\ell^2)_*$ and the von Neumann algebra $\mc M$.
    
    The space $B(\ell^2) \overline{\otimes} (\mc M \otimes_{eh} \mc M)$ is closed in a predual topology induced by the operator space projective tensor product the predual $B(\ell^2)_*$ and the extended Haagerup tensor product $\mc M_* \eh \mc M_*$. To give a flavor of how this topology works, it may be best to contrast the ``natural'' predual topologies of $B(\ell^2) \overline{\otimes} (\mc M \otimes_{eh} \mc M)$ and $B(\ell^2) \overline{\otimes} (\mc M \otimes_{\sigma h} \mc M)$. We begin with the former. Assume that $\mc M = B(\mc H)$. We have that 
    \[
    B(\ell^2) \overline \otimes (\mc M \otimes_{eh} \mc M) = (B(\ell^2)_* \hat \otimes (CB(K(\mc H),B(\mc H)))_*)^*
    \]
    and since $CB(X,Y^*) = (X \hat \otimes Y)^*$ (see \cite[Theorem 4.1]{intro-op-sp-pisier} for a reference),
    \[
    (B(\ell^2)_* \hat \otimes (CB(K(\mc H),B(\mc H)))_*)^* = (B(\ell^2)_* \hat \otimes (K(\mc H) \hat \otimes B(\mc H)_*))^*.
    \]
    Using the symmetry and associativity of $\hat \otimes$ (\cite[Chapter 1.2.2]{opsptennorm}), 
    \begin{align*}
    (B(\ell^2)_* \hat \otimes (K(\mc H) \hat \otimes B(\mc H)_*))^* &= ((K(\mc H) \hat \otimes B(\ell^2)_* \hat \otimes B(\mc H)_*))^* \\
    &= CB(K(\mc H), (B(\ell^2)_* \hat \otimes B(\mc H)_*)^*) \\
    &= CB(K(\mc H), B(\ell^2) \bar \otimes B(\mc H)) \\
    &= CB(K(\mc H), B(\ell^2 \otimes_2 \mc H)).
    \end{align*}
    Every element $[z_{ij}]_{i,j} \in B(\ell^2) \overline{\otimes} (\mc M \otimes_{eh} \mc M)$ can be considered as an infinite matrix with entries $z_{ij} \in \mc M \eh \mc M$. Every operator $T \in K(\mc H)$ turns $[z_{ij}]_{i,j}$ into an infinite matrix with $B(\mc H)$ entries by
    \[
    [z_{ij}]_{i,j} \mapsto [\Phi_{z_{ij}}(T)]_{i,j} \in B(\ell^2) \overline{\otimes} B(\mc H).
    \]
    A net $(z^\lambda := [z^\lambda _{ij}]_{i,j})_{\lambda \in \Lambda}$ converges in the $\sigma(B(\ell^2)_* \hat \otimes (\mc M_* \otimes_{h} \mc M_*))$ topology to $z = [z_{ij}]_{ij}$, then for any $T \in K(\mc H)$ and functional $\phi \in B(\ell^2 \otimes_2 \mc H)_*$ we have 
    \[
    \phi\bigg( [\Phi_{z^\lambda_{ij}-z_{ij}}(T)]_{i,j}\bigg) \overset{\lambda \in \Lambda}{\longrightarrow} 0.
    \]
    (The above convergence is a priori a coarser topology than $\sigma(B(\ell^2)_* \hat \otimes (\mc M_* \otimes_{h} \mc M_*))$, but we hope this demystifies the symbols in the predual.) The natural topology on $B(\ell^2) \overline{\otimes} (\mc M \otimes_{\sigma h} \mc M)$ is almost exactly the same except one replaces $K(\mc H)$ by $B(\mc H)$ and $\otimes_h$ by $\eh$. Although the range of $\odot$ is $B(\ell^2) \bar \otimes (\mc M \eh \mc M)$, the topology is the predual topology from $B(\ell^2) \bar \otimes (\mc M \otimes_{\sigma h} \mc M)$.
\end{remark}

\underline{Representation-Free Characterization of Quantum Graphs}  \label{par: annihilator characterization}

We can now show how the $\eh$ tensor product allows us to quantize the complementary set of edges of a graph. We begin with an example of how this works for a finite classical graph. The example also works for classical graphs on infinitely many vertices, but we only present the finite case to avoid an extra topological consideration. The infinite case will in any case be subsumed in the general characterization in Theorem \ref{thm: gp-bijection}.

\begin{example} \label{ex: eh-ann}
    Let $G = (V,E)$ be a finite graph, and let $\mc M = \ell^\infty(V)$, $\mc S \subset B(\ell^2(V))$ be its quantization as in Example \ref{ex: classical-to-quantum-graph}. Note that the map $\Phi$ (Theorem \ref{defn: Phi-map}) turns $B(\ell^2(V))$ into a left $\ell^\infty(V) \eh \ell^\infty(V)$-module. We can therefore speak of annihilators of $\ell^\infty(V)$-bimodules inside $\ell^\infty(V) \eh \ell^\infty(V)$. For $v \in V$ denote the indicator function on $\{v\}$ by $\chi_v$.
    \begin{align*}
    \Ann(\mc S) &:= \{z \in \ell^\infty(V) \otimes \ell^\infty (V) : \Phi_z(\mc S) = 0 \} \\
    &= \Span\{\chi_v \otimes \chi_w : (v,w) \notin E\}.
    \end{align*}
\end{example}

(Part of) \cite[Theorem 3.3]{gp-obs} shows that correspondence above generalizes to a bijection between quantum graphs on an arbitrary von Neumann algebra $\mc M$ and $\sigma(\mc M_* \eh \mc M_*)$-closed left ideals. 

\begin{theorem}\protect{\cite[Theorem 3.3]{gp-obs}} \label{thm: gp-bijection}
    Let $\mc M \subset B(\mc H)$ be a von Neumann algebra. There is a bijection between $\sigma(\mc M_* \eh \mc M_*)$-closed left ideals in $\mc M \eh \mc M$ and quantum graphs on $\mc M$. Namely, each quantum graph corresponds uniquely to its annihilator in $\mc M \eh \mc M$:
    \[
    \Ann(\mc S) = \{z \in \mc M \otimes_{eh} \mc M : \Phi_z(T) = 0, \, T \in \mc S\}
    \]
    where $\Phi$ is the map in Definition \ref{defn: Phi-map}.
\end{theorem}

As a result of this theorem, we may at times call a pair $(\mc M, \Ann(\mc S))$ a quantum graph, where $\Ann(\mc S)$ is a $\sigma(\mc M_* \eh \mc M_*)$-closed left ideal. 

\begin{remark}
We pause here to marvel at this correspondence. A common observation (or objection) when encountering the definition of a quantum graph is the dependence on the representation of $\mc M \subset B(\mc H)$. One of the first results \cite[Theorem 2.7]{weaverqrelations} regarding quantum graphs establishes a bijection between quantum graphs on isomorphic von Neumann algebras \cite[Theorem 2.7]{weaverqrelations}. The need for representation independence also drove the development of \textit{intrinsic quantum relations} in \cite[Definition 2.24]{weaverqrelations}, although the definition uses the weak operator topology. We shall also see in Subsubsection \ref{subsubsec: frob-morphism} that the definition of quantum graphs from \cite{musto-reutter-verdon} implicitly chooses an inner product structure for their (necessarily finite-dimensional) von Neumann algebras. However, since the $\eh$ tensor product and the topology $\sigma(\mc M_* \eh \mc M_*)$ are representation independent, we can call these left ideals ``quantum graphs'' without worrying about the ambient space $B(\mc H)$.
\end{remark}

\begin{remark}
    As the author emphasizes in \cite[Section 3]{gp-obs}, the topology under which the annihilators are closed is $\sigma(\mc M_*\otimes_{eh} \mc M_*)$, not $\sigma(\mc M_* \otimes_h \mc M_*)$. See Remark \ref{rmk: predual-tops} for the distinction between these predual topologies. The choice of topology is necessary because quantum graphs are subsets of $B(\mc H)$, not $K(\mc H)$ (the compact operators). Compare the two following topologies:
    \begin{align*}
        z_i \overset{\sigma(\mc M_* \eh \mc M_*)}{\longrightarrow} z &\implies \< \phi, \Phi_{z_i - z}(T) \> = 0 \qquad \forall T \in B(\mc H), \, \phi \in B(\mc H)_* \\
        z_i \overset{\sigma(\mc M_* \otimes_h \mc M_*)}{\longrightarrow} z &\implies \< \phi, \Phi_{z_i - z}(T)\> = 0 \qquad \forall T \in K(\mc H), \, \phi \in B(\mc H)_*.
    \end{align*}
    A left ideal in $\mc M \eh \mc M$ that is the annihilator of a quantum graph must be able to distinguish whether an arbitrary $T \in B(\mc H)$ is in the quantum graph. Even though every $T \in B(\mc H)$ is the $B(\mc H)_*$-limit of compact operators, one would need to interchange two limits in order for the two topologies above to be equivalent. 
\end{remark}

\begin{remark}
    A recent preprint has defined \textit{Hilbert Schmidt quantum relations} \cite[Definition 4.1]{hsqgraphs}. These quantum relations are still $\mc M'$-bimodules $\mc S \subset B(\mc H)$, but they are $\|\cdot \|_{HS}$-closed rather than $\sigma(\mc B(\mc H))_*$-closed. In this case, the $\otimes_{eh}$ is the incorrect tensor product to define quantum graphs in terms of their annihilators. The key point is that for each $z \in \mc M \eh \mc M$, the kernel of $\Phi_z: B(\mc H) \to B(\mc H)$ is a $\sigma(B(\mc H)_*)$-closed subspace. Hence the subspace annihilated by any subset of $\mc M \eh \mc M$ is a $\sigma(B(\mc H)_*)$-closed subspace. 
\end{remark}

\subsection{Categorical (Co)limits} \label{subsec: cat}

If we are to take categorical (co)limits of quantum graphs, we will necessarily need candidates for the ``vertices'' and ``complementary edges'' of the quantum graph in the limit. These will be furnished by the limit of von Neumann algebras and (a closure of a ) limit of operator spaces respectively. A reader who believes that one can always take the inductive/projective limits of von Neumann algebras and operator spaces may wish to skip this section of the preliminaries. The definitions below are not always the canonical category theoretic ones; we will at some times make assumptions that match the more intuitive understanding of categories. First, some vocabulary. 

\begin{definition} \cite[Chapter I.2]{cats-working}
    A category $\mc C$ is a set of objects Ob$(\mc C)$ and a set of morphisms Mor$(\mc C)$. Each morphism $\varphi \in \text{Mor}(\mc C)$ is assigned a domain $\mc A$ and codomain $\mc B$, each of which are objects. This is denoted $\varphi: \mc A \to \mc B$ or $\mc A \overset{\varphi}{\to} \mc B$. Each object $\mc A$ has an identity morphism $\iota_{\mc A}$ whose domain and codomain are both $\mc A$. If the domain of $\varphi_1$ is the same as the codomain of $\varphi_2$, then we may compose the maps $\varphi_2\circ \varphi_1$. This composition operation is associative. Finally, for any morphism $\varphi: \mc A \to \mc B$ we always have
    \[
    \iota_{\mc B} \circ \varphi = \varphi \quad \text{and}f \quad \varphi \circ \iota_{\mc A} = \varphi.    
    \]
\end{definition}

\begin{definition} \cite[Chapter I.5]{cats-working}
    A morphism $\varphi: \mc A \to \mc B$ is an \textit{isomorphism} if there exists a morphism $\psi: \mc B \to \mc A$ such that 
    \[
    \varphi \circ \psi = \iota_{\mc B} \quad \text{and} \quad \psi \circ \varphi = \iota_{\mc A}.
    \]
\end{definition}

\begin{definition} \label{defn: limit}
    Let $\mc C$ be a category and let $(J, \leq)$ be a poset. A \textit{projective system} is a collection of objects $\{\mc A_j\}_{j \in J}$ in $\mc C$ and a collection of morphisms $\{\mc A_j \overset{\varphi^{j,k}}{\leftarrow} \mc A_k\}_{j \leq k}$. If $i \leq j \leq k$, then we also require 
    \[
    \varphi^{i,k} = \varphi^{i,j} \circ \varphi^{j,k}.
    \]
\end{definition}

\begin{definition}
    Let $\{\{\mc A_j\}_{j \in J}, \{\varphi^{j,k}\}_{j\leq k}\}$ be an projective system in $\mc C$. The \textit{limit} is an object $\mc A:= \varprojlim_{\mc C} \mc A_j$ in $\mc C$ and a collection of morphisms $\{\mc A_j \overset{\varphi^{j,\infty}}{\leftarrow} \mc A\}_{j \in J}$ such that
    \begin{enumerate}
        \item for every $j \leq k$ the following diagram commutes
        \begin{center}
        \begin{tikzcd}
\mathcal A_j & \mathcal A_k \arrow[l, "{\varphi^{j,k}}"', bend right] &  & \mathcal A \arrow[ll, "{\varphi^{k,\infty}}", bend left] \arrow[lll, "{\varphi^{j,\infty}}", bend left=49]
\end{tikzcd}
    \end{center}
    \item and if there exists another object $\mc B$ and a collection of morphisms $\{\mc A_j \overset{\phi^{j,\infty}}{\leftarrow} \mc B\}_{j \in J}$ such that the following diagram commutes
    \begin{center}
        \begin{tikzcd}
\mathcal A_j & \mathcal A_k \arrow[l, "{\varphi^{j,k}}"', bend right] &  & \mathcal B \arrow[ll, "{\phi^{k,\infty}}", bend left] \arrow[lll, "{\phi^{j,\infty}}", bend left=49]
\end{tikzcd}
    \end{center}
    there exists a unique morphism $\mc A \overset{\psi}{\leftarrow} \mc B$ such that for all $j$ this last diagram commutes:
    \begin{center}
        \begin{tikzcd}
\mathcal A_j &  & \mathcal A \arrow[ll, "{\varphi^{j,\infty}}"]                         \\
             &  & \mathcal B \arrow[u, "\psi", dashed] \arrow[llu, "{\phi^{j,\infty}}"]
\end{tikzcd}
    \end{center}
    \end{enumerate}
    If $J$ is additionally a directed set, the limit $\varprojlim_{\mc C} \mc A_j$ is sometimes called a \textit{projective limit}.
\end{definition}

Inductive systems and limits are defined in precisely the dual way, but we will take this opportunity to set notation.

\begin{definition}
    Let $\mc C$ be a category and let $(J, \leq)$ be a poset. An \textit{inductive system} is a collection of objections $\{\mc A_j\}_{j \in J}$ in $\mc C$ and a collection of morphisms $\{A_j \overset{\varphi_{j,k}}{\to} A_k\}_{j \leq k}$. If $i \leq j \leq k$, then we also require that the morphisms satisfy
    \[
    \varphi_{i,k} = \varphi_{i,j} \circ \varphi_{j,k}.
    \]
\end{definition}

\begin{definition} \label{defn: colimit}
    Let $\{\{\mc A_j\}_{j \in J}, \{\varphi_{j,k}\}_{j \leq k}\}$ be an inductive system in $\mc C$. A \textit{colimit} is an object $\mc A:= \varinjlim_{\mc C} \mc A_j$ in $\mc C$ and a collection of morphisms $\{\mc A_j \overset{\varphi_{j,\infty}}{\to} \mc A\}_{j \in J}$ such that
    \begin{enumerate}
        \item for every $j \leq k$ the following diagram commutes
    \begin{center}
        \begin{tikzcd}
\mathcal A_j \arrow[r, "{\varphi_{j,k}}", bend left] \arrow[rrr, "{\varphi_{j,\infty}}"', bend right=49] & \mathcal A_k \arrow[rr, "{\varphi_{k,\infty}}", bend right] &  & \mathcal A
\end{tikzcd}
    \end{center}
    \item and if there exists another object $\mc B$ and a collection of morphisms $\{\mc A_j \overset{\phi_{j,\infty}}{\to} \mc B\}_{j \in J}$ such that the following diagram commutes
    \begin{center}
        \begin{tikzcd}
\mathcal A_j \arrow[r, "{\varphi_{j,k}}", bend left] \arrow[rrr, "{\phi_{j,\infty}}"', bend right=49] & \mathcal A_k \arrow[rr, "{\phi_{k,\infty}}"', bend right] &  & \mathcal B
\end{tikzcd}
    \end{center}
    there exists a unique morphism $\mc A \overset{\psi}{\to} \mc B$ such that for all $j$ this last diagram commutes:
    \begin{center}
\begin{tikzcd}
\mathcal A_j \arrow[rr, "{\varphi_{j,\infty}}"] \arrow[rrd, "{\phi_{j,\infty}}"'] &  & \mathcal A \arrow[d, "\psi", dashed] \\ 
&  & \mathcal B                          
\end{tikzcd}
    \end{center}
    \end{enumerate}
    If $J$ is additionally a directed set, the colimit $\varinjlim_{\mc C} \mc A_j$ is sometimes called an \textit{inductive limit}.
\end{definition}

\begin{remark}
    An immediate corollary of the second axioms of Definitions \ref{defn: limit} and \ref{defn: colimit} is that the colimit is unique up to isomorphism.
\end{remark}

\begin{example}
    The poset $(J,\leq)$ is often taken to be $\N$ with the usual ordering and the inductive system is depicted by
    
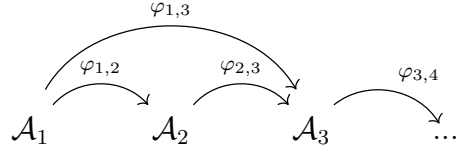
\begin{center}
        \begin{tikzcd}
        \mathcal A_1 \arrow[r, "{\varphi_{1,2}}", bend left=49] \arrow[rr, "{\varphi_{1,3}}", bend left=60, shift left=2] & \mathcal A_2 \arrow[r, "{\varphi_{2,3}}", bend left=49] & \mathcal A_3 \arrow[r, "{\varphi_{3,4}}", bend left=49] & ...
        \end{tikzcd}
    \end{center}
\end{example}

\begin{example}
    If there is an inductive limit for an inductive system along $(\N, \leq)$, we depict the inductive limit like so:
    
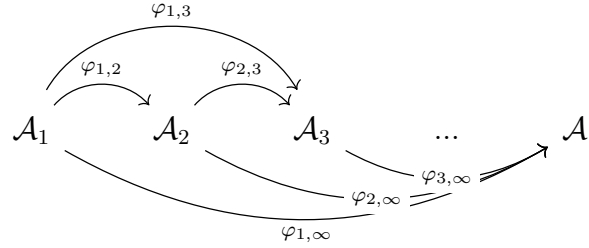
\begin{center}
        \begin{tikzcd}
        \mathcal A_1 \arrow[r, "{\varphi_{1,2}}", bend left=49] \arrow[rr, "{\varphi_{1,3}}", bend left=60, shift left=2] \arrow[rrrr, "{\varphi_{1,\infty}}"', bend right] & \mathcal A_2 \arrow[r, "{\varphi_{2,3}}", bend left=49] \arrow[rrr, "{\varphi_{2,\infty}}" description, bend right] & \mathcal A_3 \arrow[rr, "{\varphi_{3,\infty}}" description, bend right] & ... & \mathcal A
        \end{tikzcd}
    \end{center}
\end{example}

An operator space theorist might naturally assume that the morphisms of a category whose objects are operator spaces must be the completely bounded maps between them. However, this category does not have all (co)limits, as exhibited by the example below. For this reason, we consider the category of operator spaces with completely contractive morphisms in the remainder of this paper.

\begin{example} \label{ex: ccmap}
    Let $J = (\N,\leq)$ and fix a non-zero operator space $\mc S$. Define $\mc S_j$ to be the $j$-fold direct sum of $\mc S$, and define the connecting morphisms $\varphi_{j,k}$ to be inclusion into the first $j$ summands for $j \leq k$:
    \begin{align*}
    \varphi_{j,k} : \mc S_j &\to \mc S_k. \\
    (T_1, \ldots , T_j) &\mapsto (T_1, \ldots , T_j, 0,...,0)
    \end{align*}
    Suppose there does exist a colimit $(\mc A, \{\varphi_{j,\infty}\}_{j \in \N})$ of this inductive system. Define the maps
    \begin{align*}
        \psi_j : \mc S_j &\to \prod^\infty_{j} \mc S \\
        (T_1, T_2, T_3,...,T_j) &\mapsto (\|\varphi_{1,\infty}\|T_1, 2\|\varphi_{2,\infty}\|T_2, ...,j\|\varphi_{j,\infty}\|T_j, 0,...)
    \end{align*}
    (these maps scale the $n$th summand by $n\|\varphi_{n,\infty}\|$). By the universal property of colimits there exists a unique map $\psi: \mc A \to \mc S$ such that the diagram below commutes.
    
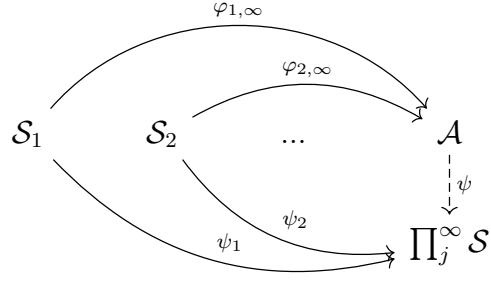
\begin{center}
        \begin{tikzcd}  
        \mc S_1 \arrow[rrrd, "\psi_1", bend right] \arrow[rrr, "{\varphi_{1,\infty}}", bend left=49] & \mc S_2 \arrow[rrd, "\psi_2", bend right] \arrow[rr, "{\varphi_{2,\infty}}", bend left] & ... & \mathcal A \arrow[d, "\psi", dashed] \\  
        &   &     & \prod_j^\infty \mc S                                  
        \end{tikzcd}
    \end{center}    
    Since the operator norm is submultiplicative, we have that 
    \[
    j\|\varphi_{j,\infty}\| \leq \|\psi_j\| \leq \|\psi\|\|\varphi_{j,\infty}\| 
    \]
    for all $j \in \N$. Therefore $\|\psi\|_{cb}$ is also unbounded, and there is no inductive limit for this particular inductive system.
\end{example}

We are therefore unable to take colimits in the category of operator spaces and CB maps. To remedy this issue, we require our morphisms between operator spaces to be completely contractive maps and will refer to this category as $\textbf{OpSp}$. This is a slightly stronger convention than necessary in order to take limits; see \cite{daws} for possible relaxations of this condition on the CB norms. However, a nice feature of this category is that we immediately have that that the categorical isomorphism must be completely isometric isomorphisms. For facts about $\textbf{OpSp}$ we will refer to the preprint \cite{categoryopsp}. (We will depart from some of their naming conventions around direct sums vs. products, however.) We begin by demonstrating a ``concrete'' construction of an inductive limit in $\textbf{OpSp}$. This description is one of several taken from \cite{daws} for an inductive limit of Banach algebras along a directed set. Inductive limits in the category of associative algebras are similar to inductive limits in the category of modules, but their colimits can be very different. We will see such an example immediately after this construction.  

\begin{example}[``Concrete'' Construction]
Let $(\mc S_j, \varphi_{i,j})$ be an inductive system along a directed set $(J,\leq)$ in $\textbf{OpSp}$. Denote the $\ell^\infty$ product $\prod^\infty_{j \in J} \mc S_j$ by
\[
\prod^\infty_{j \in J} \mc S_j := \bigg\{ (T_j)_{j \in J}  \in \prod_{j \in J} \mc S_j : \sup_{j \in J} \|T_j\| < \infty\bigg\}
\]
and the $c_0$ product $\prod^{c_0}_{j \in J} \mc S_j$ by
\[
\prod^{c_0}_{j \in J} \mc S_j := \bigg\{ (T_j)_{j \in J}  \in \prod_{j \in J} \mc S_j : \forall \epsilon > 0 \, \exists j_0 \text{ s.t. } j\geq j_0 \implies \|T_j\| < \epsilon \bigg\}.
\]
The $\ell^\infty$ product is an operator space \cite[Proposition 4.11]{categoryopsp}, so an immediate corollary is that the $c_0$ product is also an operator space as a norm-closed subset of the $\ell^\infty$ product. We begin by defining a quotient map using the spaces above:
\[
q: \prod^\infty_{j \in J} \mc S_j \to {\prod^\infty_{j \in J} \mc S_j} \Big/ \prod^{c_0}_{j \in J} \mc S_j.
\]
For each $j \in J$ and $T_j \in \mc S_j$ we define 
\begin{align*}
\psi_{j,\infty}: \mc S_j &\to \prod^\infty_{j \in J} \mc S_j \\
T_j &\mapsto \begin{cases} \varphi_{j,k}(T_j) & j \leq k \\
0 & o.w.
\end{cases}.
\end{align*}
$\varphi_{j,\infty} := q \circ \psi_{j,\infty}$ gives a well-defined complete contraction from $\mc S_j$ into ${\prod^\infty_{j \in J} \mc S_j} \Big/ \prod^{c_0}_{j \in J} \mc S_j$. The inductive limit of $(\mc S_j, \varphi_{i,j})$ is the object  
\[
\varinjlim_{\bf{OpSp}} \mc S_j := \overline{\bigcup_{j \in J}\varphi_{j,\infty}(\mc S_j)}^{\|\cdot \|}
\]
along with the morphisms $\varphi_{j,\infty}$ defined above.
\end{example}

The preceding example is probably somewhat familiar construction to operator algebraists as the colimit of operator spaces. By contrast, Theorem \ref{thm: products-equalizers} below is the usual category theoretic way to obtain (co)limits at the cost of defining (co)products and (co)equalizers.

\begin{definition} \protect{\cite[Section III.3]{cats-working}}
    Let $\mc C$ be a category and $\{\mc A_j\}_{j \in J}$ be a collection of objects in $\mc C$. The coproduct of $\{\mc A_j\}_{j \in J}$ is an object $\mc A$ and morphisms $\{\mc A_j \overset{\varphi_{j,\infty}}{\to} \mc A\}_{j \in J}$ such that for any object $\mc B$ and morphisms $\{\mc A_j \overset{\varphi_{j,\infty}'}{\to} \mc B\}$ there exists a unique morphism $\mu: \mc A \to \mc B$ making the following diagram commute for all $j \in J$:
    \begin{center}
        \begin{tikzcd}
            \mc A_j \arrow[rd, "{\varphi_{j,\infty}'}"'] \arrow[r, "{\varphi_{j,\infty}}"] & \mc A \arrow[d, "\mu.", dotted] \\
            & \mc B                         
        \end{tikzcd}
    \end{center}
\end{definition}

\begin{definition} \protect{\cite[Section III.3]{cats-working}}
    Let $\mc C$ be a category and $\mc A \overset{f,g}{\to} \mc B$ be morphisms with the same domain and codomain. The \textit{coequalizer} of $f,g$ is a morphism $\mc B \overset{q}{\to} \mc E$ such that $q \circ f = q\circ g$ and if there exists a morphism $\mc B \overset{h}{\to} \mc C$ such that $h \circ f = h \circ g$ then there exists a unique morphism $\mc E \overset{h'}{\to} \mc C$ such that the following diagram commutes:
    \begin{center}
        \begin{tikzcd}
        \mc A \arrow[r, "f", shift left] \arrow[r, "g"', shift right] & \mc B \arrow[r, "q"] \arrow[rd, "h"'] & \mc E \arrow[d, "h'.", dashed] \\
        &     & \mc C                        
        \end{tikzcd}
    \end{center}
\end{definition}

\begin{remark}
    As promised, we have a colimit that is not an inductive limit. A coproduct is a colimit along a poset $J$ in which no two distinct indices in $J$ are comparable and every index is comparable to itself.
\end{remark}

As per usual in category theory, the notions of \textit{coproduct} and \textit{coequalizer} are defined in precisely the same way but with the diretions of the morphism reversed. We omit these definitions for brevity.

\begin{theorem} \protect{\cite[Chapter V.2, Theorem 1]{cats-working}} \label{thm: products-equalizers}
    If a category has all small \footnote{\textit{Small} here refers to small classes (i.e., sets) as opposed to proper classes.} products and all equalizers, it has all small limits. If a category has all small coproducts and all coequalizers, it has all small colimits. 
\end{theorem}

\begin{proposition} \protect{\cite[Proposition 4.11]{categoryopsp}}
    The product in $\textbf{OpSp}$ is the $\ell^\infty$ product. 
\end{proposition}

\begin{proposition} \protect{\cite[Proposition 4.13]{categoryopsp}}
    Let $\mc S \overset{f,g}{\to} \mc T$ be morphisms in \textbf{OpSp}. Their equalizer is the inclusion map of $\mc S_0 := \{ x \in \mc S | f(x) = g(x)\}$ into $\mc S$.
\end{proposition}

\begin{proposition} \protect{\cite[Proposition 4.12]{categoryopsp}}
    The coproduct in $\textbf{OpSp}$ is the $\ell^1$ product:
    \[
    \prod_{j \in J}^{1} \mc S_j := \bigg\{ (T_j)_{j \in J} \in \prod_{j \in J} \mc S_j : \sum_{j \in J} \|T_j\| < \infty\bigg\}.
    \]
\end{proposition}

\begin{proposition} \protect{\cite[Proposition 4.15]{categoryopsp}}
    Let $\mc S \overset{f,g}{\to} \mc T$ be morphisms in $\textbf{OpSp}$, and define 
    \[
    \overline{Im(f-g)}:= \overline{\{(f-g)(T) : T \in \mc S\}}^{\|\cdot\|_{\mc T}}
    \]
    Then their coequalizer is the quotient map $q: \mc T \to \mc T/\overline{Im(f-g)}$.
\end{proposition}   

We now review (co)limits in the category of von Neumann algebras. In this paper we will use $W^*$-algebras and von Neumann algebras interchangeably. In some contexts (including \cite{kornell} cited below), $W^*$-algebras are abstract and von Neumann algebras are represented on some Hilbert space. However, every von Neumann algebra has a canonical faithful unital normal $*$-representation and the morphisms of $W^*$-algebras and von Neumann algebras are the same, so there is not really a distinction. To be clear: by $\mathbf{W^*}$ we mean the category in which 
\begin{itemize}
    \item each object is a unital $C^*$-algebra with a unique Banach space predual, dubbed von Neumann algebras
    \item each morphism is a normal, unital $*$-homomorphism.
\end{itemize}
We will be including the zero algebra $B(\{0\})$ in this category. The results below are taken from \cite{kornell}, but (as the author acknowledges) the existing literature contains versions of some results. Notably \cite{guichardet} proves some of the categorical features used in this paper, including a proof of ${\bf W^*}$ colimits using ${\bf C^*}$ (the category of $C^*$-algebras and $*$-homomorphisms) colimits.

\begin{proposition} \protect{\cite[Proposition 5.1]{kornell}}
    The product in the category $\mathbf{W^*}$ is the $\ell^\infty$ product.    
\end{proposition} 

\begin{proposition}\protect{\cite[Proposition 5.3]{kornell}} \label{prop: w*equalizers}
    The category $\mathbf{W^*}$ has all equalizers.
\end{proposition}
\begin{proof}
    Suppose we have morphisms $\mc M \overset{f,g}{\to} \mc N$. The sub von Neumann algebra 
    \[
    \mc K := \{x \in \mc M : f(x) = g(x)\} 
    \]
    with the inclusion map $\iota : \mc K \to \mc M$ is their equalizer.
\end{proof}

We have expanded the proof of \cite[Theorem 5.5]{kornell} below to hopefully reduce the opacity of the coproduct in ${\bf W^*}$ for those unfamiliar with category theoretical constructions. Thank you to David Penneys for an enlightening clarification of the proof.

\begin{proposition} \protect{\cite[Proposition 5.5]{kornell}}
    The category $\mathbf{W^*}$ has all small coproducts.
\end{proposition}
\begin{proof}[Proof Sketch]
    Let $\mc M$ be a von Neumann algebra. Let $C_{\mc M}$ be the collection of $\mathbf{W^*}$ morphisms from $\mc M$:
    \[
    C_{\mc M} := \{\iota^{\mc M}_{\mc N}: \mc M \to \mc N\mid \iota^{\mc M}_{\mc N} \text{ is a $\mathbf{W^*}$ morphism for some von Neumann algebra $\mc N$}\}.
    \]
    Obviously $C_{\mc M}$ is a proper class, so we choose one von Neumann algebra per isomorphism class in $C_{\mc M}$. Call this subclass $S_{\mc M}$. It follows from \cite[Lemma 5.4]{kornell} that $S_{\mc M}$ is a set. Now let $\{\mc M_j\}_{j \in J}$ be a set of von Neumann algebras. From the lines above, the collection below is also a set:
    \[
    S_J := \bigg\{\mc N : \forall j \in J, \, \iota_{\mc N}^{\mc M_j} \in S_{\mc M_j} \text{ and } \bigcup_j \iota_{\mc N}^{\mc M_j} (\mc M_j) \text{ generates $\mc N$} \bigg\}.
    \]
    For each $j \in J$ define the morphism
    \begin{align*}
    \iota_j : \mc M_j &\to \prod_{\mc N \in S_J}^{\infty} \mc N. \\
    x &\mapsto (\iota_{\mc N}^{\mc M_j}(x))_{\mc N \in S_J}
    \end{align*}
    Denote by $\mc M_\infty$ the smallest von Neumann algebra containing $\bigcup_j \iota_j(\mc M_j)$ inside $\prod_{\mc N \in S_J} \mc N$. (We can detect inclusion using equalizers, Proposition \ref{prop: w*equalizers}.) The coproduct of $\{\mc M_j\}_{j \in J}$ is $\mc M_\infty$ with the set of morphisms
    \[
    \{\mc M_j \overset{\iota_j}{\to} \mc M_\infty\}_{j \in J}.
    \]
\end{proof}

\begin{remark} \label{rmk: free product vna}
    The author of \cite{kornell} calls the coproduct in ${\bf{W^*}}$ the \textit{free product} of von Neumann algebras. This is in analogy with the category of $C^*$-algebras and $*$-homomorphisms where the coproduct is the free product $C^*$-algebra (see \cite[Section 2]{pedersen-free-product}). 
\end{remark}

\begin{proposition} \protect{\cite[Proposition 5.7]{kornell}}
    The category $\mathbf{W^*}$ has all coequalizers.
\end{proposition}
\begin{proof}
    Suppose we have morphisms $\mc M \overset{f,g}{\to} \mc N$. Let $\mc I$ be the $w^*$-closed two-sided ideal generated by $\{f(x)- g(x) : x \in \mc M\}$ so that the quotient map $q: \mc N \to \mc N/\mc I$ is a ${\bf {W^*}}$ morphism. The map $q$ is the coequalizer of $f$ and $g$. 
\end{proof}

\begin{remark}
One may wonder if this categorical machinery was truly necessary. For the construction in Section \ref{sec: limit of q graphs}, not particularly. However, in the classical case there is an interesting result that a Cayley graph of a profinite group is a profinite graph (see Question \ref{question: qcayleygraph}). The Cayley graph in question is a limit of finite graphs along a directed set that is very unlikely to be $\N$. Moving forward, this categorical machinery will mostly be swept under the rug so the upfront cost of establishing limits along any directed set seemed frugal. 
\end{remark}

To summarize: the categories $\textbf{OpSp}$ and $\mathbf{W^*}$ have all small limits and colimits. Every quantum graph can be characterized in terms of a von Neumann algebra $\mc M$ (roughly, the quantization of the vertices) and an operator space in $\mc M \eh \mc M$ (roughly, the quantization of the complementary edge set). The (co)limits above give us the candidates for the (co)limits of the quantum graphs.

\section{Quantum Graph Morphisms} \label{sec: morphisms}

We continue with a section on quantum graph morphisms rather than quantum graphs, the reason being our definition of connectedness (Definition \ref{defn: connected}) uses the morphism in Definition \ref{defn: qgraph morph}.

Various other authors have given definitions for a morphism of quantum graphs. In Subsection \ref{subsec: equiv-morphs} we will see that any two of the notions are equivalent when taking the more restrictive set of hypotheses. We begin with morphisms of classical graphs to give intuition for all these definitions.

\subsection{Morphisms of Classical Graphs} \label{subsec: classical graph morphisms}
In this subsection the classical graphs will be on a finite set of vertices. The following could be extended to infinite classical graphs as well, but this section mainly serves as motivation for the quantum case. As mentioned above, the case of infinitely many vertices will be subsumed in the quantization.

\begin{definition}
    A graph $G = (V,E)$ consists of a set of vertices $V$ and and a set of edges $E \subseteq V \times V$.
\end{definition}

\begin{definition}
    If $G_1 = (V_1, E_1)$ and $G_2 = (V_2, E_2)$ are two classical graphs, we say a function between the vertices $V_1 \overset{f}{\leftarrow} V_2$ induces a \textit{morphism of classical graphs} if
    \[
    (f(v),f(w)) \in E_1 \Longleftarrow (v,w) \in E_2
    \]
    or equivalently,
    \[
    (f(v),f(w)) \not\in E_1 \implies (v,w) \not\in E_2.
    \]
\end{definition}

In preparation for the quantization of the morphisms above, we demonstrate that a morphism of two classical graphs with (finite) vertex sets can be defined in terms of their annihilators in $\ell^\infty(V_i) \otimes_{eh} \ell^\infty(V_i) \cong \ell^\infty(V_i) \otimes \ell^\infty(V_i)$. (The isomorphism follows from the fact the respective operator spaces are finite dimensional.)

Let $G_1 = (V_1, E_1)$ and $G_2 = (V_2,E_2)$ be two classical finite graphs, and consider them as quantum graphs on $\mc M_1 := \ell^\infty(V_1)$ and $\mc M_2:= \ell^\infty(V_2)$ as in Example \ref{defn: qgraph morph}. Their respective annihilators $\Ann(\mc S_i)$ are left ideals in $\mc M_i \otimes_{eh} \mc M_i$ and a function $V_1 \overset{f}{\leftarrow} V_2$ on their vertices induces a unital $*$-homomorphism $\mc M_1 \overset{\theta}{\to} \mc M_2$ by sending $g \mapsto g \circ f$. Explicitly, for the indicator functions $\chi_u := \chi_{\{u\}} \in \ell^\infty(V_1)$ and $\chi_v := \chi_{\{v\}} \in \ell^\infty(V_2)$, 
\[
\theta(\chi_u) = \sum_{f(v) = u} \chi_v.
\]
If we consider each von Neumann algebra as a Hilbert space with inner product induced by $\<\chi_u|\chi_v\> = \delta_{uv}$, then we can consider the adjoint of $\theta$ as well:
\[
\theta^\dagger(\chi_v) = \chi_{f(v)}.
\]
%\begin{remark}
    %The map $\theta^\dagger: (\mc M_2)_* \to (\mc M_1)_*$ is more properly interpreted as the \textit{preadjoint} to $\theta$ in the infinite dimensional case. However, the preadjoint and adjoint of an operator coincide for finite dimensional $\mc M$ and we will not refer to it in later sections, so we use the usual notion and notation for adjoint here. See \cite[Section 8]{weavergraphsrelations} for further explanation.
%\end{remark}

\begin{definition} \protect{\cite[Proposition 6.12]{daws}}
    If $\mc H$ is a finite dimensional Hilbert space, the orthogonal complement of a quantum graph $\mc S \subseteq B(\mc H)$ is 
    \[
    \mc S^\perp := \{T \in B(\mc H) : Tr_{\mc H}(S^*T) = 0 \, \forall S \in \mc S\}.
    \]
\end{definition}
\begin{remark}
    If $G = (V,E)$ is a (finite) graph and $\mc S \subset B(\ell^2(V))$ is its quantization as in Example \ref{ex: classical-to-quantum-graph}, then 
    \[
    \mc S^\perp = \Span\{|v\>\<w| : (v,w) \notin E\}.
    \]
    In other words, the orthogonal complement corresponds precisely to the complement of the edges of a classical graph. Since the annihilator also corresponds to the complement of the classical graph (Example \ref{ex: eh-ann}), the following proposition may be unsurprising.
\end{remark}
\begin{proposition}\label{thm:classical-morphism}
    Let $G_1 = (V_1, E_1)$ and $G_2 = (V_2, E_2)$ be finite classical graphs, and let $\mc S_1 \subset B(\ell^2(V_1))$ and $\mc S_2 \subset B(\ell^2(V_2))$ be their quantizations. A function $V_1 \overset f \leftarrow V_2$ induces a morphism $G_1 \leftarrow G_2$ if and only if 
    \[
    \Ann(\mc S_1^\perp) \supseteq \theta^\dagger \otimes \theta^\dagger(\Ann(\mc S_2^\perp)).
    \]
\end{proposition}
\begin{proof}
    From the discussion above, the annihilator and orthogonal complement of $\mc S_k$ have concise forms:
    \[
    \Ann(\mc S_k) = \Span\{\chi_v \otimes \chi_w : (v,w) \notin E_k\} \qquad \text{and} \qquad
    \mc S_k^\perp = \Span\{|v\>\<w| : (v,w) \notin E_k\}.
    \]
    Hence
    \[
    \Ann(\mc S_2^\perp) = \Span\{\chi_v \otimes \chi_w : (v,w) \in E_2\}.
    \]
    Since $\theta^\dagger(\chi_v) = \chi_{f(v)}$, we have
    \[
    \theta^\dagger \otimes \theta^\dagger(\Ann(\mc S_2^\perp)) = \Span\{\chi_{f(v)}\otimes \chi_{f(w)}: (v,w) \in E_2\}.
    \]
    Thus the elementary tensors $\chi_{f(v)} \otimes \chi_{f(w)}$ annihilate the operator space $\mc S_1^\perp$ if and only if $(f(v), f(w)) \in E_1$.
\end{proof}

\begin{remark} 
The previous proposition is a direct way to characterize a classical graph morphism using $\ell^\infty(V)$ instead of $V$ and is the intuition behind other notions of morphisms in \cite[Proposition 5.3]{musto-reutter-verdon} and \cite[Definition 7.1]{daws}. Note, however, that it relies on the existence of an orthogonality relation in the ambient space of $B(\mc H)$ of the operator spaces and an inner product structure on the von Neumann algebra. The following characterization of morphisms does not use the orthogonal complement or inner product and corresponds to the fact that $V_1 \overset f \leftarrow V_2$ induces a morphism of classical graphs if and only if $(f(v), f(w)) \notin E_1 \implies (v,w) \not\in E_2$. This will become our definition of a morphism of quantum graphs in Definition \ref{defn: qgraph morph}.
\end{remark}

\begin{proposition}\label{thm:classical-morphism-1}
    The map $V_1 \overset f\leftarrow V_2$ induces a classical graph morphism $G_1 \leftarrow G_2$ if and only if $\theta \otimes \theta (\Ann(\mc S_1)) \subseteq \Ann(\mc S_2)$. 
\end{proposition}
\begin{proof}
    For $\chi_u \otimes \chi_t \in \Ann(\mc S_1)$ we have
    \[
    \theta \otimes \theta(\chi_u \otimes \chi_t) = \sum_{\substack{f(v) = u,\\ f(w) = t}} \chi_v \otimes \chi_w.
    \]
    If $f$ induces a morphism of graphs, none of the edges $(v,w)$ can be in $E_2$ else $(u,t) = (f(v),f(w)) \in E_1$. By linearity of $\theta \otimes \theta$, we have 
    \[
    \theta \otimes \theta (\Ann(\mc S_1)) \subseteq \Ann(\mc S_2).
    \]
    On the other hand, as a left ideal in $\ell^\infty(V_2) \otimes \ell^\infty(V_2)$, the annihilator $\Ann(\mc S_2)$ contains all such sums of elementary tensors if and only if it contains each elementary tensor. 
\end{proof}

\subsection{Morphisms of Quantum Graphs} \label{subsec: morphisms of qgraphs}

Note that in the classical case the von Neumann algebra $\ell^\infty (V)$ is its own commutant in the ambient space $B(\ell^2(V))$. Quantum graphs are bimodules over (the commutant of) a von Neumann algebra in the ambient space $B(\mc H)$ and hence the notion of a quantum graph morphism is more subtle. 

In the previous subsection we saw that any map $V_1 \overset f \leftarrow V_2$ between vertices of finite graphs induces a contravariant unital $*$-homomorphism $\ell^\infty(V_1) \overset{\theta}{\to} \ell^\infty(V_2)$. Currently we have the following maps:
\begin{center}
% https://tikzcd.yichuanshen.de/#N4Igdg9gJgpgziAXAbVABwnAlgFyxMJZARgBoAGAXVJADcBDAGwFcYkQA1AfWJAF9S6TLnyEUAZgrU6TVu24AmfoJAZseAkTLFpDFm0QgAwgApuxAJQACALxWAOvYC29HAAsAxkysBZHsqF1USJJHRo9OUNTRWs7Rxd3L0ZfLiUBQJFNFDIFXVkDEHjXT28-YgByANVhDTFkSVzw-PYixNLUyvTqoKySUnE8-RbnYqSrAGV-LrVMuskBpqHDRwhaGAAnRhgAMxx6dfWIAHdgVpLkyeI+B3s4ZgAjOBgcGABHG4TzidSqmdqichSRaRQojdz3bbAAAqEDQfF+NWCKAALECZEtQS8AB44YBuCBOGAEiDrNBuLBwJxweHTRG9VFhdEggCqRgACgienVUY0mQUAOoAPQAVDdsbirE4SWSKVSaSo-kjkICFnzhuKIcAAPLjeUZf4o-qDFns-jSGBQADm8CIoG2hycSDIIBwECQgLVhm2-hojHo9xgjDZdLEIHWWEtbhwVXtBKQChorqQkk9mLcz3oPpAfoDQZD7C2uxjDqQAFZE27EKjU451visznA8Guexw5Ho11Y47EAA2CtIADsifoWEY7BcaDgSb4lD4QA
\begin{tikzcd}
\mathbf{Set} & V_1  &  & V_2 \arrow[ll, "f"']     &  \\
             & \ell^\infty(V_1) =: \mathcal M_1 \arrow[rr, "\theta"] &  & \ell^\infty(V_2) =: \mathcal M_2        &                
\end{tikzcd}
\end{center}

If the objects $\ell^\infty(V_1)$ and $\ell^\infty(V_2)$ are generalized to be arbitrary von Neumann algebras in the quantization, we must require $\theta$ to be a normal unital $*$-homomorphism in order to stay in the category $\bf W^*$. 

\begin{center}
\begin{tikzcd}
\mathbf{Set}  & V_1  &  & V_2 \arrow[ll, "f"']               \\
\textbf{W*}  & \mathcal M_1 \arrow[rr, "\theta"]  &  & \mathcal M_2  
\end{tikzcd}
\end{center}

In order to define a morphism between quantum graphs using $\theta$, we exploit the correspondence between quantum graphs over $\mc M$ and their annihilators in $\mc M \otimes_{eh} \mc M$. The $\theta$ maps naturally induce maps $\theta \otimes \theta: \mc M_1 \otimes_{eh} \mc M_1 \to \mc M_2 \otimes_{eh} \mc M_2$ (Theorem \ref{thm: eh extension}). %The extended Haagerup tensor product of two operator algebras is generally not an operator algebra, but it certainly is a Banach algebra. Hence we take $\theta \otimes \theta$ to be in the category of Banach spaces with contractive maps ($\mathbf{BA_1}$). 
Being (norm) closed subspaces of $\mc M_k \otimes_{eh} \mc M_k$, the annihilators of quantum graphs are operator spaces, and the restriction of $\theta \otimes \theta$ to an annihilator is a morphism in the category $\mathbf{OpSp}$. In summary, we have the following maps:

\begin{center}
\begin{tikzcd}
\mathbf{Set}  & V_1  &  & V_2 \arrow[ll, "f"']               \\
\textbf{W*}  & \mathcal M_1 \arrow[rr, "\theta"]  &  & \mathcal M_2                    \\
%\textbf{BA$_1$} 
& \mathcal M_1 \otimes_{eh} \mathcal M_1 \arrow[rr, "\theta \otimes \theta"] &  & \mathcal M_2 \otimes_{eh} \mathcal M_2 \\
\textbf{OpSp} & \Ann(\mathcal S_1) \arrow[rr, "\theta \otimes \theta"] \arrow[u, hook]    &  & \Ann(\mathcal S_2). \arrow[u, hook]  
\end{tikzcd}
\end{center}

Inspired by the classical case in Theorem \ref{thm:classical-morphism}, we take the following definition as the morphism of quantum graphs.
\begin{definition} \label{defn: qgraph morph}
    Let $\mc S_1$ and $\mc S_2$ be quantum graphs over von Neumann algebras $\mc M_1$ and $\mc M_2$ respectively. We say a normal, unital $*$-homomorphism $\theta: \mc M_1 \to \mc M_2$ induces a quantum graph morphism $(\mc S_1, \mc M_1) \leftarrow (\mc S_2, \mc M_2)$ if $\theta \otimes \theta(\Ann(\mc S_1)) \subseteq \Ann(\mc S_2)$.
\end{definition}

\subsection{Equivalent Morphisms} \label{subsec: equiv-morphs}

As of this writing there seem to be two other notions of morphisms of quantum graphs from \cite{musto-reutter-verdon},  \cite{weavergraphsrelations}. We will show these are all equivalent to Definition \ref{defn: qgraph morph} up to the choices of the dimension of $\mc M$ and the interpretation of $\theta$ as a UCP map or ${\bf W^*}$ morphism.

\subsubsection{Equivalent Morphisms: Frobenius algebras} \label{subsubsec: frob-morphism}

In \cite{musto-reutter-verdon}, a quantum graph morphism between finite dimensional quantum graphs is defined via string diagrams with emphasis on characterizing a graph by its adjacency matrix. To avoid a tangent into graphical calculus, we will write their definition of graph morphism using dagger Frobenius algebra operations. The introduction of the notation is postponed to Appendix \ref{appendix: string}.

We reiterate that $^\dagger$ denotes the adjoint of a linear map between (finite dimensional) Hilbert spaces while $^*$ denotes the involution in the relevant $*$-algebra. In this Subsubsection our von Neumann algebras $\mc M$ will always be finite dimensional and are also implicitly Hilbert spaces (i.e., it is equipped with an inner product). Thus our notation $I_{\mc M}$ should be interpreted as the identity operator on $\mc M$ as a Hilbert space. The multiplicative identity of $\mc M$ as an algebra will be denoted $1_{\mc M}$. Also as a consequence of the finite dimensionality of $\mc M$, the $\eh$ tensor product reduces to 
\[
\mc M \eh \mc M = \mc M \otimes \mc M^{op},
\]
where $\mc M^{op}$ has the same $*$-vector space structure as $\mc M$ but opposite multiplication. Since $\mc M$ is finite dimensional, $\mc M \otimes \mc M^{op}$ is a von Neumann algebra with the multiplication
\[
(x_1 \otimes y_1) (x_2 \otimes y_2)  = x_1 x_2 \otimes y_2 y_1
\]
and involution\footnote{This is not the involutive structure on $\mc M \otimes \mc M^{op}$ that captures the notion of symmetric quantum graphs. See Subsection \ref{subsec: qgraph properties}.}
\[
(x \otimes y)^\diamond := x^* \otimes y^*.
\]

\begin{definition} \protect{\cite[Definition 5.1]{musto-reutter-verdon}} 
    Let $(\mc M, m, u)$\footnote{There is a bijection between finite dimensional $C^*$-algebras and SSFAs (see Appendix \ref{appendix: string}, so the map $m$ here coincides with the multiplication map $m$ defined in Corollary \ref{cor: sigma-mult} for the $\otimes_{\sigma h}$ tensor product.} be a special symmetric Frobenius dagger algebra (SSFA) (Definition \ref{defn: daggerfrob}). We say a map $A: \mc M \to \mc M$ is a \textit{quantum adjacency operator} if $A= A^\dagger$ and the two following equations hold:
    \begin{align}
    A &= m \circ (A \otimes A) \circ m^\dagger \label{eqn: schur} \\
    A &= (I_{\mc M} \otimes u^\dagger) \circ (I_{\mc M} \otimes A \otimes I_{\mc M})\circ (u \otimes I_{\mc M}) \label{eqn: snake}.
    \end{align}
\end{definition}

\begin{remark}
The assumption that $A = A^\dagger$ corresponds to the assumption that the graph is undirected \cite[Proposition 2.6]{daws} or that $A$ corresponds to a symmetric quantum relation in the language of \cite[Definition 2.4]{weaverqrelations}. Proposition \ref{thm: mrv-ann} still holds with or without this assumption. Equation \ref{eqn: schur} above is known as ``Schur idempotence'' since in the classical case $m \circ (A \otimes A) \circ m^\dagger$ implements the Schur product of the adjacency operator $A$ with itself with respect to an implicit basis. This requirement is equivalent to being a quantum graph according to our definition in Definition \ref{defn: qgraph}. Equation \ref{eqn: snake} is the ``snake equation.'' Using the terminology of \cite[Definition 2.4]{daws}, this is the generalization of the adjacency matrix of a classical graph being self-transpose.
\end{remark}

The authors of \cite{musto-reutter-verdon} do not define morphisms of quantum graphs via adjacency operators but rather their \textit{projectors}.

\begin{definition} \protect{\cite[Section 5.1]{musto-reutter-verdon}} \label{defn: projector}
    The \textit{projector} of a quantum adjacency operator $A$ is the map $P :\mc M \otimes_{eh} \mc M \to \mc M \otimes_{eh} \mc M$ defined by
    \begin{align*}
    P &:= (m \otimes I_{\mc M})\circ (I_{\mc M} \otimes A \otimes I_{\mc M})\circ (I_{\mc M} \otimes m^\dagger) \\
    &= (I_{\mc M} \otimes m)\circ (I_{\mc M} \otimes A \otimes I_{\mc M})\circ (m^\dagger \otimes I_{\mc M}).
    \end{align*}
\end{definition}

To translate from this language of projectors to annihilators, we must take an intermediate step from the projectors $P\in \text{End}(\mc M \eh \mc M)$ defined above to projections $p \in \mc M \eh \mc M$. We take the convention that projections are orthogonal; that is, $p^2 = p = p^\diamond$. In the remainder of this subsection, if $f: V \to W$ is a function, we will denote the evaluation of $f$ at a point $v \in V$ by $ev(f,v)$. This is to (hopefully) reduce the cognitive load in interpreting the string of symbols in $P(1_{\mc M} \otimes 1_{\mc M})$, for example. (This string should be interpreted as $ev(P, 1_{\mc M} \otimes 1_{\mc M})$ but might be confused for $P \circ (I_{\mc M} \otimes I_{\mc M})$.)

\begin{lemma} \protect{\cite[Remark 7.3]{musto-reutter-verdon}} \label{lem: proj-proj}
    One obtains a bijection between the projections $p \in \mc M \eh \mc M$ and projectors $P: \mc M \eh \mc M \to \mc M \eh \mc M$ from Definition \ref{defn: projector} via evaluation:
    \[
    p = ev(P,1_{\mc M} \otimes 1_{\mc M}).
    \]
    This process is invertible; see Remark \ref{rmk: left-right-convention}.
\end{lemma}

\begin{remark} \label{rmk: left-right-convention}
In \cite[Remark 7.3]{musto-reutter-verdon}, a projector $P$ is implemented by \textit{left} multiplication by a projection $p \in \mc M^{op} \otimes \mc M$ (which is the correct category-theoretic convention). However, we take the convention that that $p \in \mc M \otimes \mc M^{op}$ to align with the operator space theoretic convention for $\otimes_{eh}$, so $P$ is implemented by \textit{right} multiplication by $p$. 
\end{remark}

From \cite[Corollary 5.13]{daws}, for a von Neumann algebra $\mc M \subset B(\mc H)$ there is a bijection between projections in $\mc M \eh \mc M$ and quantum graphs $\mc S$ on $\mc M$. Namely, for every a projection $p \in \mc M \eh \mc M$ we know $\Phi_p$ is a (completely) bounded operator on $B(\mc H)$, and the image of $\Phi_p$ is a quantum graph $\mc S$ on $\mc M$. Since
\begin{align*}
    \text{Ker}(p) \oplus \text{Im}(p) = B(\mc H) 
\end{align*}
we say $p$ is a projection onto $\mc S = \text{Im}(p)$ along $\mc S^\perp = \text{Ker}(p)$ (\cite[Definition 5.11]{daws}). The relationship between $p$ and the annihilator of $\mc S$ is given by the following corollary.

\begin{corollary} \label{cor: proj-ann}
    Let $p \in \mc M \eh \mc M$ be a projection onto $\mc S$ along $\mc S^\perp$. Then
    \begin{align*}
    \Ann(\mc S) &= (\mc M \eh \mc M)(1_{\mc M} \otimes 1_{\mc M} -p) &&\text{and} & \Ann(\mc S^\perp) &= (\mc M \eh \mc M)p.
    \end{align*}
\end{corollary}

\begin{definition}[\protect{\cite[Definition 5.4]{musto-reutter-verdon}}] \label{defn: string morph}
    Let $A_1: \mc M_1 \to \mc M_1$ and $A_2: \mc M_2 \to \mc M_2$ be quantum graph adjacency operators. A $*$-cohomomorphism $\theta^\dagger: \mc M_2 \to \mc M_1$ (Definition \ref{defn: *-cohom}) induces a quantum graph homomorphism from $A_2$ to $A_1$ if 
    \[
    (\theta^\dagger \otimes \theta^\dagger) \circ P_{2} = P_{1} \circ (\theta^\dagger \otimes \theta^\dagger) \circ P_{2}
    \]
    where each $P_{k}$ is the projector of the quantum adjacency operator $A_k$.
\end{definition}

\begin{proposition} \label{thm: mrv-ann}
    If $A_1: \mc M_1 \to \mc M_1$ and $A_2: \mc M_2 \to \mc M_2$ are quantum graph adjacency operators on finite dimensional von Neumann algebras $\mc M_k$, the quantum graph morphisms in Definition \ref{defn: qgraph morph} and Definition \ref{defn: string morph} are equivalent. 
\end{proposition}

\begin{proof}
    We first rewrite the condition in Definition \ref{defn: string morph} terms of $\theta$. By definition, projectors are self-adjoint \cite[Section 1.4]{musto-reutter-verdon} and projectors defined as in Definition \ref{defn: projector} are indeed projectors in the broad sense. Therefore
    \begin{align*}
        (\theta^\dagger \otimes \theta^\dagger) \circ P_{2} = P_{1} \circ (\theta^\dagger \otimes \theta^\dagger) \circ P_{2} \iff P_{2} \circ (\theta \otimes \theta) &= P_{2} \circ (\theta \otimes \theta) \circ P_{1}.
    \end{align*}
    We now show the latter condition is equivalent to 
    \begin{align} \label{mrv-ann: eqn1}
      p_2 = ev(\theta \otimes \theta, p_1)p_2  
    \end{align}
    where $p_k \in \mc M_k \eh \mc M_k$ are projections from Remark \ref{rmk: left-right-convention}. One direction is obtained from by evaluating both sides of
    \[
    P_{2} \circ (\theta \otimes \theta) = P_{2} \circ (\theta \otimes \theta) \circ P_{1}
    \]
    on $1_{\mc M_1} \otimes 1_{\mc M_1}$. The l.h.s. is reduced to $p_2$ since $\theta$ is unital. The r.h.s. is reduced to $ev(\theta \otimes \theta, p_1)p_2$ since $P_k$ is implemented by right multiplication by $p_k$.
    
    On the other hand, if Equation \ref{mrv-ann: eqn1} holds then for any $a \otimes b \in \mc M_1 \otimes \mc M_1$
    \begin{align*}
        P_{2} \circ ev(\theta \otimes \theta, a \otimes b) &= ev(\theta \otimes \theta, a \otimes b)  p_{2} & \text{Remark } \ref{rmk: left-right-convention} \\
        &= ev(\theta \otimes \theta, a \otimes b) ev(\theta \otimes \theta, p_1)p_2 &\text{Equation } \ref{mrv-ann: eqn1} \\
        &= ev(\theta \otimes \theta, (a \otimes b)p_1)p_2 &\text{homomorphism} \\
        &= P_2 \circ ev(\theta \otimes \theta, (a \otimes b)p_1) & \text{Remark } \ref{rmk: left-right-convention} \\
        &= P_2 \circ ev(\theta \otimes \theta \circ P_1, a \otimes b) & \text{Remark } \ref{rmk: left-right-convention}
    \end{align*} 
    that is, $P_2 \circ (\theta \otimes \theta) = P_2 \circ (\theta \otimes \theta) \circ P_1$. Now since $\theta$ is a $*$-homomorphism, $ev(\theta \otimes \theta, p_1)$ is a projection and Equation \ref{mrv-ann: eqn1}
    \[
    p_2 = ev(\theta \otimes \theta, p_1) p_2
    \]
    is equivalent to 
    \[
    ev(\theta \otimes \theta, 1_{\mc M_1} \otimes 1_{\mc M_1} -p_1) = ev(\theta \otimes \theta, 1_{\mc M_1} \otimes 1_{\mc M_1} -p_1) (1_{\mc M_2} \otimes 1_{\mc M_2} -p_2).
    \]
    From Corollary \ref{cor: proj-ann}
    \[
    \Ann(\mc S_1) = (\mc M_1 \eh \mc M_1) (1-p_1) \qquad \Ann(\mc S_2) = (\mc M_2 \eh \mc M_2) (1-p_2).
    \]
    Finally, we use the fact that $\theta \otimes \theta$ is a homomorphism to conclude
    \[
    (\theta\otimes \theta)\Ann(\mc S_1) \subseteq \Ann(\mc S_2).
    \]
\end{proof}

\begin{remark}
    The morphism of quantum graphs in Defintion \ref{defn: string morph} would be a ``classical morphism of quantum graphs'' in the language of \cite[Remark 5.5]{musto-reutter-verdon}. The authors of \cite{musto-reutter-verdon} are primarily interested in ``quantum morphisms of quantum graphs.'' See Question \ref{question: t-morphisms}.
\end{remark}

\subsubsection{Equivalent Morphisms: CP Morphisms}

Another variation of morphism was introduced by Stahlke \cite[Defintion 7]{stahlke} using CPTP morphisms between traceless quantum graphs. Weaver extended this notion under the label \textit{CP morphism} in \cite[Definition 8.5]{weavergraphsrelations} for CP maps between matrix algebras and showed the notion was independent of the choice of Kraus operator associated to the CP map. Daws further extended the definition to normal CP maps between von Neumann algebras of arbitrary dimension in \cite[Section 7]{daws}. For the interested reader, the same section also contains a construction of a ``Kraus form" for normal CP maps between arbitrary von Neumann algebras, not just finite dimensional ones. Note that contravariant to our convention, Stahlke and Weaver in use the preadjoint $^\dagger \theta: (\mc M_2)_* \to (\mc M_1)_*$ as the primary CP map to induce the quantum graph morphism.

\begin{definition} [\protect{\cite[Section 7]{daws}}] \label{defn: cp morphism}
    Let $\mc M_k \subseteq B(\mc H_k)$ be von Neumann algebras (of arbitrary dimension) and suppose $\mc S_k \subseteq B(\mc H_k)$ is a quantum graph on $\mc M_k$. Suppose $\mc M_1 \overset{\theta}{\to} \mc M_2$ is a normal UCP map with Kraus form $\theta(x) = \sum_{i \in I} K_i^\dagger x K_i$. We say $\theta$ induces a CP morphism from $\mc S_2$ to $\mc S_1$ if 
    \[
    \mc S_1 \supseteq \overline{\Span}^{w^*}\{K_i \mc S_2 K_j^\dagger : i,j \in I\}.
    \]
\end{definition}

\begin{example}
    To motivate the above definition we turn to the classical, finite case (i.e., when the $\mathbf{W^*}$ morphism $\ell^\infty(V_1) \overset{\theta}{\to} \ell^\infty(V_2)$ is implemented by a function $V_1 \overset{f}{\leftarrow} V_2$). We may choose the Kraus operators of $\theta$ to have the form $K_v :=|f(v)\>\<v|$ for $v \in V_2$. If the edge $|v\>\<w|$ is in $\mc S_2$, then 
    \begin{align*}
    |f(v)\>\<f(w)| &=  |f(v)\>\<v| v\>\<w|w\>\<f(w)| \\
    &= K_v|v\>\<w|K_w^\dagger \in \overline{\Span}^{w^*}\{|f(v)\>\<v| y|w\>\<f(w)| : v,w \in V_2, y \in \mc S_2\}.
    \end{align*}
    Thus $V_1 \overset{f}{\leftarrow} V_2$ induces a classical graph morphism if and only if 
    \begin{align*}
    \mc S_1 &\supseteq \overline{\Span}^{w^*}\{|f(v)\>\<v| y|w\>\<f(w)| : v,w \in V_2, \, y \in \mc S_2\} \\
    &= \overline{\Span}^{w^*}\{ K_v \mc S_2 K_w^\dagger: v,w \in V_2 \}.
    \end{align*}
\end{example}

In preparation for the following equivalence, we record a lemma from \cite{weavergraphsrelations}. The lemma is stated for finite dimensional $\mc H_1, \mc H_2$, but the proof lift to Hilbert spaces of arbitrary dimension. 

\begin{lemma}[\protect{\cite[Lemma 8.3]{weavergraphsrelations}}] \label{lemma: positivity-weaver}
Using the same notation as in Definition \ref{defn: cp morphism}, let $A,B \in B(\mc H_1)$ be positive, $T \in B(\mc H_2)$, and $K_i, K_j \in B(\mc H_2,\mc H_1)$. Then
            \[
            (K_i^\dagger A K_i )T (K_j^\dagger B K_j) =0 \quad \iff \quad (AK_i )T (K_j^\dagger B) =0.
            \]
\end{lemma}
\begin{proof}
    One direction is obvious, so we only address the $\Rightarrow$ implication. Let $D = T (K_j^\dagger B K_j)$. From our assumption we have 
    \[
    (K_i^\dagger A K_i )\underbrace{T (K_j^\dagger B K_j)}_{D} =0 \implies D^\dagger(K_i^\dagger AK_i) D = 0.
    \]
    Since we assumed $A$ was positive, we can take its square root and
    \[
    0 = D^\dagger(K_i^\dagger AK_i) D = (D^\dagger K_i^\dagger \sqrt{A} )(\sqrt{A} K_i D ) = (\sqrt{A} K_i D )^\dagger (\sqrt{A} K_i D ).
    \]
    In any $C^*$-algebra $a^* a = 0$ if and only if $a =0$ so we have that $\sqrt{A} K_i D =0$. Multiplying from the left by $\sqrt{A}$ we get 
    \[
    0 = A K_i D = A K_i T (K_j^\dagger B K_j).
    \]
    Performing an analogous procedure for $D = AK_i T$ yields the desired desired result.
\end{proof}

A version of the following lemma is also proven in \cite[Lemma 8.3]{weavergraphsrelations} for finite dimensions. We will use the same ideas, but we need extra topological considerations in the infinite dimensional case.

\begin{lemma} \label{lem: positivity-weaver-2}
    Let $\mc M \subset B(\mc H)$ be a von Neumann algebra and suppose
        \[
        \sum_{i\in I} x_i, \sum_{j \in J} y_j \in B(\ell^2) \bar{\otimes} \mc M
        \]
        are $\sigma(B(\ell^2)_* \hat \otimes \mc M_*)$ convergent sums such that each $x_i, y_j$ is positive. For each $T \in B(\ell^2) \overline{\otimes} B(\mc H)$,
        \[
        \bigg(\sum_{i\in I} x_i\bigg) T \bigg(\sum_{j \in J} y_j\bigg) = 0 \quad \iff \quad x_i T y_j = 0 \quad  \forall i,j.
        \]
\end{lemma}
\begin{proof}
    Since $B(\ell^2) \overline{\otimes} B(\mc H)$ is a dual Banach algebra, multiplication is separately $w^*$ continuous. We can therefore write
    \[
    \bigg(\sum_{i \in I} x_i\bigg) T \bigg(\sum_{j \in J} y_j\bigg) = \lim_{G \Subset J}\lim_{F \Subset I} \sum_{i \in F} x_i T \sum_{j \in J}y_j.
    \]
    Thus if we assume $x_i T y_j = 0$ for all $i,j$ then we have one direction of the statement. For the other direction, denote $T_Y := T\sum_j y_j$. The above equation reduces to
    \[
    \bigg(\lim_{F \Subset I} \sum_{i \in F} x_i\bigg) T_Y= 0,
    \]
    where the limit converges in the $\sigma(B(\mc H)_*)$ topology. Again, using the continuity of multiplication,
    \[
     T_Y^\dagger\bigg( \lim_{F \Subset I} \sum_{i \in F} x_i \bigg)T_Y = 0 \implies  \lim_{F \Subset I} \sum_{i \in F} T_Y^\dagger x_iT_Y = 0.
    \]
    Since the latter limit is an increasing net of positive operators, for all $i$ we have
    \begin{align} \label{posi-weaver-2-eqn2}
    T_Y^\dagger x_iT_Y = 0 \implies x_i^{1/2}T_Y = 0 \implies x_i T_Y = 0.
    \end{align}
    Similarly, Equation \ref{posi-weaver-2-eqn2} implies for all $i$
    \begin{align} \label{posi-weaver-2-eqn3}
    x_iT\bigg( \lim_{F \Subset J} \sum_{j \in F} y_j \bigg) = 0 \implies x_iTy_j = 0.
    \end{align}
\end{proof}

\begin{remark}
Recall that we chose our maps $\theta: \mc M_1 \to \mc M_2$ to be normal unital $*$-homomorphisms. In the proof below, we only need that $\theta$ is a normal CP map, so perhaps we could implemented a morphism of quantum graphs by a normal (U)CP map instead of a normal $*$-homomorphism. See Remark \ref{rmk: w*/ucp} for a comment about this choice.
\end{remark}

\begin{proposition} \label{prop: cp-morph-qmorph}
    Let $\mc M_k \subseteq B(\mc H_k)$ be von Neumann algebras and suppose $\mc S_k \subseteq B(\mc H_k)$ are quantum graphs over the $\mc M_k$ for $k =1,2$. Suppose $\theta: \mc M_1 \to \mc M_2$ is a normal $*$-homomorphism. The map $\theta(x) = \sum_i K_i^\dagger x K_i$ with Kraus operators $K_i: \mc H_2 \to \mc H_1$ implements a CP morphism (Definition \ref{defn: cp morphism}) if and only if $\theta$ implements a quantum graph morphism (Definition \ref{defn: qgraph morph}). 
\end{proposition}
\begin{proof}
    $\underline{\implies}$: By definition, the map $\theta$ implements a CP morphism if 
    \[
    \mc S_1 \supseteq \overline{\Span}^{w^*}\{K_i T K_j^\dagger : T \in \mc S_2, i,j \in I\}.
    \]
    Since every weak representation $\sum_{k \in K} a_k \otimes b_k \in \Ann(\mc S_1) \subseteq \mc M_1 \eh \mc M_1$ implements a $w^*$-continuous map on $B(\mc H_1)$ via $\Phi$, the above condition holds if and only if for all $T \in \mc S_2$, $i,j \in I$, and $\sum_{k \in K} a_k \otimes b_k \in \Ann(\mc S_1)) \subseteq \mc M_1 \eh \mc M_1$ we have
    \begin{align} \label{eqn: 1}
    \Phi_{\sum_k a_k \otimes b_k} (K_i T K_j^\dagger) = \sum_k  a_k K_iTK_j^\dagger  b_k = 0.
    \end{align}
    What follows are some miracles about the $\eh$ tensor product. The tensor map
    $\theta \otimes \theta$ extends to a unique, well-defined map on the respective extended Haagerup tensor products (\cite[Proposition 3.7]{blecher-smith-w*h}), so in particular the double sum $\sum_{i,j}$ below converges in the $\sigma ({\mc M_2}_* \eh {\mc M_2}_*)$ topology. 
    \begin{align*}
    \sum_{i,j} \sum_k K_i^\dagger a_kK_i \otimes K_j^\dagger b_kK_j = (\theta \otimes \theta) \sum_k a_k \otimes b_k = \sum_k \theta(a_k) \otimes \theta(b_k).
    \end{align*}
    Since $\Phi$ is a completely isometric isomorphism between $\mc M_1 \eh \mc M_1$ and $CB^\sigma_{\mc M_1', \mc M_1'}(B(\mc H_1))$ \cite[Theorem 4.2]{blecher-smith-w*h}, 
    \[
    \Phi_{\sum_k \theta(a_k) \otimes \theta(b_k)}(T) =\sum_k \theta(a_k)T \theta(b_k)= \sum_{i,j} \sum_k (K_i^\dagger a_kK_i )T (K_j^\dagger b_kK_j) .
    \]
    The assumption that $\sum_k a_k K_iTK_j^\dagger b_k = 0$ yields an equality 
    \[
    \Phi_{\sum_k \theta(a_k) \otimes \theta(b_k)}(T) = 0,
    \]
    showing that $\theta \otimes \theta(\Ann(\mc S_1)) \subseteq \Ann(\mc S_2)$. \\\\
    $\underline{\Longleftarrow}:$ The other direction portion of this proof follows \cite[Theorem 8.4]{weavergraphsrelations}. We assume $\theta \otimes \theta(\Ann(\mc S_1)) \subseteq \Ann(\mc S_2)$. From \cite[Theorem 3.3]{gp-obs}, any quantum graph $\mc S_1$ can be uniquely identified by a collection of pairs of projections in $B(\ell^2) \overline{\otimes} \mc M_1$ in the following way
    \[
    \mc R_1 := \{(P,Q) \in \mc P(B(\ell^2) \overline{\otimes} \mc M_1)^2 : P \odot Q \notin B(\ell^2) \overline{\otimes} \Ann(\mc S_1)\}
    \]
    (this is what Weaver calls an \textit{intrinsic quantum relation} \cite[Definition 2.24]{weaverqrelations}). Of course, we can also identify $\mc S_1$ by the complement of $\mc R_1$ in the cartesian product of projections:
    \[
    \mc R_1^c := \{(P,Q) \in \mc P(B(\ell^2) \overline{\otimes} \mc M_1)^2 : P \odot Q \in B(\ell^2) \overline{\otimes} \Ann(\mc S_1)\}.
    \]
    Every $P \odot Q$ can be written as an infinite matrix with entries in $\mc M_1 \eh \mc M_1$. Explicitly, if $P = [p_{st}]_{s,t}$ and $Q = [q_{st}]_{s,t}$ then
    \[
    P \odot Q = \bigg[ \sum_{} p_{su} \otimes q_{ut} \bigg]_{s,t}, \quad \sum_u p_{su} \otimes q_{ut} \in \Ann(\mc S_1) \subseteq  \mc M_1 \eh \mc M_1.
    \]
    For $T \in \mc S_2$ and $i,j \in I$, 
    \begin{align*}
     K_i T K_j^\dagger \in \mc S_1 &\iff \forall (P,Q) \in \mc R_1^c \, \, \forall s,t  \, \,\sum_u p_{su} K_i T K_j^\dagger q_{ut} = 0 \\
     &\iff  \forall (P,Q) \in \mc R_1^c  \, \, P(I_{\ell^2} \otimes K_iTK_j^\dagger) Q  = 0.
    \end{align*}
    We now use our assumption to prove the last condition. Suppose $(P,Q) \in \mc R_1^c$. Since $\theta$ is normal, $(I_{\ell^2} \otimes \theta)P$ and $(I_{\ell^2} \otimes \theta)Q$ are well-defined elements of $B(\ell^2) \overline{\otimes} \mc M_2$. %well-defined because the spatial tensor product is the dual of B(\ell^2)_* \hat \otimes \mc M_1*. Convergence in spatial tensor product is equivalent to ``entrywise w*'' because finite rank operators are norm dense in B(\ell^2)_* and those can be decomposed to entry-corners. Continuity of 1 \otimes \theta is equivalent to entrywise w^* cty then.
    Entrywise we can write them as
    \[
    (I_{\ell^2} \otimes \theta)P = [\theta(p_{st})]_{s,t} \qquad (I_{\ell^2} \otimes \theta)Q = [\theta(q_{st})]_{s,t}
    \]
    so their multiplicative product is
    \[
    (I_{\ell^2} \otimes \theta)P \odot(I_{\ell^2} \otimes \theta)Q =  \bigg[ \sum_u \theta(p_{su}) \otimes \theta(q_{ut}) \bigg]_{s,t}.
    \]
    By hypothesis (and the functorality of the $\eh$ tensor product) we conclude that 
    \[
    (I_{\ell^2} \otimes \theta)P \odot(I_{\ell^2} \otimes \theta)Q \in B(\ell^2) \overline{\otimes} \Ann(\mc S_2)
    \]
    which by \cite[Theorem 3.3]{gp-obs} is equivalent to
    \begin{align} \label{eqn: blah}
    [(I_{\ell^2} \otimes \theta)P] (I_{\ell^2} \otimes T) [(I_{\ell^2} \otimes \theta)Q] = 0 \qquad \forall T \in \mc S_2.
    \end{align}
    A Kraus form of $I_{\ell^2} \otimes \theta$ is
    \[
    I_{\ell^2} \otimes \theta(\cdot) = \sum_i (I_{\ell^2} \otimes K_i^\dagger ) \cdot (I_{\ell^2} \otimes K_i ).
    \]
    %Same justification as above.
    Continuing Equation \ref{eqn: blah} and using the Kraus form of $1 \otimes \theta$,
    \[
    \bigg[\sum_i (I_{\ell^2} \otimes K_i^\dagger ) P(I_{\ell^2} \otimes K_i ) \bigg] (I_{\ell^2} \otimes T) \bigg[\sum_j (I_{\ell^2} \otimes K_j^\dagger ) Q(I_{\ell^2} \otimes K_j ) \bigg] = 0.
    \]
    From Lemma \ref{lem: positivity-weaver-2}, the previous equation holds if and only if for all $i,j \in I$
    \[
    \bigg[(I_{\ell^2} \otimes K_i^\dagger ) P(I_{\ell^2} \otimes K_i )\bigg] (I_{\ell^2} \otimes T) \bigg[(I_{\ell^2} \otimes K_j^\dagger ) Q(I_{\ell^2} \otimes K_j)\bigg]  = 0
    \]
    which from Lemma \ref{lemma: positivity-weaver} holds if and only if
    \[
    P(I_{\ell^2} \otimes K_iTK_j^\dagger) Q  = 0.
    \]
    Recalling that $(P,Q)$ was an arbitrary pair in $\mc R_1^c$, we conclude that $K_iTK_j^\dagger \in \mc S_1$. 
\end{proof}

\section{Quantum Graphs} \label{sec: limits-qgraphs}
Recall that in Subsubsection \ref{par: annihilator characterization} we saw that we have a representation-free way to express quantum graphs. We repeat it here for convenience.

\begin{theorem}\protect{\cite[Theorem 3.3]{gp-obs}} 
    Let $\mc M \subset B(\mc H)$ be a von Neumann algebra. There is a bijection between $\sigma(\mc M_* \eh \mc M_*)$-closed left ideals in $\mc M \eh \mc M$ and quantum graphs on $\mc M$. Namely, each quantum graph corresponds uniquely to its annihilator in $\mc M \eh \mc M$:
    \[
    \Ann(\mc S) = \{z \in \mc M \otimes_{eh} \mc M : \Phi_z(T) = 0, \, T \in \mc S\}
    \]
    where $\Phi$ is the map in Definition \ref{defn: Phi-map}.
\end{theorem}

\subsection{Properties of Quantum Graphs} \label{subsec: qgraph properties}

In this subsection we establish some properties of quantum graphs (namely, reflexivity, symmetry, and connectivity) that are easily detected in their annihilators. The following lemma is surely known (or obvious) to experts, but we explicitly record it here to make sense of the involution $\ddagger$ on $\eh$. (Recall $(x \otimes y)^\ddagger := y^* \otimes x^*$ on simple tensors.)

\begin{lemma} \label{lem: inv-cty}
    The involution on $\mc M \eh \mc M$ extending 
    \[
    (x \otimes y)^\ddagger  = y^* \otimes x^*
    \]
    is $\sigma(\mc M_* \otimes_{h} \mc M_*)$-continuous and $\sigma(\mc M_* \otimes_{eh} \mc M_*)$-continuous.
\end{lemma}
\begin{proof}
    We will represent $\mc M \subset B(\mc H)$ and only prove $\sigma(\mc M_* \otimes_{eh} \mc M_*)$-continuity since $\sigma(\mc M_* \otimes_{h} \mc M_*)$-continuity follows the same proof by replacing the operator $B$ below by a compact operator. By definition, a net $(z_\lambda)_{\lambda \in \Lambda}$ in $\mc M \eh \mc M$ converges to $z \in \mc M \eh \mc M$ in the $\sigma(\mc M_* \eh \mc M_*)$-topology if and only if for every $\phi \in \mc M_* \eh \mc M_*$ we have
    \[
    \phi(z_\lambda - z) \to 0.
    \]
    From \cite[Lemma 2.3]{gp-obs}, every linear functional $\phi \in \mc M_* \otimes_{eh} \mc M_*$ there exists a unique $B \in B(\ell^2 \otimes_2 \mc H)$ and a trace class operator $T \in B(\ell^2 \otimes_2 \mc H)_*$ such that for $z \in \mc M \eh \mc M$,
    \[
    \phi(z) = Tr_{\ell^2 \otimes_2 \mc H} (T^*(I_{\ell^2} \otimes \Phi(z))(B)).
    \]
    Since $Tr_{\ell^2 \otimes_2 \mc H}(S^*) = \overline{Tr_{\ell^2 \otimes_2 \mc H}(S)}$,
    \[
    \Tr(T^*(Id \otimes \Phi(z_\lambda - z))(B))\to 0 \iff \Tr(T(Id \otimes \Phi(z_\lambda^\ddagger - z^\ddagger))(B^*))\to 0.
    \]
    Since the trace class operators $B(\ell^2 \otimes_2 \mc H)$ and $B(\ell^2 \otimes_2 \mc H)$ are closed under involution, the latter statement suffices to show that $z_\lambda^\ddagger \to z^\ddagger$ in the $\sigma(\mc M_* \eh \mc M_*)$ topology.
\end{proof}

\cite[Corollary 5.13]{daws} provides the finite dimensional version of the proposition below, but note that the correspondence that the author of \cite{daws} uses between quantum graphs and projections is order-preserving while the correspondence between quantum graphs and  annihilators is order-reversing. To be precise, the projection $e$ in \cite[Corollary 5.13]{daws} is related to the annihilator by way of the equality
\[
\Ann(\mc S) = (\mc M \otimes_{eh} \mc M)(1-e).
\]

\begin{proposition} \label{prop: qr-to-ann}
    Let $\mc S$ be a quantum graph on a von Neumann algebra $\mc M \subseteq B(\mc H)$, and let $\Ann(\mc S)$ be its annihilator in $\mc M \otimes_{eh} \mc M$. The following properties of $\mc S$ are reflected in properties of its annihilator:
    \begin{enumerate}
        \item $\mc S$ is reflexive quantum graph if and only if for every $z \in \Ann(\mc S)$ we have $m(z) = 0$.
        \item $\mc S$ is symmetric quantum graph if and only if $\Ann(\mc S)$ is closed under the involution $(x \otimes y)^\ddagger = y^* \otimes x^*$.
    \end{enumerate}
\end{proposition}
\begin{proof}
    For (1): Since $\mc M \eh \mc M$ embeds completely isometrically into $\mc M \otimes_{\sigma h} \mc M$ (Remark \ref{rmk: predual-tops}), the multiplication map $m: \mc M \otimes_{\sigma h} \mc M \to \mc M$ from Corollary \ref{cor: sigma-mult} is also defined on $\mc M \eh \mc M$. We have that $z \in \Ann(\mc M')$ if and only if for all $T \in \mc M'$
    \[
    0 = \Phi_z(T) = \sum_i x_iTy_i = T \sum_i x_i 1 y_i = T\Phi_z(1).
    \]
    Since $1 \in \mc M'$, we have
    \[
    \sum_i x_i y_i = \Phi_z(1) = 0 \iff z \in \Ann(\mc M').
    \]
    For (2): From the definition of a symmetric quantum graph, $T \in \mc S$ if and only if $z = \sum_i x_i \otimes y_i \in \Ann(\mc S)$ annihilates $T^*$ as well:
    \[
    \Phi_z(T^*) = \sum_i x_i T^* y_i = 0
    \]
    where the partial sums are uniformly bounded and converging in the WOT topology \cite[Theorem 2.2]{blecher-smith-w*h}. Hence the last equality holds if and only if 
    \[
    \sum_i y_i^* T x_i^* = 0.
    \]
    That is, $\sum_i y_i^*\otimes x_i^* \in \Ann(\mc S)$. 
\end{proof}

\begin{remark}
    The astute reader may notice we are missing a characterization of one of the key properties of quantum graphs established in Definition \ref{defn: qgraph properties}: transitivity. This has proven to be a particularly thorny problem. See Question \ref{question: transitivity}.
\end{remark}

A recent paper \cite{connected-qgraphs} has established a unifying notion of connectivity for quantum graphs in the case where $\mc M$ is finite dimensional in terms of adjacency operators. To avoid introducing these operators, we instead present two of the equivalent notions to point a way to generalize connectedness to the infinite dimensional case.

\begin{theorem} \protect{\cite[Theorem 3.4]{connected-qgraphs}}
    A quantum graph $\mc S$ on $\mc M \subset B(\mc H)$ for a finite dimensional von Neumann algebra $\mc M$ is \textit{strongly connected} if and only if there does not exist a graph homomorphism in the sense of \cite[Proposition 5.3]{musto-reutter-verdon} from $\mc S$ to $T_2$, where $T_2$ is the classical totally disconnected graph on two vertices. 
\end{theorem}

\begin{proposition} \protect{\cite[Proposition 3.7]{connected-qgraphs}}
    A quantum graph $\mc S$ on $\mc M \subset B(\mc H)$ is \textit{strongly connected} for finite dimensional $\mc M$ if and only if the algebra generated by $\mc S$ is equal to $B(\mc H)$. 
\end{proposition}

The following definition is equivalent to the definition of connectedness given by \cite[Definition 3.2]{jun-alg-connect} in the finite dimensional case for a symmetric, reflexive quantum graph.

\begin{definition} \label{defn: connected}
    A quantum graph $\mc S$ on $\mc M \subset B(\mc H)$ is \textit{strongly connected} if there does not exist a non-trivial projection $p \in \mc M$ such that 
    \[
    (1-p) \otimes p \in \Ann(\mc S).
    \]
\end{definition}

\begin{proposition}
    A quantum graph $\mc S$ on $\mc M \subset B(\mc H)$ is \textit{strongly connected} if and only if for every $\mathbf{W^*}$ morphism $\ell^\infty(\{1,2\}) \overset{\theta}{\to} \mc M$ such that 
     \[
     (\theta \otimes \theta )(\Ann(T_2)) \subseteq \Ann(\mc S)
     \]
    either $\theta(\chi_1) = 0$ or $\theta(\chi_2) = 0$.
\end{proposition}
\begin{proof}
    A quantum graph is not strongly connected if and only if there exists a $*$-homomorphism $\ell^\infty(\{1,2\}) \overset{\theta}{\to} \mc M$ such that $1 - p := \theta(\chi_1)$ and $p = \theta(\chi_2)$ are both non-zero and complementary projections.
\end{proof}

\begin{remark}
    In the case where $\mc S$ is a quantum graph on $\mc M$, each invariant subspace of the von Neumann algebra $W^*(\mc S)$ generated by $\mc S$ is given by a projection $p$ in the commutant $W^*(\mc S)'$. Since $\mc M' \subset \mc S$, any such non-trivial $p$ witnesses that $\mc S$ is not strongly connected. This is a generalization of \cite[Theorem 3.4(4)]{connected-qgraphs}.
\end{remark}

\begin{remark}
    Characterizing strong connectedness via morphisms is another way to define \textit{quantum independence number} of a quantum confusability graph (i.e., where $\mc S$ is genuinely the confusability graph of some quantum channel). Recall the definition:
    \begin{definition} \cite[p. 6]{dsw}
        If $\mc S$ is a quantum confusability graph on $\mc M =  B(\mc H)$ for dim$(\mc H) < \infty$, its \textit{quantum independence number} $\alpha_q(\mc S)$ is the dimension of the largest subspace $C$ of $\mc H$ such that 
        \[
        P_C \mc S P_C = \C P_C
        \]
        where $P_C$ is the projection onto $C$.
    \end{definition}
     In the terminology of \cite{connected-qgraphs}, $\alpha_q(\mc S)$ is number of (strongly) connected components of $\mc S$. Let $T_n$ be the classical graph on $n$ vertices whose only edges are self-loops, and let Mor$(\mc S, T_n)$ be the collection of ${\bf W^*}$ morphisms $\theta$ implementing a morphism from $\mc S$ to $T_n$. In the terminology of morphisms,
        \[
        \alpha_q(\mc S) = \min\{n : \theta \in \text{Mor}(\mc S, T_n) \text{ injective}\}.
        \]
    One should think of $\theta$ as the dualization of a function that picks one ``vertex'' from each (strongly) connected component for every $k \in \{1,..,n\}$.
\end{remark}

\section{Limit of Quantum Graphs} \label{sec: limit of q graphs}

We have arrived at the promised construction from the introduction. This section will construct the projective limit of quantum graphs in detail. The inductive limit can also be constructed in an analogous manner; see Remark \ref{rmk: inductive limit}. The diagrams in this section suggest that the limits are taken along $\N$, but the construction will work for any directed set $J$.

Recall the maps that led to our definition of a morphism of quantum graphs:
\begin{center}
\begin{tikzcd} [scale cd=1.1]
\mathbf{Set}  & V_1  &  & V_2 \arrow[ll, "f"']               \\
\textbf{W*}  & \mathcal M_1 \arrow[rr, "\theta"]  &  & \mathcal M_2                    \\
& \mathcal M_1 \otimes_{eh} \mathcal M_1 \arrow[rr, "\theta \otimes \theta"] &  & \mathcal M_2 \otimes_{eh} \mathcal M_2 \\
\textbf{OpSp} & \Ann(\mathcal S_1) \arrow[rr, "\theta \otimes \theta"] \arrow[u, hook]    &  & \Ann(\mathcal S_2). \arrow[u, hook]  
\end{tikzcd}
\end{center}

Suppose we have a projective system $(G_j = (V_j, E_j), V_j \overset{f^{j,k}}{\leftarrow} V_k)$ along a directed set $J$ of classical graphs. A schematic for this system is

\begin{center}
\begin{tikzcd}[scale cd=1.2]
V_1 &  & V_2 \arrow[ll, "{f^{1,2}}", bend right]           &  & V_3 \arrow[ll, "{f^{2,3}}", bend right] \arrow[llll, "{f^{1,3}}"', bend right=49]                     & \cdots \arrow[l] \\
    &  &    &  &     &    \\
    &  &       &  &    &     \\
E_1 &  & E_2 \arrow[ll, "{(f^{1,2},f^{1,2})}", bend right] &  & E_3 \arrow[ll, "{(f^{2,3},f^{2,3})}", bend right] \arrow[llll, "{(f^{1,3},f^{1,3})}"', bend right=49] & \cdots \arrow[l]
\end{tikzcd}
\end{center}
In the classical case, the projective limits of the vertices and edges will provide the projective limit of graphs. Since we have dualized the above morphisms, the \textit{projective} limit of quantum graphs will be build on the \textit{inductive} limits of von Neumann algebras (the vertices) and the annihilators of the quantum graphs (the complementary set of edges). 

Now suppose for each $j \in J$ we have a von Neumann algebra $\mc M_j$ and quantum graphs $\mc S_j$ on $\mc M_j$. Further suppose for each $j,k \in J$ with $j \leq k$ we have ${\bf{W^*}}$ morphisms $\mc M_j \overset{\theta_{j,k}} {\to} \mc M_k$ inducing quantum graph morphisms $\mc S_j \leftarrow \mc S_k$. This yields a $\bf{W^*}$ inductive limit $(\mc M, \theta_{j,\infty})$:
\begin{center}
\begin{tikzcd}
\mathcal M_1 \arrow[r, "{\theta_{1,2}}", bend left=49] \arrow[rrrr, "{\theta_{1,\infty}}" description, bend right] \arrow[rr, "{\theta_{1,3}}", bend left=60, shift left=2] & \mathcal M_2 \arrow[r, "{\theta_{2,3}}", bend left=49] \arrow[rrr, "{\theta_{2,\infty}}" description, bend right] & \mathcal M_3 \arrow[rr, "{\theta_{3, \infty}}" description, bend right] & ... &  \varinjlim_{\bf{W^*}} \mc M_j=: \mc M.
\end{tikzcd}
\end{center}
To reduce notational clutter, we will make the following abbreviations in the remainder of this paper with $j \in J$ and $k \in J \cup\{\infty\}$: 
\begin{align*}
    \theta_{j,k}^{(2)} &:= \theta_{j,k} \otimes \theta_{j,k}\\
    \theta_{j, \infty}^{(2)} &:= \theta_{j,\infty} \otimes \theta_{j,\infty}\\
    \mc M_j^{(2)} &:= \mc M_j \otimes_{eh} \mc M_j \\
    \mc N &:= \bigcup_j \theta_{j,\infty}(\mc M_j) \\
    \mc M &:= \varinjlim_{\bf W^*} \mc M_j \\
    \mc M^{(2)} &:= \mc M \otimes_{eh} \mc M \\
    \mc A &:= \bigcup_j \theta_{j,\infty}^{(2)}(\Ann(\mc S_j)) \subseteq \mc M^{(2)}.
\end{align*}
Directly from Theorem \ref{thm: eh extension} we have well-defined extensions
\[
\mc M_j^{(2)} \overset{\theta_{j,k}^{(2)} }{\longrightarrow} \mc M_k^{(2)} 
\]
making the following diagram commute.
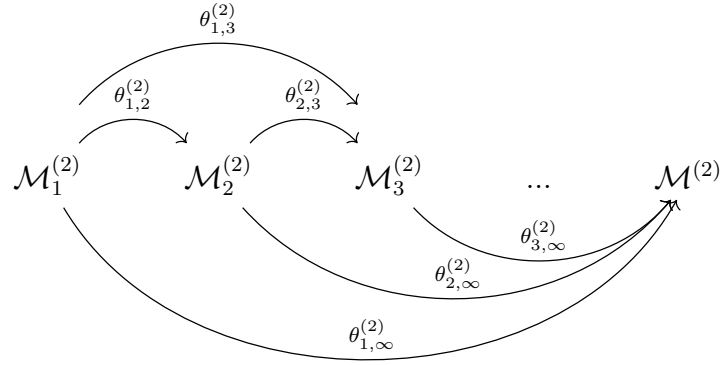
\begin{center}
\begin{tikzcd}
\mathcal M_1 ^{(2)} \arrow[r, "{\theta_{1,2}^{(2)}}", bend left=49] \arrow[rrrr, "{\theta_{1,\infty}^{(2)}}", bend right=60] \arrow[rr, "{\theta_{1,3} ^{(2)}}", bend left=49, shift left=5] & \mathcal M_2^{(2)} \arrow[r, "{\theta_{2,3}^{(2)}}", bend left=49] \arrow[rrr, "{\theta_{2,\infty}^{(2)}}", bend right=49] & \mathcal M_3^{(2)} \arrow[rr, "{\theta_{3,\infty}^{(2)}}", bend right=49] & ... &  \mc M^{(2)}
\end{tikzcd}
\end{center}
(Note that we are not saying that $\mc M^{(2)}$ is the inductive limit of $(\mc M_k^{(2)}, \theta_{j,k})$, only specifying the domains/codomains of the maps.) Since the $\theta_{j,k}$ induce quantum graph morphisms by hypothesis, the following diagram also commutes:
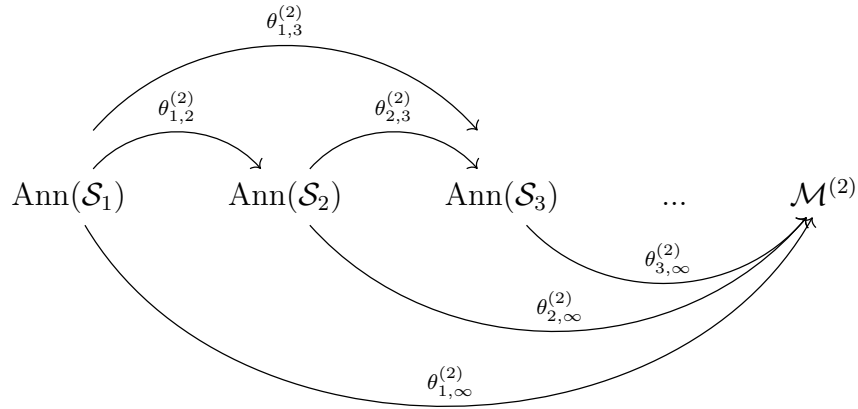
\begin{center}
\begin{tikzcd}
\Ann(\mc S_1) \arrow[r, "{\theta_{1,2}^{(2)}}", bend left=49] \arrow[rrrr, "{\theta_{1,\infty}^{(2)}}", bend right=60] \arrow[rr, "{\theta_{1,3} ^{(2)}}", bend left=49, shift left=5] & \Ann(\mc S_2) \arrow[r, "{\theta_{2,3}^{(2)}}", bend left=49] \arrow[rrr, "{\theta_{2,\infty}^{(2)}}", bend right=49] & \Ann(\mc S_3) \arrow[rr, "{\theta_{3,\infty}^{(2)}}", bend right=49] & ... & \mc M^{(2)}
\end{tikzcd}
\end{center}
It is at this point we must confront an uncomfortable truth. In order to take the projective limit of quantum graphs, we must restrict ourselves to the quantum graphs whose annihilators are $\sigma(\mc M_* \otimes_h \mc M_*)$-closed left ideals. This is a consequence of the following facts.

\begin{enumerate}
    \item The projective limit of the quantum graphs $(\mc M_j, \Ann(\mc S_j))$ must be a quantum graph $\mc S$ on $\mc M$ (that is, a $\sigma(\mc M_* \eh \mc M_*)$-closed left ideal in $\mc M \eh \mc M$) in order to satisfy the universal property of projective limits.
    \item The annihilator $\Ann(\mc S)$ must contain $\mc A$ in order to satisfy the universal property.
    \item We are only guaranteed that $\mc A$ absorbs left multiplication by $\mc N \otimes \mc N \subset \mc M \eh \mc M$. 
    \item If we take the $\Ann(\mc S) := \overline{\mc A}^{\sigma(\mc M_* \eh \mc M_*)}$, then $\mc A$ is not necessarily a left ideal in $\mc M \eh \mc M$. We do have that for every $n \in \mc N \otimes \mc N$ and $a \in \mc A$ that $na \in \mc A$. Using $w^*$ continuity of multiplication in dual Banach algebras, we even have that 
    \[
    n (\lim a_\lambda) \in \overline{\mc A}^{\sigma(\mc M_* \eh \mc M_*)}
    \]
    for $a_\lambda \in \mc A$ and the limit taken in the $\sigma(\mc M_* \eh \mc M_*)$ topology. However, since $\mc N$ is only a $\sigma(\mc M_*)$-dense subspace of $\mc M$, the algebraic tensor product $\mc N \otimes \mc N$ is not $\sigma(\mc M_* \eh \mc M_*)$-dense in $\mc M \eh \mc M$. We therefore cannot conclude that $\overline{\mc A}^{\sigma(\mc M_* \eh \mc M_*)}$ is a left ideal with respect to $\mc M \eh \mc M$ by taking the appropriate limits of elements of $\mc N \otimes \mc N$.
    \item However, the tensor product $\mc N \otimes \mc N$ \textit{is} $\sigma(\mc M_* \otimes_h \mc M_*)$-dense in $\mc M \eh \mc M$ (Lemma \ref{lem: density 1}). We can then use continuity of multiplication to conclude that 
    \[
    (\lim_\mu n_\mu )(\lim_\lambda a_\lambda )
    \]
    converges (where the limits are taken in the $\sigma(\mc M_* \otimes_h \mc M_*)$ topology). Since every element of $\mc M \eh \mc M$ can be written as $\lim_\mu n_\mu$, we can thusly show that $\overline{\mc A}^{\sigma(\mc M_* \h \mc M_*)}$ is a left ideal in $\mc M \eh \mc M$.
    \item Finally, every object is the projective limit along the directed set $(\{1\}, \leq)$ (the directed set of one element that is comparable to itself). Hence we must necessarily exclude the $\sigma(\mc M_* \eh \mc M_*)$-closed left ideals if we would like to take arbitrary projective limits of a class of quantum graphs. 
\end{enumerate}

See Subsection \ref{subsec: h-closed qgraphs} for examples and a non-example of such quantum graphs. This class of graphs does include all quantizations of classical graphs and quantum Cayley graphs of discrete quantum groups \cite{quantum-cayley}, for example. With this caveat out of the way, we proceed to prove that the projective limit of $\sigma(\mc M_* \h \mc M_*)$ quantum graphs exists.

\begin{lemma} \label{lem: density 1}
    Suppose $\mc M$ is a von Neumann algebra. If $\mc N \subseteq \mc M$ is a $\sigma(\mc M_*)$-dense unital $*$-algebra then $\mc N \otimes \mc N$ is a $\sigma(\mc M_* \otimes_{h} \mc M_*)$-dense subset of $\mc M \otimes_{eh} \mc M$.
\end{lemma}
\begin{proof}
    Recall that $\mc M_* \otimes \mc M_*$ is norm-dense in $\mc M_* \h \mc M_*$ (Definition \ref{defn: h}) and that $\mc M \otimes \mc M$ is $\sigma(\mc M_* \otimes_{h} \mc M_*)$-dense in $\mc M \eh \mc M$ (Corollary \ref{cor: alg-eh-density}). It thus suffices to show that we cannot separate $\mc N \otimes \mc N$ from $\mc M \otimes \mc M$ by functionals in $\mc M_* \otimes \mc M_*$, and this is precisely the property guaranteed by the $\sigma(\mc M_*)$-density of $\mc N$.
\end{proof}

\begin{lemma} \label{lem: alg-left-ideal}
    If $z \in \bigcup_j \theta^{(2)}(\mc M_j^{(2)})$ and $a \in \mc A$, then $za \in \mc A$.
\end{lemma}
\begin{proof}
    There must exist $j,k\in I$ such that 
    \[
    z = \theta_{j,\infty}^{(2)} (z_j) \quad a = \theta_{k,\infty}^{(2)}(a_k).
    \]
    Since $I$ is a directed set, there must further exist $\ell \in I$ such that
    \begin{align*}
    z &= \theta_{j,\infty}^{(2)}(z_j) = \theta_{\ell,\infty}^{(2)} \circ \theta_{j,\ell}^{(2)}(z_j) \\
    a &= \theta_{k,\infty}^{(2)}(a_k) = \theta_{\ell,\infty}^{(2)} \circ \theta_{k,\ell}^{(2)}(a_k).
    \end{align*}
    Since $\theta_{\ell,\infty}^{(2)}$ is a homomorphism
    \[
    za = \theta_{\ell,\infty}^{(2)}(\theta_{j,\ell}^{(2)}(z_j)\theta_{k,\ell}^{(2)}(a_k)).
    \]
    But $\theta_{k,\ell}^{(2)}(a_k) \in \Ann(\mc S_\ell)$ since each $\theta_{k,\ell}$ induces a morphism of quantum graphs. Hence 
    \[
    za = \theta_{j,\infty}^{(2)}(z_j)\theta_{k,\infty}^{(2)}(a_k) \in \theta_{\ell,\infty}^{(2)} (\Ann(\mc S_\ell)) \subset \mc A.
    \]
\end{proof}

\begin{lemma} \label{lemma: left ideal}
    The following closure of $\mc A$ 
    \[ 
    \overline{\bigcup_{k \in I} \theta_{k,\infty}^{(2)}(\Ann(\mc S_k))}^{\sigma(\mc M_*\otimes_{h} \mc M_*)} \subseteq \mc M^{(2)}
    \]
    is a $\sigma(\mc M_* \otimes _{h} \mc M_*)$-closed left ideal in $\mc M^{(2)}$. 
\end{lemma}
\begin{proof}
    From our assumption need only show this closure is a left ideal. We will use the following notation for arbitrary elements in the respective spaces:
    \begin{align*}
        a_\mu &\in \bigcup_k \theta_{k,\infty}^{(2)} (\Ann(\mc S_k)=: \mc A \\
        a &\in \overline{\bigcup_k \theta_{k,\infty}^{(2)} (\Ann(\mc S_k))}^{\sigma(\mc M_* \otimes_{h} \mc M_*)}=:\overline{\mc A}^{w^*} \\
        z_\lambda & \in N \otimes \mc N \\
        z & \in \mc M^{(2)}.
    \end{align*}
    From Lemma \ref{lem: density 1}, for any $z \in \mc M^{(2)}$ there exists $z_\lambda \to z$ in the $\sigma(\mc M_* \otimes_{h} \mc M_*)$ topology with $z_\lambda \in \mc N \otimes \mc N$. By the slice map property (\cite[Theorem  3.1]{blecher-smith-w*h}), we can assume that 
    \[
    z_\lambda \in \bigcup_j \theta^{(2)}(\mc M_j^{(2)}).
    \]
    Multiplication in any dual Banach algebra is separately continuous in the predual topology. Since $\mc M \eh \mc M$ is a dual Banach algebra in the $\sigma(\mc M_* \h \mc M_*)$ topology, we have by Lemma \ref{lem: alg-left-ideal} that $z_\lambda a \in \mc A$.  
    \[
    z a = (\lim_\lambda z_\lambda) a \in \mc A.
    \]
    Taking $a_\mu \to a$ is a $\sigma(\mc M_* \h \mc M_*)$-limit with $a_\mu \in \mc A$, we can again use the continuity of multiplication to show that $za = z (\lim_\mu a_\mu) \in \overline{\mc A}^{w^*}$. Thus $\overline{\mc A}^{w^*}$ is a left ideal in $\mc M^{(2)}$.
\end{proof}

From the discussion above, $\overline{\mc A}^{w^*}$ satisfies the defining universal property of a projective limit. We thus have our main theorem:

\begin{theorem} \label{thm: qraphlim}
    Suppose $J$ is a directed set, and for each $k \in J$ we have quantum graphs $(\mc M_j, \Ann(\mc S_j))$. For $j,k \in J$ such that $j \le k$, suppose we have normal, unital $*$-homomorphisms $\theta_{j,k}: \mc M_j \to \mc M_k$  implementing morphisms of quantum graphs $\mc S_j \leftarrow \mc S_k$. Let 
    \[
    \overline{\mc A}^{w^*} := \overline{\bigcup_{j \in J} \theta_{j,\infty}^{(2)}(\Ann(\mc S_j))}^{\sigma(\mc M_* \otimes_{eh} \mc M_*)} \subseteq \mc M^{(2)}.
    \]
    The projective limit of the quantum graphs $(\mc M, \overline{\mc A}^{w^*})$.
\end{theorem}
\begin{proof}
    The previous discussion shows we have candidates for an object and morphisms witnessing that $(\mc M, \bar {\mc A}^{w^*})$ is indeed the projective limit. What remains is to establish uniqueness, but this follows immediately from the uniqueness of $\mc M$ as an inductive limit in $\mathbf{W^*}$. 
\end{proof}

\subsection{Examples and a Non-Example of \texorpdfstring{$\sigma(\mc M_* \h \mc M_*)$}{}-closed Quantum Graphs} \label{subsec: h-closed qgraphs}

To cut down on word count, by ``$\tau$-closed quantum graph'' we mean ``a quantum graph $\mc S$ on $\mc M$ such that $\Ann(\mc S)$ is $\tau$-closed in $\mc M \eh \mc M$'' and $\tau$ is either $\sigma(\mc M_* \h \mc M_*)$ or $\sigma(\mc M_* \eh \mc M_*)$. Below we present an example and a non-example of $\sigma(\mc M_* \h \mc M_*)$-closed quantum graphs. A conjecture for a characterization of $\sigma(\mc M_* \h \mc M_*)$-closed quantum graphs is the content of Question \ref{question: left ideals}.

\begin{proposition} \label{prop: K(H) example}
    If $\mc S$ is a quantum graph on $\mc M \subset B(\mc H)$ such that 
    \[
    \overline{\mc S \cap K(\mc H) }^{\sigma(B(\mc H)_*)} = \mc S
    \]
    then $\mc S$ is a $\sigma(\mc M_* \h \mc M_*)$-closed quantum graph.
\end{proposition}
\begin{proof}
    Suppose $z_\lambda \to z$ is a net in $\Ann(\mc S)$ converging in the $\sigma(\mc M_* \h \mc M_*)$ topology. From \cite[Theorem 4.2(ii)]{blecher-smith-w*h}, we also have the convergence of $\Phi_{z_\lambda} \to \Phi_z$ in the $\sigma(K(\mc H) \hat \otimes B(\mc H)_*)$ topology. From our hypothesis, for every $T \in \mc S$ there exists a net $T_\mu \to T$ in the $\sigma(\mc B(\mc H)_*)$ topology such that $T_\mu \in K(\mc H) \cap \mc S$. In particular, $ \Phi_{z_\lambda}(T_\mu) = 0$ for all $\lambda, \mu$, so
    \[
    0 = \Phi_z(T) = \lim_\lambda \lim_\mu \Phi_{z_\lambda}(T_\mu) 
    \]
    so $z \in \Ann(\mc S)$.
\end{proof}

\begin{remark}
    Note that the proof above works if we replace $K(\mc H)$ with the finite rank operators in $B(\mc H)$.
\end{remark}

\begin{corollary} \label{cor: K(H)graphex}
    If $\mc M = \prod_{i \in I}^\infty M_{n_i} \subset B(\prod_{i \in I}^{2} \C^{n_i}) =: B(\mc H)$ is the $\ell^\infty$ product of an arbitrary number of matrix algebras, then every quantum graph is $\sigma(\mc M_* \h \mc M_*)$-closed.
\end{corollary}
\begin{proof}
     In this case the commutant $\mc M'$ is isomorphic to the collection of $\ell^\infty$ functions on the central projections $c_i$, and for every $i,j \in I$ we have that the block $c_i B(\mc H) c_j$ is finite dimensional. Since $\mc S$ is an $\mc M'$-bimodule, every $T \in \mc S$ is the $\sigma(B(\mc H)_*)$-limit of finite rank operators:
    \[
    T = \lim_{F \Subset I} \bigg(\sum_{i \in F} c_iTc_i\bigg)
    \]
    Using Proposition \ref{prop: K(H) example}, we reach the desired conclusion.
\end{proof}

\begin{corollary}
     Every quantization of a classical graph as constructed in Example \ref{ex: classical-to-quantum-graph} and the quantum Cayley graphs constructed in \cite{quantum-cayley} are $\sigma(\mc M_* \h \mc M_*)$-closed quantum graphs.
\end{corollary}

One does not have to go far to find a quantum graph that is not $\sigma(\mc M_* \h \mc M_*)$-closed.

\begin{example}[Non-Example] \label{ex: non-example graph}
    Assume $0 \notin \N$ for this non-example. If $\mc M = B(\ell^2(\N)) = B(\ell^2)$, the annihilator of the quantum graph $\mc S := \C I_{\ell^2}$ is not $\sigma(\mc M_* \h \mc M_*)$-closed. It suffices to find 
    \begin{itemize}
        \item a net $(T_\mu)_\mu \subset K(\ell^2)$ converging in the $\sigma(B(\ell^2)_*)$ topology to $I_{\ell^2}$
        \item a net $(z_\lambda)_\lambda \subset \mc M \eh \mc M$ with $I_{\ell^2} \in \text{ker}(\Phi_{z_\lambda})$ converging in the $\sigma(\mc M_* \h \mc M_*)$ topology
    \end{itemize}
    such that 
    \[
    0 = \lim_\lambda \Phi_{z_\lambda}(\lim_\mu T_\mu) =  \lim_\lambda \lim_\mu  \Phi_{z_\lambda}(T_\mu)
    \quad\text{and}\quad
    0 \ne \lim_\mu  \lim_\lambda \Phi_{z_\lambda}(T_\mu).
    \]
    For $F \Subset \N$ let $P_F \in B(\ell^2)$ be the projection onto $\ell^2(F)$. Our net $(T_\mu)_\mu$ will be the net $(P_F)_{F \Subset \N}$. The construction of the elements $z_\lambda$ is based on the observation that
    \[
    \Phi_{\begin{pmatrix}
        1 & -1 \\ 0 & 0 
    \end{pmatrix} \otimes \begin{pmatrix}
        1 & 0 \\ 1 & 0
    \end{pmatrix}} \bigg( \begin{pmatrix}
        1 & 0 \\ 0 & 1 
    \end{pmatrix}\bigg) = \begin{pmatrix}
        1 & -1 \\ 0 & 0 
    \end{pmatrix}
    \begin{pmatrix}
        1 & 0 \\ 0 & 1 
    \end{pmatrix}
    \begin{pmatrix}
        1 & 0 \\ 1 & 0
    \end{pmatrix} = \begin{pmatrix}
        0 & 0 \\ 0 & 0
    \end{pmatrix}.
    \]
    We will generalize this phenomenon by replacing $\begin{pmatrix}
        1 & 0 \\ 0 & 1
    \end{pmatrix}$ by $I_{\ell^2}$ and the matrix $\begin{pmatrix}
        1 & 0 \\ 1 & 0 
    \end{pmatrix}$ by the sum of $P_F$ and an off-diagonal operator denoted $Q_F$ below. Define 
    \[
    F + \max F = \{n + \max F : n \in F\}
    \]
    so that $F + \max F$ is disjoint from $F$. Now define $Q_F \in B(\ell^ 2)$ to be the following operator that shifts all functions supported on $F$ to those on $F + \max F$:
    \begin{align*}
        Q_F(f) = \sum_{n \in F} \<n| f|n+ \max F\>.
    \end{align*}
    Setting 
    \[
    z_F := (P_F - Q_F^\dagger) \otimes (P_F + Q_F)
    \]
    we have that $z_F \in \mc M \otimes \mc M \subset \mc M \eh \mc M$ and $\Phi_{z_F}(I_{\ell^2})=0$ by construction. To see that $(z_F)_{F \Subset \N}$ converges, let $K \in K(\ell^2)$ be a compact operator and let $\phi \in B(\ell^2)_*$. Explicitly writing out $\phi(\Phi_{z_F}(K))$, 
    \[
    \phi(\Phi_{z_F}(K)) = \phi(P_F K P_F + P_F K Q_F - Q^\dagger_F KP_F  - Q^\dagger_FKQ_F)
    \]
    Note that $\Phi_{z_F}(K) \overset{F \Subset \N}{\to} K$ in norm and so certainly $ \phi(\Phi_{z_F}(K)) \overset{F \Subset \N}{\to} \phi(K)$. Thus $z_F \overset{F \Subset \N}{\to} I_{\ell^2} \otimes I_{\ell^2} $ in the $\sigma(\mc M_* \h \mc M_*)$ topology\footnote{Alternatively, one can see convergence is assured by first noting that $\|z_\lambda\| = \|\Phi_{z_F}\|_{cb} = \|\Phi_{z_F}\|$ since $\Phi_{z_F}$ is a map between operator spaces where the codomain is equipped with its minimal OSS \cite[Chapter 3]{intro-op-sp-pisier} and then applying Banach-Alaoglu.}. Now, to finish the example: we have
    \begin{align*}
        \lim_\lambda \lim_\mu \Phi_{z_\lambda} (T_\mu) &= \lim_{G \Subset \N} \lim_{F \Subset \N} \Phi_{z_G}(P_F) \\
        &= \lim_{G \Subset \N} \Phi_{z_G}( \lim_{F \Subset \N} P_F) \\
        &= \lim_{G \Subset \N} \Phi_{z_G}(I_{\ell^2}) \\
        &= 0
    \end{align*}
    and also
    \begin{align*}
         \lim_\mu \lim_\lambda  \Phi_{z_\lambda} (T_\mu) &=\lim_{F \Subset \N} \lim_{G \Subset \N} \Phi_{z_G}(P_F) \\
         &= \lim_{F \Subset \N} \Phi_{I_{\ell^2} \otimes I_{\ell^2}} P_F \\
         &= I_{\ell^2}
    \end{align*}
    so the nets above witness that $I_{\ell^2}$ is a $\sigma(\mc M_* \eh \mc M_*)$-closed quantum graph that is not $\sigma(\mc M_* \h \mc M_*)$-closed.
\end{example}

\subsection{Implications of Quantum Graph Limits}

Here we make some observations on the projective limit above and collect some easy corollaries.

\begin{remark} \label{rmk: inductive limit}
    The work above is done for the projective limit of quantum graphs. The same procedure can be done for the \textit{inductive} limit of quantum graphs. Recall the categories of ${\mathbf{W^*}}$ and ${\mathbf{OpSp}}$ contain all small limits as well (Section \ref{sec: limits-prelim}). That is, 
    \[
    (\mc N, \varphi_k) := \varprojlim_{\bf W^*} \mc M_k \qquad \mc B := \varprojlim_{\bf OpSp} \mc \Ann_k
    \]
    both exist. The algebraic tensor product $\mc N \otimes \mc N$ is still $\sigma(\mc N_* \eh \mc N_*)$-dense in $\mc M \otimes_{eh} \mc M$. Following Lemma \ref{lem: alg-left-ideal}, we can show the set
    \[
    \bigg\{b \in \mc B : \exists k \in I, b_k \in \mc M_k^{(2)}  \text{ s.t. } \varphi_{k}^{(2)}(b_k) = b\bigg\}
    \]
    absorbs left multiplication by the following set
    \[
    \bigg\{z \in \mc N \otimes \mc N : \exists k \in I, z_k \in \mc M_k^{(2)}  \text{ s.t. } \varphi_{k}^{(2)}(z_k) = z\bigg\}.
    \]
The remainder of the argument follows similarly by using separate $w^*$ continuity of multiplication and approximation of elements.
\end{remark}

\begin{corollary}
    The inductive or projective limit of reflexive quantum graphs is not necessarily reflexive. 
\end{corollary}
\begin{proof}
    Example \ref{ex: non-example graph} demonstrates that although the multiplication map $m: \mc M \otimes_{\sigma h} \mc M$ (Corollary \ref{cor: sigma-mult}) can be defined on $\mc M \eh \mc M$, it is not $\sigma(\mc M_* \h \mc M_*)$-continuous. Hence we cannot conclude that the projective limit $\overline{\mc A}^{w^*}$ in Theorem \ref{thm: qraphlim} is the annihilator of a reflexive quantum graph even if the quantum graphs in the projective system are reflexive.
\end{proof}

\begin{corollary}
    The inductive or projective limit of symmetric quantum graphs is  symmetric. 
\end{corollary}
\begin{proof}
    This follows from the $\sigma(\mc M_* \h \mc M_*)$-continuity of the involution $(x \otimes y)^\ddagger$ (Lemma  \ref{lem: inv-cty}). 
\end{proof}

Since we framed connectedness in terms of morphisms (Definition \ref{defn: connected}), we also immediately have the following corollary.

\begin{corollary}
    If $(\mc M_k, \Ann_k, \varphi_{j,k})$ is an inductive system of quantum relations that are each not strongly connected, the inductive limit is also not strongly connected.
\end{corollary}
\begin{proof}
    If each $(\mc M_k, \Ann_k)$ is a strongly disconnected quantum relation, for each $k \in I$ there exists a $\mathbf{W^*}$ morphism $\C^2 \overset{\psi_k}{\to} \mc M_k$ inducing a non-trivial morphism from $\Ann(T_2)$ to $\Ann(\mc S_k)$. By the universal property of the {\bf{OpSp}} limit, there must be a non-trivial morphism from $\Ann(T_2)$ to $\overline{\mc B}^{w^*}$. 
\end{proof}

\section{Operator C*-Spaces} \label{sec: opc*sp}

Working with annihilators $\Ann(\mc S)$ instead of the bimodules $\mc S$ directly seems rather unintuitive. Moreover, much of the motivation for quantum graphs arises from the finite-dimensional case in which the dualization of the vertex set (i.e., $C(V)$ or $\ell^\infty(V)$) is both a $C^*$-algebra and a von Neumann algebra. This section explores the implications of alternative choices for these conventions; that is, we will generalize graphs as $C^*$-algebra bimodules and attempt to define a morphism strictly in terms of the bimodules. 

A belated acknowledgement: This entire paper was initially inspired by \cite{mawtod} in which the authors take some projective limits along $\N$ of finite classical graphs and define their graphs on $C^*$-algebras instead of von Neumann algebras. We sketch their construction below and use it to inspire our $C^*$-bimodule graphs.

\subsection{Operator C*-Systems}

The authors of \cite{mawtod} define the category of \textit{operator $C^*$-systems} which we will denote ${\mathbf {OC^*Sy}}$. The objects in this category consist of an operator system $\mc S$ and a unital $C^*$-algebra $\mc A$ such that $\mc S$ is an $\mc A$-bimodule. The authors use the representation-independent notion of operator system consisting of an order unit and cones, but for this section it will suffice to know that every operator system ``is'' a unital, $*$-closed subspace of some unital $C^*$-algebra where the unit in $\mc S$ is the same unit as the $C^*$-algebra. A morphism of this category is a pair $(u, \pi): (\mc S_1, \mc A_1) \to (\mc S_2, \mc A_2)$ such that $u: \mc S_1 \to \mc S_2$ is a UCP map and $\pi: \mc A_1 \to \mc A_2$ is a ${\bf C^*}$ morphism (i.e., a unital $*$-homomorphism) such that for all $a_1, a_2 \in \mc A$ and $s \in \mc S$
\begin{align} \label{eqn: opsys morphism}
u(a_1 \cdot s \cdot a_2) = \pi(a_1) \cdot u(s) \cdot \pi(a_2)
\end{align}
where we use $\cdot$ to indicate the bimodule action here. The authors then use the category above to define the inductive limit of graph operator systems (these are assumed to be undirected graphs with all self-loops). Let $\mc S_k \subseteq B(\mc H_k)$ be classical finite graph operator systems which will necessarily be bimodules over the diagonal algebras $\mc D_k$ in each $B(\mc H_k)$. Assume there are ${\bf{C^*}}$ morphisms $\pi_{k,k+1} : B(\mc H_k) \to B( \mc H_{k+1})$ such that we get the following inductive system
\begin{center}
\begin{tikzcd}
B(\mathcal H_1) \arrow[r, "\pi_{1,2}"] & B(\mathcal H_2) \arrow[r, "\pi_{2,3}"]  & B(\mathcal H_3) \arrow[r, "\pi_{3,4}"]  & ...            
\end{tikzcd}
\end{center}
and that $\pi_k(\mc S_k) \subseteq \mc S_{k+1}$. Being unital $*$-homomorphisms, the $\pi_k$ maps induce an inductive systems of the operator systems in the category ${\bf{OS}}$ of operator systems
\begin{center}
    \begin{tikzcd}
    \mc S_1 \arrow[r, "\pi_{1,2}"] & \mc S_2 \arrow[r, "\pi_{2,3}"]  & \mc S_3 \arrow[r, "\pi_{3,4}"]  & ...            
    \end{tikzcd}
\end{center}
and of the diagonal algebras in each $B(\mc H_k)$ 
\begin{center}
    \begin{tikzcd}
    \mc D_1 \arrow[r, "\pi_{1,2}"] & \mc D_2 \arrow[r, "\pi_{2,3}"]  & \mc D_3 \arrow[r, "\pi_{3,4}"]  & ...            
    \end{tikzcd}
\end{center}
where the last inductive system is taken in ${\bf{C^*}}$. The authors then go on to show that
\[
(\varinjlim_{\bf{C^*}}\mc D_k, \overline{\varinjlim_{\bf{OS}} \mc S_k}^{\|\cdot\|}, \varinjlim_{\bf{C^*}}B(\mc H_k))
\]
is the inductive limit of $(\mc D_k, \mc S_k, B(\mc H_k))_{k \in \N}$ in ${\mathbf {OC^*Sy}}$ and also define a corresponding inductive limit graph.

\begin{remark} \label{rmk: when is it ucp}
    Note that not every classical graph morphism $G_1 \leftarrow G_2$ induces a UCP map $\mc S_1 \to \mc S_2$ on the graph operator systems. Consider the two graphs 
    \begin{align*} 
    G_1 &= (V_1 = \{v\}, E_1 = \{(v,v)\}) \\
    G_2 &= (V_2 = \{w_1, w_2\}, E_2 = \{(w_1, w_1), (w_2,w_2), (w_1,w_2)\}).
    \end{align*}
    The map
    \begin{align*}
    V_1 \overset{f}{\leftarrow} V_2 
    \end{align*}
    such that $v = f(w_1)= f(w_2)$ induces a graph morphism $G_1 \leftarrow G_2$, but the corresponding induced map $\mc S_1 \overset{\phi}{\leftarrow} \mc S_2$ is not trace-preserving. For instance, we have
    \[
    |v\>\< v| = \phi(|w_1\>\<w_2|).
    \]
    The (pre)adjoint $\phi^\dagger: \mc S_1 \to \mc S_2$ therefore cannot be unital. However, this is essentially the only obstruction. A classical graph morphism $G_1 \overset{f}{\leftarrow} G_2$ will induce a UCP map $\mc S_1 \to \mc S_2$ if and only if 
    \[
    (u,v) \in E_2 \text{ and } u \neq v\implies f(u) \neq f(v).
    \]
\end{remark}

Although classical graph morphisms (between finite graphs) do not generally induce morphisms of operator \textit{systems}, they do always induce morphisms of operator \textit{spaces}. To reframe the \cite{mawtod} work for arbitrary classical graph morphisms, we introduce \textit{operator $C^*$-spaces} in analogy with operator $C^*$-systems. 

\subsection{Operator C*-Spaces} 

Let us begin with the finite classical case where we have a map $V_1 \overset{f}{\leftarrow} V_2$ inducing a morphism of graphs $G_1 = (V_1, E_1)$ and $G_2 = (V_2, E_2)$. The map $f$ induces a unital $*$-homomorphism $C(V_1) \overset{\pi}{\to} C(V_2)$ via the formula
\begin{align*}
    \pi: C(V_1) &\to C(V_2) \\
    g &\mapsto g \circ f
\end{align*}
in the same manner as the quantum graph case (Subsection \ref{subsec: morphisms of qgraphs}). Also as before, if we equip $C(V_k)$ with the standard inner product we have the maps
    \begin{center}
    \begin{tikzcd}
              & V_1            &  & V_2 \arrow[ll, "f"']         \\
            \mathbf {C^*} & C(V_1) \arrow[rr, "\pi"] &  & C(V_2)           \\
            \mathbf {C^*}^{op}              & C(V_1)         &  & C(V_2) \arrow[ll, "\pi^\dagger"']   .
    \end{tikzcd}
    \end{center}
    Equip $C(V_k \times V_k) \cong C(V_k) \otimes C(V_k) $ with the standard inner product to obtain the map $e = \pi^\dagger \otimes \pi^\dagger$:
    \begin{center}
        \begin{tikzcd}
            \mathbf{OpSp} & C(V_1) \otimes C(V_1)  &  & C(V_2) \otimes C(V_2). \arrow[ll, "e"']  
        \end{tikzcd}
    \end{center}
    Define the operator spaces
    \[
    \mc S_k := \Span\{|v\>\<w| : (v,w) \in E_k\} \subset C(V_k) \otimes C(V_k).
    \]
    If $f$ induces a graph morphism, then in particular $e$ must map $\mc S_2$ into $\mc S_1$:
    \begin{center}
        \begin{tikzcd}
            \mathbf{OpSp} & C(V_1) \otimes C(V_1) \supseteq \mathcal S_1   &  & \mathcal S_2 \subseteq C(V_2) \otimes C(V_2). \arrow[ll, "e"']  
        \end{tikzcd}
    \end{center}

(Compare this construction to the operator $C^*$-systems above: the pair $(e^\dagger, \theta)$ implements an operator $C^*$-system homomorphism if and only if $e$ is a UCP map. We saw in Remark \ref{rmk: when is it ucp} that $e$ is not generally UCP when $f$ induces a graph morphism.)
\begin{proposition}
        Let $G_1 = (V_1, E_1)$ and $G_2 = (V_2, E_2)$ be finite classical graphs. Let $V_1 \overset{f}{\leftarrow} V_2$, $C(V_1) \overset{\pi}{\to} C(V_2)$, and $ C(V_1) \otimes C(V_1) \overset{e}{\leftarrow}  C(V_2) \otimes C(V_2)$ be defined as above. Then 
        \begin{align} \label{eqn: c*graph hom}
        a_1 \cdot e(s) \cdot a_2 = e(\pi(a_1) \cdot s \cdot \pi(a_2))
        \end{align}
        and $f$ induces a graph morphism if and only if $\mc S_1 \supseteq e(\mc S_2)$.
\end{proposition}
\begin{proof}
    We first check that Equation \ref{eqn: c*graph hom} is does indeed hold. Since we are working in finite dimensions, it suffices to show equality for $a_1 = \chi_t$, $a_2 = \chi_u$ and $s = |v\>\<w|$. In this case, Equation \ref{eqn: c*graph hom} becomes
    \[
    \chi_t \cdot |f(v)\>\<f(w)| \cdot \chi_u = e\bigg( \bigg( \sum_{t = f(v')} \chi_{v'}  \bigg) |v\>\<w| \bigg( \sum_{u = f(w')} \chi_{w'}\bigg)\bigg).
    \]
    The LHS is 0 if and only if $f(v) \ne t$ or $f(w) \ne u$. Therefore either 
    \[
    \bigg(\sum_{t = f(v')} \chi_{v'}  \bigg) |v\> = 0 \qquad \text{or} \qquad \<w| \bigg( \sum_{u = f(w')} \chi_{w'}\bigg) = 0
    \]
    so the equation holds in that case. On the other hand, the LHS is non-zero if and only if 
    \[
     \chi_t \cdot |f(v)\>\<f(w)| \cdot \chi_u =  |f(v)\>\<f(w)| \iff f(v) = t \text{ and } f(w) = u
    \]
    in which case the RHS becomes 
    \[
    e\bigg( \bigg( \sum_{t = f(v')} \chi_{v'}  \bigg) |v\>\<w| \bigg( \sum_{u = f(w')} \chi_{w'}\bigg)\bigg) = e(|v\>\<w|) = |f(v)\>\<f(w)|
    \]
    just as claimed. Finally, we see that $f$ induces a graph morphism if and only if for every $|v\>\<w| \in \mc S_2$ we have $|f(v)\>\<f(w)| \in \mc S_1$, which is clearly exactly the condition that $\mc S_1 \supseteq e(\mc S_2)$.
\end{proof}

We now generalize to arbitrary dimensions. In such cases, we cannot always impose an inner product structure on the $C^*$-algebra $C(V)$, and so the adjoint maps $\pi^\dagger$ and $e = \pi^\dagger \otimes \pi^\dagger$ are not defined. There is also the issue that the algebraic tensor product $C(V) \otimes C(V)$ is not necessarily an operator space. We will therefore simply ask for a $C^*$-algebra morphism $\pi$ and an operator space morphism $e$ that satisfies Equation \ref{eqn: c*graph hom} in order to define a morphism of these bimodules. (For a classical, finite graph this would correspond to asking for both a map $V_1 \overset f \leftarrow V_2$ and a map $E_1 \overset e \leftarrow E_2$ such that
\[
(f(v), f(w)) = e((v,w)).
\]
There may be other ways around the ill-definedness of $^\dagger$ of which I am unaware.)

\begin{definition} \label{defn: c*graph}
    A \textit{C$^*$-graph} is a unital $C^*$-algebra $\mc A$ and an operator space $\mc S$ such that $\mc S$ is an $\mc A$-bimodule. If $(\mc A_1, \mc S_1)$ and $(\mc A_2, \mc S_2)$ are C$^*$-graphs, a pair of maps
    \[
    \mc S_1 \overset{e}{\longleftarrow} \mc S_2 \qquad \mc A_1 \overset{\pi}{\longrightarrow} \mc A_2
    \]
    is a \textit{C$^*$-graph morphism} from $(\mc A_2, \mc S_2)$ to $(\mc A_1, \mc S_1)$ if $e$ is completely contractive and $\pi$ is a $*$-homomorphism such that
    \begin{align} \label{eqn: c*graphmorph}
    a_1 \cdot e(s) \cdot a_2 = e(\pi(a_1) \cdot s \cdot \pi(a_2))
    \end{align}
    for all $a_1,a_2 \in \mc A_1$ and $s \in \mc S_2$.
\end{definition}

\begin{remark}
    Note that Equation \ref{eqn: c*graphmorph} is \textit{not} the usual notion of morphism of $C^*$-correspondences (or of operator bimodules). Generally, one would expect a morphism $X_{\mc A} \to Y_{\mc B}$ of right modules to be given by two maps $e:X_{\mc A} \to  Y_{\mc B}$ and $\pi: A \to B$ such that 
    \[
    e(x \cdot a) = e(x) \cdot \pi(a).
    \]
    In particular, $e$ and $\pi$ are \textit{covariant}. However, we see that our maps $e$ and $\pi$ in Definition \ref{defn: c*graph} above must be contravariant to recover the classical case. 
\end{remark}
    
Naturally, we would like to see if this category has limits. We will show that projective limits exist. See \cite[Chapter 2.6]{tensor-products-pisier} for a quick construction of the (amalgamated) free product of $C^*$-algebras; we hope this convinces the reader that one can take inductive limits in the category ${\bf C^*}$.
    
\begin{theorem} \label{thm: c*graphlimit}
    If $((\mc A_j, \mc S_j),(\pi_{j,k},e^{j,k}))_{j \in J}$ is a projective system of $C^*$-Graphs then
    \[
   \bigg((\mc A := \varinjlim_{\mathbf{C^*}} \mc A_j, \pi_{j,\infty}), (\mc S := \varprojlim_{\mathbf{OpSp}} \mc S_j, e^{j, \infty})\bigg)
    \]
    is its projective limit. 
\end{theorem}
\begin{proof}
    The uniqueness of the object $(\mc A,\mc S)$ in the projective limit follows from the uniqueness of the limits in $\mathbf{OpSp}$ and $\mathbf{C^*}$. We will verify that $\mc S$ is an $\mc A$-bimodule in a canonical way and each $(e^{j,\infty}, \pi_{j,\infty})$ induces a $C^*$-graph morphism from $(\mc A,\mc S)$ to $(\mc A_j,\mc S_j)$. We thus need only show that $\mc S$ is an $\mc A$-bimodule. There are many ways to write down the inductive/projective limits (again, see \cite{Daws-Limits} for inductive limits of Banach algebras). For $C^*$-algebras, the inductive limit can be realized as
    \[
    \mc A:= \varinjlim_{\mathbf{C^*}} \mc A_j  = \overline{\bigcup_j \pi_{j,\infty}(\mc A_j)}^{\|\cdot\|}.
    \]
    where $\|\cdot \|$ is the norm on $\mc A$. The projective limit in $\mathbf{OpSp}$ can be written as
    \[
    \mc S :=  \varprojlim_{\mathbf{OpSp}} \mc S_j = \bigg\{(s_j)_{j \in J} \in \prod_{j \in J}^\infty \mc S_j: s_j \overset{e^{j,k}}{\mapsfrom} s_k\bigg\}.
    \]
    We will only show the left action $\bigcup_j \pi_{j,\infty}(\mc A_j) \cdot \mc S$ is well-defined, as showing the right action is precisely analogous. Fix
    \[
    \pi_{j_1,\infty}(a_{j_1}) \in \bigcup_j \pi_{j,\infty}(\mc A_j) \quad (s_j)_{j \in J} \in \mc S.
    \]
    We will define $\pi_{j_1,\infty}(a_{j_1}) \cdot s_{j_2}$ for an arbitrary $ s_{j_2} \in (s_j)_{j \in J}$. The action $\mc A \cdot \mc S$ will extend by the norm-density of the union in $\mc A$. Since $J$ is directed, there exists $j_3 \geq j_1, j_2$. By definition, $s_{j_2} = e^{j_2, j_3}(s_{j_3})$. We claim the action
    \[
    \pi_{j_1,\infty}(a_{j_1}) \cdot s_{j_2} :=  e^{j_2, j_3}(\pi_{j_1, j_3}(a_{j_1}) \cdot s_{j_3})
    \]
    is well-defined. Let us see a diagram. 
    
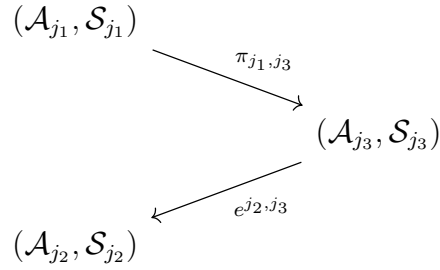
\begin{center}
        \begin{tikzcd}
        {(\mathcal A_{j_1}, \mathcal S_{j_1})} \arrow[rrd, "{\pi_{j_1, j_3}}"] &  &  \\
        &  & {(\mathcal A_{j_3}, \mathcal S_{j_3})} \arrow[lld, "{e^{j_2, j_3}}"] \\
        {(\mathcal A_{j_2}, \mathcal S_{j_2})}   &  &  
        \end{tikzcd}
    \end{center}    
    We now show that this action is independent of the choice of $j_3$. Suppose $j_4 \geq j_2, j_1$ and $j_5 \geq j_4, j_3$. We then have the following diagram
    \begin{figure}[H]
    \centering
        \begin{tikzcd} 
        {(\mathcal A_{j_1}, \mathcal S_{j_1})} \arrow[rr, "{\pi_{j_1, j_3}}" description] \arrow[rrdddd, "{\pi_{j_1, j_4}}" description, bend left] \arrow[rrrrdd, "{\pi_{j_1, j_5}}" description, bend left=49] &  & {(\mathcal A_{j_3}, \mathcal S_{j_3})} \arrow[lldddd, "{e_{j_2, j_3}}" description, bend left] \arrow[rrdd, "{\pi_{j_3, j_5}}" description, bend left] &  &   \\
        &  &    &  &      \\
        &  &  &  & {(\mathcal A_{j_5}, \mathcal S_{j_5})} \arrow[lldd, "{e^{j_4, j_5}}" description, bend left] \arrow[lllldd, "{e^{j_2, j_5}}" description, bend left=49] \arrow[lluu, "{e^{j_3, j_5}}" description, bend left] \\
        &  &   &  &   \\
        {(\mathcal A_{j_2}, \mathcal S_{j_2})}   &  & {(\mathcal A_{j_4}, \mathcal S_{j_4})} \arrow[ll, "{e^{j_2, j_4}}" description] \arrow[rruu, "{\pi_{j_4, j_5}}", bend left] &  &      
        \end{tikzcd}
        \caption{Implicitly, an argument for working with annihilators.}
    \end{figure}
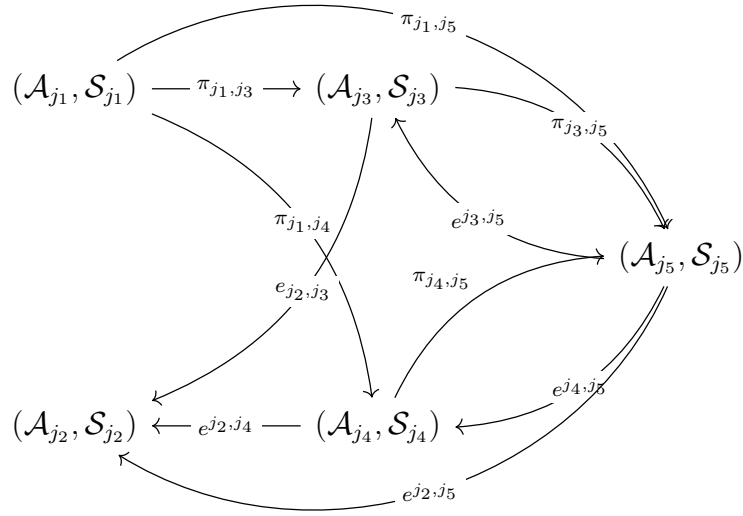
It suffices to show that
    \begin{align*}
        e^{j_2,j_5}(\pi_{j_1, j_5}(a_{j_1}) \cdot s_{j_5}) &= e^{j_2,j_3}(\pi_{j_1, j_3}(a_{j_1}) \cdot s_{j_3}) \\
        e^{j_2,j_5}(\pi_{j_1, j_5}(a_{j_1}) \cdot s_{j_5}) &= e^{j_2,j_4}(\pi_{j_1, j_4}(a_{j_1}) \cdot s_{j_4}) 
    \end{align*}
    and since the proofs are precisely analogous we will only show the first equation. Since $e^{j,k}$ and $\pi_{j,k}$ are maps from projective/inductive systems, we have 
    \begin{align} \label{eqn: c*graphlim}
    e^{j_2,j_5}(\pi_{j_1, j_5}(a_{j_1}) \cdot s_{j_5}) = e^{j_2,j_3} \circ e^{j_3, j_5} (\pi_{j_3, j_5} \circ \pi_{j_1, j_3} (a_{j_1}) \cdot s_{j_5}).
    \end{align}
    From the definition of $C^*$-graph morphisms, we generally have
    \[
    e^{j,k} (\pi_{j,k}(a_j) \cdot s_k) = a_j \cdot e^{j,k}(s_k) = a_j \cdot s_j
    \]
    so applying this to Equation \ref{eqn: c*graphlim}
    \[
    e^{j_2,j_3} \circ e^{j_3, j_5} (\pi_{j_3, j_5} \circ \pi_{j_1, j_3} (a_{j_1}) \cdot s_{j_5}) = e^{j_2,j_3}(\pi_{j_1, j_3}(a_{j_1}) \cdot s_{j_3})
    \]
    we obtain precisely what we need. Recall, however, that $\pi_{j_1,\infty}(a_{j_1})$ does not have a unique form so we must show the action is independent of this choice as well. Suppose
    \[
    \pi_{j_1, \infty}(a_{j_1}) = \pi_{k_1,\infty}(a_{k_1}).
    \]
    This proof follows the same rhythm as above: we choose indices further along than $j_1$ and $k_1$ and use the inductive/projective morphisms to cancel the appropriate maps. We include it for completeness. Find $\ell_2 \geq \ell_1 \geq j_1, k_1, j_2$. We always have 
    \[
    \pi_{\ell_1,\infty} \circ \pi_{j_1,\ell_1} (a_{j_1}) = \pi_{\ell_1,\infty} \circ \pi_{k_1,\ell_1} (a_{k_1})
    \]
    and similarly if we replace $\ell_1$ by $\ell_2$. Therefore there exists $a_{\ell_1}$ such that $\pi_{\ell_1, \infty}(a_{\ell_1}) = \pi_{j_1, \infty}(a_{j_1})$ and similarly for $\ell_2$. We claim that 
    \begin{align*}
        \pi_{j_1, \infty}(a_{j_1}) \cdot s_{j_2} &= \pi_{\ell_1, \infty}(a_{\ell_1}) \cdot s_{j_2}     \\
        \pi_{k_1, \infty}(a_{k_1}) \cdot s_{j_2} &= \pi_{\ell_1, \infty}(a_{\ell_1}) \cdot s_{j_2}  
    \end{align*}
    and as before we only prove the first equation. Since $\ell_1 \leq \ell_2$,
    \begin{align*}
    \pi_{j_1, \infty}(a_{j_1}) \cdot s_{j_2} &= e^{j_2, \ell_2} (\pi_{j_1, \ell_2}(a_{j_1}) \cdot s_{\ell_2}) \\
    &= e^{j_2, \ell_1} \circ e^{\ell_1, \ell_2} (\pi_{\ell_1, \ell_2} \circ \pi_{\ell_1, j_1} (a_{j_1}) \cdot s_{\ell_2}) \\
    &= e^{j_2, \ell_1}(\pi_{\ell_1, j_1} (a_{j_1}) \cdot s_{\ell_2})\\
    &= \pi_{\ell_1, \infty}(a_{\ell_1}) \cdot s_{j_2} .
    \end{align*}
\end{proof}

\section{Concluding Remarks} \label{sec: remarks}
We dub quantum graphs ``$W^*$-graphs'' in this subsection for aesthetic symmetry.

\subsection{\texorpdfstring{$C^*$}{}-Graphs}
    
\begin{remark}
    Since $W^*$-graphs quantize the vertices as $\ell^\infty(V)$, there are no extra topological considerations on $V$. If we are to quantize the vertices as $C(V)$, however, Gelfand duality requires $V$ be compact (or at least locally compact) and $f$ be continuous. The dual category to $C^*$-graphs on commutative $C^*$-algebras is therefore no longer quite the category of graphs.
\end{remark}
\begin{remark}
    An argument in favor of $C^*$-algebraic route is that the inductive limit of C$^*$-algebras generally preserves more of the structure of the algebras. For instance, the enveloping von Neumann algebras of any two uniformly hyperfinite (UHF) algebras $\varinjlim_{\mathbf{C^*}} M_{n_k}$ and $\varinjlim_{\mathbf{C^*}} M_{n_j}$ are isomorphic \cite[Corollary 5.2]{pedersenuhf}. For UHF algebras, $\varinjlim_{\mathbf{W^*}} M_{n_k}$ and $\varinjlim_{\mathbf{W^*}} M_{n_j}$ are precisely the enveloping von Neumann algebras of the respective $C^*$-algebraic limits. (One can see this by 1) noting the adjoint to the inclusion function $F: \mathbf{W^*} \to \mathbf{C^*}$ is the functor that sends each $C^*$-algebra to its double dual and 2) recalling left adjoints preserve colimits.) The relationship between $\varinjlim_{\mathbf{C^*}} \mc M_j$ and $\varinjlim_{\mathbf{W^*}} \mc M_j$ for arbitrary von Neumann algebras $\mc M_j$ is generally not as neat, however. We again refer to reader to \cite{guichardet}.
\end{remark}
\begin{remark}
    Could one define a morphism of $C^*$-graphs using annihilators as in Section \ref{sec: morphisms}? No, at least not with the $\eh$ tensor product. The key property of $\eh$ is that for any unital subalgebras $\mc A, \mc B \subset B(\mc H)$ (with no topological assumptions)
        \[
        \mc A' \eh \mc B' = CB_{\mc A,\mc B}(K(\mc H) , B(\mc H)) = CB^\sigma_{\mc A,\mc B}(B(\mc H))
        \]
        \cite[Theorem 4.2(ii)]{blecher-smith-w*h}. Thus $\mc M \eh \mc M$ naturally identifies the $\mc M'$ bimodules, but this does not necessarily hold for bimodules over other algebras. It may be possible to characterize $C^*$-graphs with annihilators in other tensor products.
\end{remark}
\begin{remark}
    Could one define a morphism of $W^*$-graphs by replacing the $\mathbf{C^*}$ morphisms with $\mathbf{W^*}$ morphism and using a similar map $e$ to induce the mapping of edges/operator spaces? Possibly. However, map $e$ would need to be relaxed to a $w^*$-continuous morphism so we must descend to a subcategory of ${\bf {OpSp}}$ where the objects are dual operator spaces. 
\end{remark}

\subsection{\texorpdfstring{$W^*$}{}-Graphs}

\begin{remark} \label{rmk: limit of q graphs}
    As noted above, one can take the projective (or indeed, inductive limit) of $W^*$-graphs along an index category that is a directed set. The construction above does require that the index be directed so that in Lemma \ref{lem: alg-left-ideal} the space $\mc A$ naturally absorbs left multiplication by a $w^*$-dense subset of the ambient von Neumann algebra. 
    
    Given the fuss about products and coproducts in the preliminaries (Section \ref{sec: limits-prelim}), in retrospect it may have been more elegant to find the (co)product and (co)equalizer in the category of $W^*$-graphs to prove such all small (co)limits exist. See Question \ref{question: coproduct}.
\end{remark}

\begin{remark} \label{rmk: w*/ucp}
    In quantum information theory, quantum channels can be characterized as UCP maps between von Neumann algebras (or dually, CPTP maps). More precisely, the quantum channel is a UCP map on the observables that generate a von Neumann algebra. One would naturally try to take a (co)limit in the category of von Neumann algebras with normal UCP maps as morphisms. However, I do not know whether the category of von Neumann algebras with UCP maps has (co)limits.
    
    In any case, this category of von Neumann algebras and normal UCP maps may be the wrong interpretation in the setting of quantum channels. Let $\Phi: \mc M_1 \to \mc M_2$ be a quantum channel (i.e., a normal UCP map) with a Kraus form $\Phi(T) = \sum_{i \in I} K_i^\dagger T K_i$ and define 
    \[
    \mc S := \overline{\Span}^{\sigma(B(\mc H)_*)}\{K_i^\dagger K_j : i,j \in I\}
    \]
    to be its quantum confusability graph. (I do not know if confusability graphs have been defined for quantum channels beyond the finite dimensional case, but this seems like a reasonable definition.) Similarly, let $\mc T$ be the quantum confusability graph for a quantum channel $\Psi: \mc N_1 \to \mc N_2$. A morphism from $\mc S$ to $\mc T$ is implemented by a $\mathbf{W^*}$ morphism $\theta: \mc N_1 \to \mc M_1$. The algebraic relations of $\mc N_1$ and $\mc M_1$ are a fundamental property of the observables in the respective von Neumann algebras -- namely, observables commute if and only if they can be simultaneously measured. A UCP map would not preserve these relations.
\end{remark}

\section{Further Questions} \label{sec: questions}

Here we collect sketch some questions inspired by the work above. Fuller expositions may be found in the author's thesis. Questions \ref{question: transitivity}, \ref{question: left ideals}, and \ref{question: coproduct} are some loose ends from the work above. Question \ref{question: t-morphisms} has connections with quantum information theory. Question \ref{question: qcayleygraph} proposes that quantum Cayley graphs of quantum groups should be considered as left ideals in some $\eh$ tensor product. Question \ref{question: categorification} proposes that the $\eh$ tensor product characterization of quantum graphs could extend the categorification of finite quantum graphs.

\subsection{Transitivity of Quantum Graphs} \label{question: transitivity}

A specter that has haunted many investigations of quantum graphs returns to the question of transitivity. A quantum graph $\mc S$ on a von Neumann algebra $\mc M \subset B(\mc H)$ is \textit{transitive} if $\mc S^2 \subset \mc S$ \cite[Definition 2.4]{weaverqrelations}. In \cite[Remark 3.4(iii)]{gp-obs}, the author asks how one might obtain $\Ann(\mc S^2)$ from $\Ann(\mc S) \subset \mc M \eh \mc M$. It also plays a role in abstractly defining a \textit{quantum distance operator}. Ideally, such an operator should have a categorical characterization as adjacency operators do in Appendix \ref{appendix: string}.

\subsection{\texorpdfstring{$\sigma(\mc M_* \h \mc M_*)$}{}-Closed Left Ideals} \label{question: left ideals}

Given the significance of $\sigma(\mc M_* \h \mc M_*)$-closed graphs in Section \ref{sec: limit of q graphs}, it would nice to have another characterization of such graphs. As noted in Corollary \ref{cor: K(H)graphex}, all quantum graphs on a von Neumann algebra of the form $\mc M = \prod_{i \in I}^\infty M_{n_i}$ are examples.

\subsection{Coproduct of Quantum Graphs} \label{question: coproduct}

The categorical product of two finite classical graphs $G_1$ and $G_2$ is the so-called \textit{tensor product} of graphs, so-called because the adjacency operator of $G_1 \prod G_2$ is the tensor product of the adjacency operators of $G_1$ and $G_2$. The vertex set of this categorical product is $V_1 \otimes V_2$. A naive quantization of the vertex set might be $\ell^\infty(V_1) \otimes \ell^\infty(V_2)$.

In the category of quantum graphs, however, the ``vertex set'' of the categorical product of $G_1$ and $G_2$ is the free product of von Neumann algebras (see Remark \ref{rmk: free product vna}). Without going into precise detail about how to form the free product, it is without a doubt much bigger than $\ell^\infty(V_1) \otimes \ell^\infty(V_2)$. However, this is not so much an obstruction as an interesting feature of non-commutative algebras. The coproduct in the category of \textit{commutative} associative algebras is the tensor product, which aligns with the classical case. Allowing our algebras to be non-commutative necessitates free product-like constructions in the coproduct, a byproduct of which is that a product of quantum graphs must be a much larger object than in the classical case.

\subsection{\textit{t}-Morphisms of Quantum Graphs} \label{question: t-morphisms}

The morphisms explored in Subsection \ref{subsec: equiv-morphs} are considered \textit{classical} morphisms of quantum graphs because no external quantum resources are involved in the morphisms. Allowing the use of quantum resources naturally leads to the graph homomorphism game (see, for example, \cite{qtc-graph-hom}). Non-local game theory distinguishes among the types of quantum resources available to the players in the non-local game. \cite[Definition 4.1]{jun-alg-connect} generalizes $t$-homomorphisms from \cite{qtc-graph-hom} for $t \in \{loc, q, qa, C^*, alg\}$ via ${\bf W^*}$ morphisms. Namely, a $t$-homomorphism from $(\mc S_1, \mc M_1)$ to $(\mc S_2, \mc M_2)$ for finite dimensional von Neumann algebras $\mc M_k$ is a unital $*$-algebra $\mc A_t$ and a unital $*$-homomorphism $\theta: \mc M_2 \to \mc M_1 \otimes \mc A_t$ satisfying
\[
\theta^\dagger (\mc S_1 \otimes 1_{\mc A_t}) \theta \subset \mc S_2 \otimes \mc A_t
\]
and $\mc A_t$ is an algebra whose type depends on $t$. (The notation $\theta^\dagger$ does \textit{not} indicate the adjoint of a map between Hilbert spaces; we refer the reader to \cite{jun-alg-connect} for the definition.) As of this writing there does not seem to be a translation for $t$-homomorphisms in terms of annihilators, but following in the vein of Subsection \ref{subsec: equiv-morphs} such a translation would yield a compatible notion of $t$-homomorphism for infinite quantum graphs. 

\subsection{Quantum Cayley Graphs of (Profinite) Quantum Groups} \label{question: qcayleygraph}

A classical theorem states that a Cayley graph of a profinite group is always a profinite graph \cite[Example 2.1.12]{profinitebook}. The question of whether this statement still holds if ``quantum'' is inserted before every noun is still open. As of this writing, there does not seem to be a definition for the Cayley graph of a general quantum group, though we have one for discrete quantum groups from \cite{quantum-cayley}. In \cite{gp-obs}, the author proposes a definition for the invariant quantum relations (invariant quantum graphs, according to the conventions of this paper) of a quantum group. The author proves that his proposed definition works in the case in which the quantum group is a classical locally compact group and claims that the proof works the same when replacing the relevant objects by their analogues in \cite{lcqg}. 

We have a notion of profinite quantum graphs from Section \ref{sec: limit of q graphs}. If one allows Cayley graphs to be degenerate (i.e., suppose that the generators of the Cayley graph do not generate the group), then $G$-invariant relations in $G \times G$ are precisely the Cayley graphs of $G$. We thus have candidates for the quantum Cayley graphs of a profinite quantum groups. It remains to see if the classical theorem holds in the quantum case.

\subsection{Categorification of Quantum Graphs} \label{question: categorification}

The paper \cite{categorified-graphs} is very roughly speaking the categorification of quantum graphs on finite dimensional von Neumann algebras. One could possibly extend this categorification to quantum graphs on arbitrary von Neumann algebras through the relativized extended Haagerup tensor product as defined in \cite[Definition 1.3]{magajna-strong}. (Thank to Srivatsav Kunnawalkam Elayavalli for pointing out this resource.)

\begin{center}
    
\end{center}

\appendix

\addcontentsline{toc}{section}{Appendix}

\section{A Taste of Categorification} \label{appendix: string}

The purpose of this appendix is to introduce just enough terminology to show that classical homomorphisms of finite quantum graphs introduced in \cite[Section 5.2]{musto-reutter-verdon} are equivalent to a morphism of finite quantum graphs in Definition \ref{defn: qgraph morph}. 

The authors of \cite{cq-metric} quantize (finite) sets, functions, and (finite) graphs through categorification. Roughly, one characterizes the classical objects or properties via morphisms in the category \textbf{Set} and interpret them in another category (\textbf{FHilb}, the category of finite dimensional Hilbert spaces and linear maps, for instance) to obtain their quantum analogs.

Definitions will be accompanied by motivating examples. The authors in \cite{musto-reutter-verdon} develop quantum graphs by quantizing the adjacency matrix of a graph via string diagrams. This is part of a broader trend of categorification: by rephrasing familiar objects and phenomena in categorical terms, we achieve a new intuition of the categorified object. In this appendix we will categorify undirected graphs on finitely many vertices. We will not introduce the graphical methods here. For short overviews on graphical methods, see \cite{musto-reutter-verdon} and \cite[Chapter 4]{adina}. For an excellent textbook, see \cite{heunen-vicary}.

Our conventions: in this appendix, all Hilbert spaces are finite dimensional. If $\varphi: \mc H_1 \to \mc H_2$ is a map between Hilbert spaces, we denote its adjoint by $\varphi^\dagger: \mc H_2 \to \mc H_1$. 

\begin{definition}
    An \textit{algebra} is a Hilbert space $\mc H$ with a multiplication map $m: \mc H \otimes \mc H \to \mc H$ and unit map $u: \C \to \mc H$ satisfying the following equations:
    \begin{align*}
    (m \otimes I_{\mc H}) \circ (m \otimes I_{\mc H}) &= (I_{\mc H} \otimes m)\circ (I_{\mc H} \otimes m) &&(\text{associativity)} \\
    m \circ (I_{\mc H} \otimes u) &= I_{\mc H} = m \circ (u \otimes I_{\mc H}) &&(\text{unitality}).
    \end{align*}
    A \textit{coalgebra} is a Hilbert space $\mc H$ with a comultiplication map $\delta: \mc H \to \mc H \otimes \mc H$ and a counit map $\varepsilon: \mc H \to \C$ satisfying the following equations:
    \begin{align*}
    (\delta \otimes I_{\mc H}) \circ \delta &= (I_{\mc H} \otimes \delta) \circ \delta  &&(\text{coassociativity}) \\
    (I_{\mc H} \otimes \varepsilon) \circ \delta &= I_{\mc H} = (\varepsilon \otimes I_{\mc H}) \circ \delta &&(\text{counitality)}
    \end{align*}
\end{definition}
As the names suggest, if $(\mc H, m, u)$ is an algebra then $(\mc H, m^\dagger, u^\dagger)$ is a coalgebra. 

\begin{example} \label{ex: algebra}
    Let $\mc H$ be a Hilbert space with orthonormal basis $\{|i\>\}_{i \in [n]}$. The following maps make $\mc H$ an algebra:
    \begin{align*}
        m(|i\> \otimes |j\>) &= \<i|j\> |i\> & u(1_{\C}) = \sum_{i \in [n]} |i\>.
    \end{align*}
    Their adjoints are given by
    \begin{align*}
        m^\dagger(|i\>) &= |i\> \otimes |i\> & u^\dagger(|i\>) = 1_{\C}.
    \end{align*}
\end{example}

The objects of interest are both algebras and coalgebras. Namely, the structures $m$ and $u$ are compatible with $m^\dagger$ and $u^\dagger$ in the following way.

\begin{definition} \label{defn: daggerfrob}
    A \textit{dagger Frobenius algebra} is an algebra $(\mc H, m,u)$ such that 
    \begin{align*}
    (m \otimes I_{\mc H}) \circ (I_{\mc H} \otimes m^\dagger) = m^\dagger \circ m = (I_{\mc H} \otimes m) \circ (m^\dagger \otimes I_{\mc H}).
    \end{align*}
    Define the swap map $\sigma: \mc H \otimes \mc H \to \mc H \otimes \mc H$ by $\sigma(|i\> \otimes |j\>) = |j\> \otimes |i\>$. Such an algebra is \textit{special} if
    \[
    m^\dagger \circ m = I_{\mc H}.
    \]
    \textit{symmetric} if
    \[
    u^\dagger \circ m \circ \sigma = u^\dagger \circ m
    \]
    and \textit{commutative} if
    \[
    m \circ \sigma = m.
    \]
\end{definition}

Example \ref{ex: algebra} is also special symmetric dagger Frobenius algebra (abbreviated SSFA). In fact, there is a bijection between finite sets and SCFAs \cite[Corollary 7.2]{scfa-sets} and a bijection between finite dimensional $C^*$-algebras and SSFAs \cite[Theorem 4.6]{vicary}. The authors of \cite{musto-reutter-verdon} thus motivate their definition of SSFAs as \textit{quantum sets}. These will play the role of vertices in quantum graphs. 

In order to describe homomorphisms between quantum graphs, we also need the quantum analog of functions between vertices. For finite sets $V_1, V_2$ there is a bijection between
\[
\{\text{functions } f: V_1 \to V_2\} \leftrightarrow \{*\text{-homomorphisms } \theta: \C(V_2) \to \C(V_1)\}
\]
and 
\[
\{*\text{-homomorphisms } \theta: \C(V_2) \to \C(V_1)\} \leftrightarrow \{*\text{-cohomomorphisms } \theta: \C(V_1) \to \C(V_2)\}.
\]
Unsurprisingly, we can express $*$-(co)homomorphisms using dagger Frobenius algebra structures.
\begin{definition} \label{defn: *-cohom}
    Let $A = (\mc H_A, m_A, u_A)$ and $B = (\mc H_B, m_B, u_B)$ be SSFAs. A $*$-homomorphism $\theta: A \to B$ is a linear map such that
    \begin{align*}
    \theta \circ m_A &= m_B \circ (\theta \otimes \theta) \\
    \theta \circ u_A &= u_A \\
    \theta^\dagger &= (u_B^\dagger \circ m_B \otimes I_A) \circ (I_B \otimes \theta \otimes I_A) \circ (I_B \otimes m_A^\dagger \circ u_A).
    \end{align*}
    Dually, a $*$-cohomomorphism $f: A \to B$ is a linear map such that
    \begin{align*}
        m_B^\dagger \circ f &= (f \otimes f) \circ m_A^\dagger \\
        u_B^\dagger \circ f &= u_B^\dagger \\
        f^\dagger &= (I_A \otimes u_B^\dagger \circ m_B) \circ (I_A \otimes f \otimes I_B) \circ (m_A^\dagger \circ u_A \otimes I_B).
    \end{align*}
\end{definition}
These definitions for $*$-homomorphism and $*$-cohomomorphism coincide with the popular ones for finite dimensional $C^*$-algebras \cite[Theorem 4.7]{vicary}.

\clearpage

\printbibliography[
heading=bibintoc,
title={References}
]

\end{document}